\documentclass[onecolumn,final,a4page]{siamltex}
\usepackage{amsmath}
\usepackage{amssymb}
\usepackage{bm}
\usepackage{stmaryrd}
\usepackage{enumerate}
\usepackage{graphicx,color,overpic,rotating}
\usepackage[normalem]{ulem}
\numberwithin{equation}{section}
\usepackage{pgfplots}
\usepackage[labelformat=simple]{subcaption}

\usepackage{tikz}
\usetikzlibrary{calc}

\newcommand{\sref}[1]{Section~\ref{#1}}
\newcommand{\dref}[1]{Definition~\ref{#1}}
\newcommand{\assref}[1]{Assumption~\ref{#1}}

\renewcommand{\Re}{{\rm{I\!R}}}

\renewcommand{\k}{p} 

\newcommand{\D}{\Omega}
\newcommand{\dD}{\partial\D}
\newcommand{\E}{E} 
\newcommand{\dE}{\partial\E} 
\newcommand{\e}{e} 

\newcommand{\s}{s} 
\newcommand{\edge}{s}
\newcommand{\patch}[1]{\widetilde{#1}}
\newcommand{\vertex}{{\bf z}}

\newcommand{\Vh}{V_h} 
\newcommand{\VhE}{V_h^\E} 
\newcommand{\dimVhE}{N_{\E}}
\newcommand{\NE}{\dimVhE} 
\newcommand{\Vhs}{V_h^\s}

\renewcommand{\P}[2]{\mathcal{P}_{#2}(#1)} 
\newcommand{\PE}[1]{\P{\E}{#1}} 

\newcommand{\bigspace}{\ensuremath{\mathcal{W}_h}}
\newcommand{\biglocalspace}{\ensuremath{\bigspace^{\E}}} 
\newcommand{\biglocalspaceS}{\ensuremath{\bigspace^{\s}}} 

\newcommand{\NO}[1]{\mathcal{N}^{#1}}
\newcommand{\MO}[2]{\mathcal{M}^{#1}_{#2}}
\newcommand{\mindex}{\boldsymbol{\alpha}}

\newcommand{\intE}{\int_\E} 
\renewcommand{\d}{\operatorname{d}}
\newcommand{\dx}{\d\vx}
\newcommand{\ds}{\d\!s}
\newcommand{\ints}{\int_{\s}} 
\newcommand{\edgeint}{\ints}

\newcommand{\vx}{\bm{x}}

\renewcommand{\u}{u}
\renewcommand{\v}{v}
\newcommand{\w}{w} 
\newcommand{\interp}{\operatorname{I}} 
\newcommand{\wI}{{\w_{\interp}}} 
\newcommand{\wIs}{{\w_{\interp}^\s}} 
\newcommand{\vI}{{\v_{\interp}}}
\newcommand{\vIs}{{\v_{\interp}^\s}}
\newcommand{\vIdE}{{\v_{\interp}^{\dE}}}

\newcommand{\vClem}{\v_{c}}
\newcommand{\uh}{\u_h}

\newcommand{\vh}{\v_h}
\newcommand{\wh}{\w_h} 

\newcommand{\m}{m} 
\newcommand{\ma}{\m_{\alpha}} 

\renewcommand{\(}{\left(} 
\renewcommand{\)}{\right)} 
\newcommand{\abs}[1]{{\lvert#1\rvert}} 
\newcommand{\norm}[1]{{\lVert#1\rVert}} 
\newcommand{\triplenorm}[1]{| \kern -.4mm \|{#1}\| \kern -.4mm |}
\newcommand{\jump}[1]{{\left\llbracket#1\right\rrbracket}}

\newcommand{\Id}{\operatorname{I}} 
\newcommand{\dof}{\operatorname{dof}}

\newcommand{\Po}[1]{\Pi^0_{#1}} 
\newcommand{\Pos}[1]{\Pi^{0,\s}_{#1}} 
\newcommand{\spaceProj}[1]{\ensuremath{\Pi^{*}_{#1}}} 

\newcommand{\apprx}[1]{{{#1}_h}}
\newcommand{\diff}{\boldsymbol{\kappa}} 
\newcommand{\diffh}{\apprx{\boldsymbol{\kappa}}} 
\newcommand{\diffNB}{\kappa} 
\newcommand{\conv}{\boldsymbol{\beta}} 
\newcommand{\convh}{\apprx{\boldsymbol{\beta}}} 
\newcommand{\reac}{{\gamma}} 
\newcommand{\reach}{\apprx{\gamma}} 
\newcommand{\reacSym}{\mu} 
\newcommand{\force}{f}
\newcommand{\forceh}{\apprx{f}}

\newcommand{\A}{A} 
\newcommand{\Ah}{\A_h} 
\renewcommand{\AE}{\A^\E} 
\newcommand{\AhE}{\A_h^\E} 

\renewcommand{\a}{a} 
\newcommand{\aE}{\a^\E} 
\newcommand{\ahE}{\a_h^\E} 
\renewcommand{\b}{b} 
\newcommand{\bE}{\b^\E} 
\newcommand{\bhE}{\b_h^\E} 

\newcommand{\Sta}{S} 
\newcommand{\StaE}{\Sta^{\E}} 

\newcommand{\altStaE}{\mathcal{I}^{\E}}

\newcommand{\Th}{\mathcal{T}_h} 
\newcommand{\Edges}{\mathcal{S}_h} 
\newcommand{\InternalEdges}{\mathcal{S}_h^{\text{int}}} 
\newcommand{\BoundaryEdges}{\mathcal{S}_h^{\text{bdry}}} 
\newcommand{\nuE}{\nu_{\E}}
\newcommand{\subtriang}{\widehat{\Th}} 

\newcommand{\meshReg}{\rho} 

\newcommand{\elipLower}{\diffNB_*} 
\newcommand{\elipLowerE}{\diffNB_*^{\E}} 
\newcommand{\elipUpper}{\diffNB^*} 
\newcommand{\elipUpperE}{\diffNB^*_{\E}} 
\newcommand{\reacLower}{\mu_0} 

\newcommand{\reacLowerE}{\mu_0^{\E}} 

\newcommand{\spacedim}{d} 
\newcommand{\n}{\bm{n}} 

\newcommand{\reg}{m}

\newcommand{\LTWO}{L^2}

\newcommand{\err}{e}
\newcommand{\elresid}{R_{\E}}
\newcommand{\elresidd}{R_{\E'}}
\newcommand{\edgeresid}{J_{\edge}}

\newcommand{\dataresid}{\theta}
\newcommand{\edgedataresid}{\dataresid_{\edge}}
\newcommand{\eldataresid}{\dataresid_{\E}}
\newcommand{\eldataresidd}{\dataresid_{\E'}}
\newcommand{\edgepatch}{\omega_{\edge}}
\newcommand{\elpatch}{\omega_{\E}}
\newcommand{\estimator}{\eta}
\newcommand{\dataest}{\Theta}
\newcommand{\virtualosc}{\Psi}
\newcommand{\nonPolyResid}{B}
\newcommand{\stabest}{\mathfrak{S}}

\newcommand{\Cbub}{C_{\operatorname{bub}}}
\newcommand{\Cproj}{C_{\operatorname{proj}}}
\newcommand{\Cclem}{C_{\operatorname{Clem}}}
\newcommand{\CclemT}{\hat{C}_{\operatorname{Clem}}}

\newcommand{\Ctrace}{C_{\operatorname{tr}}}
\newcommand{\stabConst}{s}
\newcommand{\Cstab}{C_{\operatorname{stab}}}

\newcommand{\CequivUp}{\widehat{C}_{\operatorname{equiv}}}
\newcommand{\Cequiv}{C_{\operatorname{equiv}}}
\newcommand{\Cinv}{C_{\operatorname{inv}}}
\newcommand{\Cpf}{C_{\operatorname{PF}}}
\newcommand{\Cp}{C_{\operatorname{P}}}
\newcommand{\bubEl}{\psi_{\E}}
\newcommand{\bubEdge}{\psi_{\edge}}

\newcommand{\pecletE}{\operatorname{Pe}_{\E}}

\usepackage{ntheorem}

\renewtheorem{theorem}{Theorem}
\renewtheorem{proposition}[theorem]{Proposition}
\renewtheorem{lemma}[theorem]{Lemma}
\newtheorem{remark}[theorem]{Remark}
\renewtheorem{definition}[theorem]{Definition}
\newtheorem{assump}[theorem]{Assumption}
\renewtheorem{corollary}[theorem]{Corollary}

\definecolor{gold}{rgb}{1.0, 0.84, 0.0}

\newcommand{\errClem}{\err_{\interp}}

\renewcommand{\Re}{\mathbb{R}}

\begin{document}

\title{A Posteriori Error Estimates for the Virtual Element Method}
\author{Andrea Cangiani \and Emmanuil H.~Georgoulis  \and Tristan Pryer \and Oliver J. Sutton}

\date{}

\maketitle

\begin{abstract}
  An posteriori error analysis for the virtual element method (VEM) applied to general elliptic problems is presented. The resulting error estimator is of residual-type and applies on very general polygonal/polyhedral meshes.
The estimator is fully computable as it relies only on quantities available from the VEM solution, namely its degrees of freedom and element-wise polynomial projection.
Upper and lower bounds of the error estimator with respect to the VEM approximation error are proven.
The error estimator is used to drive adaptive mesh refinement in a number of test problems.
Mesh adaptation is particularly simple to implement since elements with consecutive co-planar edges/faces are allowed and, therefore, locally adapted meshes do not require any local mesh post-processing.
\end{abstract}

\section{Introduction}

The virtual element method (VEM) is a numerical framework
introduced in~\cite{BasicsPaper} for the approximation of solutions
of partial differential equations (PDEs). Key attributes of
VEM are its ability to permit the
use of meshes with very general polygonal/polyhedral elements~\cite{EquivalentProjectors,ArbitraryRegularityVEM,Hitchhikers,MixedVEM,NonconformingVEM,General,UnifiedVEM} and the seamless incorporation of approximation spaces with arbitrary global regularity~\cite{ArbitraryRegularityVEM}.
There has been a strong interest in recent years in the development of numerical methods on general polygonal/polyhedral meshes~\cite{mimetic,PFEM,BasicsPaper,HHO,dg1,dg2,Giani-Houston:2014,Chen-Wang-Ye:2014,Cockburn-DiPietro-Ern:2015,Droniou:Eymard-Herbin:2015,Weisser:2015}, not least due to the potential appeal of such mesh generality in the context of Lagrangian and/or adaptive refinement/coarsening algorithms.
The virtual element method revolves around a \emph{virtual element space} of trial and test functions, defined implicitly
through local PDE problems on each element.
The local spaces are designed to contain a space of (physical frame) polynomials, ultimately responsible
for the accuracy of the method, as well as a complementary space of more general
non-polynomial functions.
In this respect, VEM belongs to the wide family of Generalised FEM~\cite{GFEM}, as do other approaches to general meshes such as the Polygonal FEM~\cite{PFEM},
numerical multiscale methods (see, e.g.~\cite{MFEM,HMM,LOD} and the references therein), and, so-called, Trefftz-type methods
in general~\cite{babuska_osborn,XFEM,hiptmair_moiola_perugia,Rjasanow-Weisser:2014}.  Quite differently from all the above
approaches, though, the virtual element method does \emph{not} require
the evaluation of non-polynomial functions, even
in a rough fashion.
Instead, to produce \emph{fully computable} and
accurate VEM formulations on general meshes, the method's degrees of
freedom are carefully chosen so that relevant projections of the
local virtual element functions into the local subspace of
polynomials are computable.
A crucial consequence of this approach is that the VEM computed solution is not available
in the form of a (virtual element) function. Rather, the
solution is represented via the values of its degrees of freedom, from which we can access, for instance, the piecewise polynomial projection of the corresponding complete virtual element function.

Given the virtual nature of the method, the design and analysis of
fully computable a posteriori error bounds for VEM is a
challenging task.  In~\cite{VEMApostM3NA},
a posteriori error bounds for the $C^1$-conforming VEM for the two-dimensional Poisson problem are proven. The
$C^1$-continuity of the VEM space was employed to circumvent the fact that the inter-element normal fluxes of the virtual basis functions are not computable in the more standard $C^0$-conforming method.
Furthermore, the analysis of~\cite{VEMApostM3NA} relies on a Cl\'ement-type interpolant construction requiring quadratic (or higher-order) virtual element spaces.
To the best of our knowledge,~\cite{VEMApostM3NA} is the only a
posteriori error analysis for VEM currently available in the
literature.

In this work, we present a new residual-type a posteriori error
analysis for the $C^0$-conforming VEM introduced in~\cite{UnifiedVEM} for the discretization of
second order linear elliptic reaction-convection-diffusion
problems with non-constant coefficients in two and three dimensions. We circumvent the
fact that the VEM solution normal fluxes are not computable by replacing them by a
suitable projection of the fluxes instead, resulting in the introduction of \emph{virtual inconsistency} terms in the a posteriori error estimator to account for this replacement error.
Moreover, the analysis is based on a new Cl\'ement-type VEM interpolant in two and three dimensions, which, crucially, allows for minimal regularity interpolation by \emph{linear} VEM functions. This new interpolant, which may be of independent interest, is constructed starting from the standard finite element Cl\'ement interpolant on a regular sub-triangulation; cf. also~\cite{SteklovVEM} for a related idea for a two-dimensional VEM interpolant. 
In two spatial dimensions, the resulting constants in the Cl\'ement interpolation estimate
are dependent on the respective FEM Cl\'ement interpolant on a regular sub-triangulation, which are in principle available~\cite{Verfurth:1999,veeser_verfurth:2012}, along with other computable quantities. In the three-dimensional case, when general polygonal element interfaces are present in the mesh, a second, not easily computable in general, constant appears; see Remark \ref{rem:constants} for a detailed discussion.
Once equipped with the above developments, a posteriori bounds are derived by careful treatment of the inconsistency terms, whereby appropriate projection operators are introduced into the discrete problem formulation; we refer to~\cite{Tristan:2016} for a related general framework for a posteriori analysis of inconsistent discontinuous Galerkin methods.
Lemma~\ref{lemma_lower_bound_incon} gives a lower bound for these inconsistency terms, indicating that they are of correct order up to data oscillation.
Although the focus of this work is the VEM introduced in~\cite{UnifiedVEM}, the proof of the a posteriori error bounds is quite general and can be adapted in a straightforward manner to other VEM approaches, such
as the VEM proposed in~\cite{General}, cf. the discussion in Section~\ref{sec:conclusion} below.

Adaptive mesh refinement driven by a posteriori error estimators is a
well established tool for the efficient numerical solution of PDEs exhibiting local, numerically challenging,
solution features.
In this context, the extreme mesh flexibility allowed by the VEM
approach offers a number of potential advantages. For instance,
locally adapted meshes do not require any local post-processing: very
general polygonal/polyhedral meshes are admissible due to the physical
frame polynomial subspaces included in the VEM space, therefore
removing any restrictions posed by maximum angle conditions or
mesh-distortion, as is the case for standard adaptive FEM. Moreover,
VEM avoids the need to introduce additional degrees of freedom for
hanging node/face removal (`green refinement') during mesh refinement: hanging nodes
introduced by the refinement of a neighbouring element are simply
treated as new nodes since adjacent co-planar elemental interfaces are
perfectly acceptable. Furthermore, in the VEM context, coarsening
becomes trivial and inexpensive to implement as node removal does not
necessitate any further local mesh modification. The latter is
particularly attractive in the context of numerical solution of
evolution PDEs where mesh-coarsening is standard practice to track evolving fronts and singularities efficiently. Indeed, apart from making
mesh change straightforward to implement, the mesh flexibility offered
by VEM may have the potential to provide complexity reduction with
respect to standard FEM on traditional simplicial or box-type
meshes. At the time of writing, no results in this direction are available.

The remainder of this work is structured as follows. In
Section~\ref{sec:contProb}, we describe the model problem and in \sref{sec:vemFramework} we introduce the virtual element method.
Some fundamental approximation results are presented in Section~\ref{subsec:approximation-properties}, which are used to prove upper and lower bounds for an a posteriori error estimator in \sref{sec:aposteriori}.
This estimator is then used in \sref{sec:numerics} within an automatic mesh adaptivity algorithm in a series of numerical examples, confirming numerically its optimality. Finally, in \sref{sec:conclusion} we give some concluding remarks.

Below, we shall use standard notation for the relevant function
spaces. For a Lipschitz domain $\omega \subset {\mathbb R}^d$,
$d=2,3$, we denote by $H^s(\omega)$ the Hilbert space of index $s\ge
0$ of real--valued functions defined on $\omega$, endowed with the
seminorm $|\cdot |_{s,\omega}$ and norm $\|\cdot\|_{s,\omega}$;
further $(\cdot,\cdot)_\omega$ stands for the standard $L^2$
inner-product. Finally, $|\omega|$ denotes the $d$--dimensional
Hausdorff measure of $\omega$.

\section{The Continuous Problem}
\label{sec:contProb}
Let $\D \subset \Re^\spacedim$ be a polygonal domain for $\spacedim=2$
or a polyhedral domain for $\spacedim=3$ and consider the linear
second order elliptic boundary value problem
\begin{equation}
\label{eq:pde}
\begin{split}
  -\nabla\cdot(\diff(\vx)\nabla \u)
  +
  \conv(\vx) \cdot \nabla \u
  +
  \reac(\vx) \u
  &=
  \force(\vx)\phantom{0} \quad\text{in}~\D,
  \\
  \u &= 0\phantom{f(\vx)}\quad\text{on}~\dD.
\end{split}
\end{equation}
We assume that $ \reac, \force\in L^{\infty}(\D)$, $\conv\in
[W^{1,\infty}(\D)]^d$ and that $\diff\in [L^{\infty}(\D)]^{d\times d}$
is a strongly elliptic symmetric diffusion tensor, i.e. there exist
$\elipLower,\elipUpper > 0$, independent of $\vec{v}$ and $\vx$, such
that
\begin{equation}
  \elipLower \abs{\vec{v}(\vx)}^2
  \leq
  \vec{v}(\vx) \cdot \diff(\vx) \vec{v}(\vx)
  \leq
  \elipUpper \abs{\vec{v}(\vx)}^2,
  \label{eq:diffEllipticity}
\end{equation}
for almost every $\vx \in \D$ and for any $\vec{v} \in
[H^1_0(\D)]^{\spacedim}$, with $\abs{\cdot}$ denoting the standard
Euclidean norm on $\Re^\spacedim$.  Finally, we assume that, for
almost every $\vx \in \D$, there exists a constant $\reacLower$ such
that
\begin{equation}
	\reacSym(\vx) := \reac(\vx) - \frac{1}{2} \nabla \cdot \conv(\vx) \geq \reacLower > 0.
	\label{eq:reacBound}
\end{equation}
Problem~\eqref{eq:pde} can be written in variational form:
Find $\u \in H^1_0(\D)$ such that
\begin{equation}
  (\diff \nabla \u, \nabla \v)
  +
  (\conv \cdot \nabla \u, \v)
  +
  (\reac \u, \v)
  =
  (\force, \v)
  \quad \quad
  \forall \v \in H^1_0(\D),
  \label{eq:origVariationalForm}
\end{equation}
with $(\cdot,\cdot)$ denoting the (standard) $\LTWO$ inner-product
over $\D$.  Following~\cite{UnifiedVEM}, we split the bilinear form on
the left-hand side of~\eqref{eq:origVariationalForm} into its
symmetric and skew-symmetric parts
\begin{subequations}
  \begin{align}
    \a(\u,\v)
    &:=
    (\diff \nabla \u, \nabla \v)+\(\reacSym \u, \v\), \label{eq:a:def}
    \\
    \b(\u,\v)
    &:=
    \frac{1}{2} \left[ \(\conv \cdot \nabla \u, \v\)
      -
      \(\u, \conv \cdot \nabla \v\) \right],\label{eq:b:def}
  \end{align}
\end{subequations}
and we consider the problem written in the equivalent form: find $\u
\in H^1_0(\D)$ such that
\begin{equation}
  A(u,v)
  :=\,
  \a(\u,\v) + \b(\u,\v)
  =
  (\force, \v) \quad \quad
  \forall \v \in H^1_0(\D).
  \label{eq:newVariationalForm}
\end{equation}
Rewriting the bilinear form in this fashion
is a useful step in view of preserving the coercivity of $A$ at the virtual (discrete) level,  independently of the mesh size.
An alternative VEM based on the original variational form~\eqref{eq:origVariationalForm} and \emph{without} assuming coercivity is presented in~\cite{General}, whose well-posedness relies on selecting sufficiently small mesh size. The a posteriori error analysis presented below can be also applied with to the method of~\cite{General}, with minor modifications, cf. Section~\ref{sec:conclusion}.

\section{The Virtual Element Method}
\label{sec:vemFramework}

\subsection{Polygonal and polyhedral partitions}
Let $\{\Th\}$ be a family of partitions of the domain $\D$ into
non-overlapping \emph{simple polygonal/polyhedral elements} with
maximum size $h$; a polygon/polyhedron is termed simple when its
boundary is not self-intersecting.  We further assume that the
boundary $\partial E$ of each element $\E\in\Th$ is made of a
uniformly bounded number of \emph{interfaces}: line segments if $d=2$ and planar polygons with a uniformly bounded number of straight edges if $d=3$.
Elemental interfaces are either part of the boundary of $\Omega$ or shared
with another element in the decomposition.
By $\s$ we shall denote the generic $(\spacedim-1)$-dimensional mesh
interface (either an edge when $\spacedim = 2$, or a face when
$\spacedim = 3$) of a mesh element $\E\in\Th$; the set of all mesh
interfaces in $\Th$ will be denoted by $\Edges$, which is subdivided
into the set of boundary interfaces $\BoundaryEdges := \{ \s \in
\Edges : \s \subset \dD \}$ and the set of internal interfaces
$\InternalEdges := \Edges \setminus \BoundaryEdges$. Also, $\nuE$ will
be the (uniformly bounded) number of interfaces $\s\in\dE$.

We note that, in particular, partitions including non-convex elements
are allowed, as also are elements with consecutive co-planar
edges/faces, such as those typical of locally refined meshes with
hanging nodes. We also make the following mesh regularity assumptions
which are standard in this context,
cf.~\cite{EquivalentProjectors,UnifiedVEM}.
\begin{assump}{(Mesh regularity).}
  \label{ass:meshReg}
  We assume the existence of a constant $\meshReg > 0$
  such that
  \begin{enumerate}
  \item Every element $\E$ of $\Th$ is star-shaped with respect
    to a ball of radius $\meshReg h_{\E}$;
  \item For every element $\E$ of $\Th$ and every interface $\s$
    of $\E$, $h_{\s} \geq \meshReg h_{\E}$;
  \item For $\spacedim=3$, every interface $\s \in \Edges$
    viewed as a $2$-dimensional element satisfies assumptions 1 and 2 above.
  \end{enumerate}
\end{assump}

\begin{remark}[Global shape regularity]
  \label{rem:shapeRegSubTri}
As in the a priori setting~\cite{BasicsPaper,UnifiedVEM}, 
the a posteriori error analysis presented below extends in a straightforward fashion to the case of polygonal/polyhedral elements which result from simply connected finite union of sub-elements  each satisfying Assumption~\ref{ass:meshReg}.
Moreover, the extension of the VEM a priori error analysis recently presented in~\cite{VEM-stability} for the $\spacedim=2$ case, indicates that it may be possible to relax the condition on the size of the interfaces.
This hypothesis is explored  numerically in Section~\ref{sec:numerics}. 
Therein, by not imposing any restrictions on the size of the edges in the mesh, we show that the performance of the method or the estimators are not affected in practice.
\\
  An immediate consequence of the above, simplifying, mesh regularity 
assumptions is
  that each element $\E$ admits a sub-triangulation 
$\Th^{\E}$, a partition of $\E$ into triangles when 
$\spacedim=2$ and tetrahedra when $\spacedim=3$, in such a 
way that the resulting global triangulation 
$\subtriang := \bigcup_{\E\in \Th} \Th^{\E}$ is shape 
regular.  
For $\spacedim=2$ this is obtained by joining each vertex 
of $\E$ with a point with respect to which $\E$ is 
starred. For $\spacedim=3$ the same procedure can be
applied starting from the corresponding triangulation of 
each face.
\end{remark}

Throughout the paper, we denote by $\Po{\ell} : \LTWO(\E) \rightarrow
\PE{\ell}$ the $\LTWO(\E)$-orthogonal projection onto the
space $\PE{\ell}$ of polynomials with total degree $\ell$, for any $\E\in\Th$ and $\ell
\in\mathbb{N}\cup\{0\}$.

\subsection{Virtual element spaces}
\label{sec:vemSpaces}
We begin by recalling the construction of the conforming virtual
element space from~\cite{UnifiedVEM}. For each $\Th$ and $\k\in {\mathbb N}$, we shall construct a \emph{virtual element space} $\Vh\subset H^1_0(\D)$ of order $\k$ in an
element-wise fashion; $\Vh$ will be of order $\k\in {\mathbb N}$ if,
for each element $\E \in \Th$, the space $\VhE := \Vh|_{\E}$ contains
the space $\PE{\k}$ of polynomials of degree $\k$ on $\E$.  In
general, the space $\VhE$ will also contain non-polynomial functions.
However, the distinctive idea of VEM is that of \emph{computability} based on degrees of freedom, which stems from the
view that the complement of $\PE{\k}$ in $\VhE$ is made up of
functions which are deemed expensive to evaluate.
\begin{definition}[Computability]
  \label{def:computable}
  A term is \emph{computable} if it may be evaluated using the data of
  the problem, the degrees of freedom, and the polynomial component of
  the virtual element space only.
\end{definition}

We shall consider two types of degrees of freedom: nodal values and
polynomial
moments.
\begin{definition}[Degrees of freedom]
  \label{def:dofs}
  Let $\omega\subset\Re^\spacedim$, $1\le \spacedim\le 3$, be an $\spacedim$-dimensional
  polytope, that is, a line segment, polygon, or polyhedron,
  respectively. For any regular enough function $v$ on $\omega$, we define the following sets of \emph{degrees of freedom}:
  \begin{itemize}
  \item $\NO{\omega}$ are the \emph{nodal values}. For a vertex $\vertex$
    of $\omega$, $\NO{\omega}_\vertex (v):=v(\vertex)$ and
    $\NO{\omega}:=\{\NO{\omega}_\vertex: \vertex \text{ is a
      vertex}\}$;
  \item $\MO{\omega}{l}$ are the \emph{polynomial moments} up to order
    $l$. For $l\ge
    0$,
    \begin{equation*}
      \MO{\omega}{\mindex}(v)=\frac{1}{\abs{\omega}} (v, \ma)_\omega
      \quad \text{ with}\quad m_{\mindex}
      :=
      \(\frac{ \vx - \vx_{\omega}}{h_{\omega}}\)^{\mindex} \text{ and}\quad \abs{\mindex}\le l,
    \end{equation*}
    where $\mindex$ is a multi-index with $\abs{\mindex} := \alpha_1
    +\cdots +\alpha_\spacedim$ and $x^{\mindex} := x_1^{\alpha_1} \dots
    x_\spacedim^{\alpha_\spacedim}$ in a local coordinate system, and $x_\omega$
    denoting the barycentre of $\omega$. Further,
    $\MO{\omega}{l}=\{\MO{\omega}{\mindex}:\abs{\mindex}\le l\}$. The
    definition is extended to $l=-1$ by setting
    $\MO{\omega}{-1}:=\emptyset$.
  \end{itemize}
\end{definition}

The local virtual element space is constructed recursively in space
dimensions. We first consider the case $\spacedim=2$.
On each edge interface $s\in\dE$ we take
$\Vhs:=\P{\s}{\k}$ and we define the auxiliary space $\biglocalspace$
as
\begin{equation}\label{eq:biglocalspace}
  \biglocalspace
  :=
  \left\{ \vh \in H^1(\D) : \Delta \vh \in \PE{\k} \text{ and } \vh|_{\s} \in \Vhs \text{ for all } \s\in\dE\right\},
\end{equation}
noting that $\PE{\k} \subset\biglocalspace\subset C^0(\overline{\E})$. The elements of $\biglocalspace$ can be uniquely identified by the
following set of degrees of freedom~\cite{EquivalentProjectors,UnifiedVEM}:
\begin{equation*}
\text{DoF}(\biglocalspace)
:=
\NO{\E}
\cup
\{\MO{\s}{\k-2} \text{ for each edge interface } \s\in\dE\}
\cup
\MO{\E}{\k}\}.
\end{equation*}
These degrees of freedom make the terms
$\Po{\k} \vh$ and $\Po{\k-1} \nabla \vh$ computable for any $\vh \in
\biglocalspace$~\cite{UnifiedVEM}; for instance, the projection $\Po{\k} \vh$ is given \emph{directly} by the internal degrees of freedom $\MO{\E}{\k}$.

The crucial property $\PE{\k} \subset\biglocalspace$ would still be satisfied by the smaller space obtained by requiring $ \Delta \vh \in \PE{\k-2}$ instead of
$ \Delta \vh \in \PE{\k}$ in~\eqref{eq:biglocalspace}. This is, indeed, the \emph{original} virtual element space introduced in~\cite{BasicsPaper}. However, since the elemental projection $\Po{\k}$ is \emph{not} computable in the original VEM space of~\cite{BasicsPaper}, a \emph{different} subspace of $\biglocalspace$  with the same dimension as the original space of~\cite{BasicsPaper} was introduced~\cite{EquivalentProjectors}. In the latter subspace definition, the $\LTWO$-projection onto $\PE{\k}$ is computable using the extra higher-order moments.

The crucial observation is that, in the polynomial space
$\PE{\k}\subset\biglocalspace$, the moments in $\MO{\E}{\k}\setminus \MO{\E}{\k-2}$ are redundant.
Hence, it is possible to construct $L^2$-stable projection operators
$\spaceProj{\k}:\biglocalspace\rightarrow \PE{\k}$ which \emph{only} depend
on the \emph{reduced} set of degrees of freedom
\begin{equation}\label{eq:dofs}
\text{DoF}(\VhE)
:=
\NO{\E}
\cup
\{\MO{\s}{\k-2} \text{ for each edge interface } \s\in\dE\}
\cup
\MO{\E}{\k-2}.
\end{equation}
In particular, and following~\cite{UnifiedVEM},  this can simply be taken as the  projection corresponding to the Euclidean inner product on $\text{DoF}(\VhE)$. This is, indeed, the choice used in all numerical tests presented in Section~\ref{sec:numerics}.

Given any such projection operator $\spaceProj{\k}$, we can
define a local virtual element space $\VhE\subset\biglocalspace$ by
clamping the internal higher-order moments:
\begin{equation}
  \label{eq:localVEMSpace}
  \VhE := \Big\{ \vh \in \biglocalspace :
\(\vh, m_\alpha\)_{\E}
  =
  \(\spaceProj{\k}\vh, m_\alpha\)_{\E}  \,\,\, \forall m_\alpha \text{ with } \k-1\le|\alpha|\le \k \Big\}.
\end{equation}
By construction, the space $\VhE$ is identified by the  degrees of
freedom
$\text{DoF}(\VhE)$ of~\eqref{eq:dofs},
used to define $\spaceProj{\k}$.
A counting argument shows that the cardinality of the above sets of
degrees of freedom is $\NE = \nuE + \nuE N_{1,\k-2} + N_{2,\k-2}$, where $N_{\spacedim, k} := \dim \P{\Re^{\spacedim}}{k}$.
Representative examples are illustrated in
Fig.~\ref{fig:dofs:hexa:conf-VEM}.
\begin{figure}[!t]
  \centering
  \begin{tabular}{cccc}
    \includegraphics[width=0.2\textwidth]{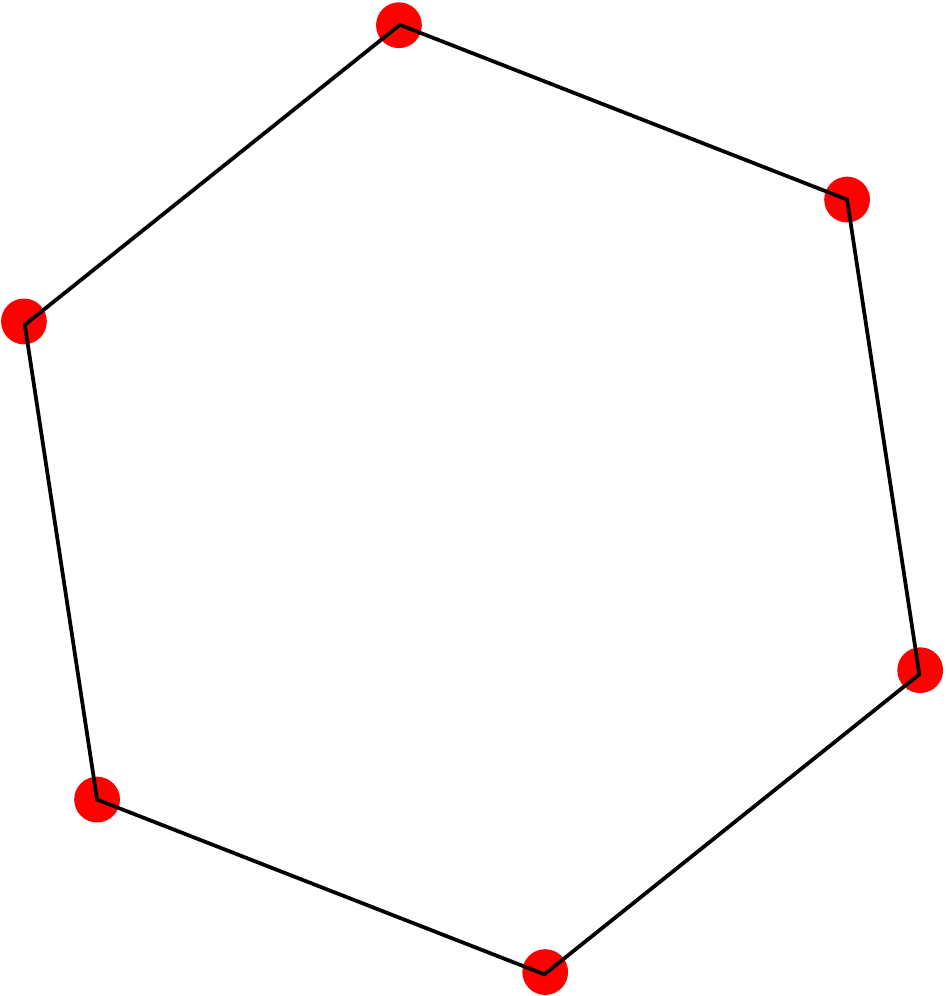} &\quad
    \includegraphics[width=0.2\textwidth]{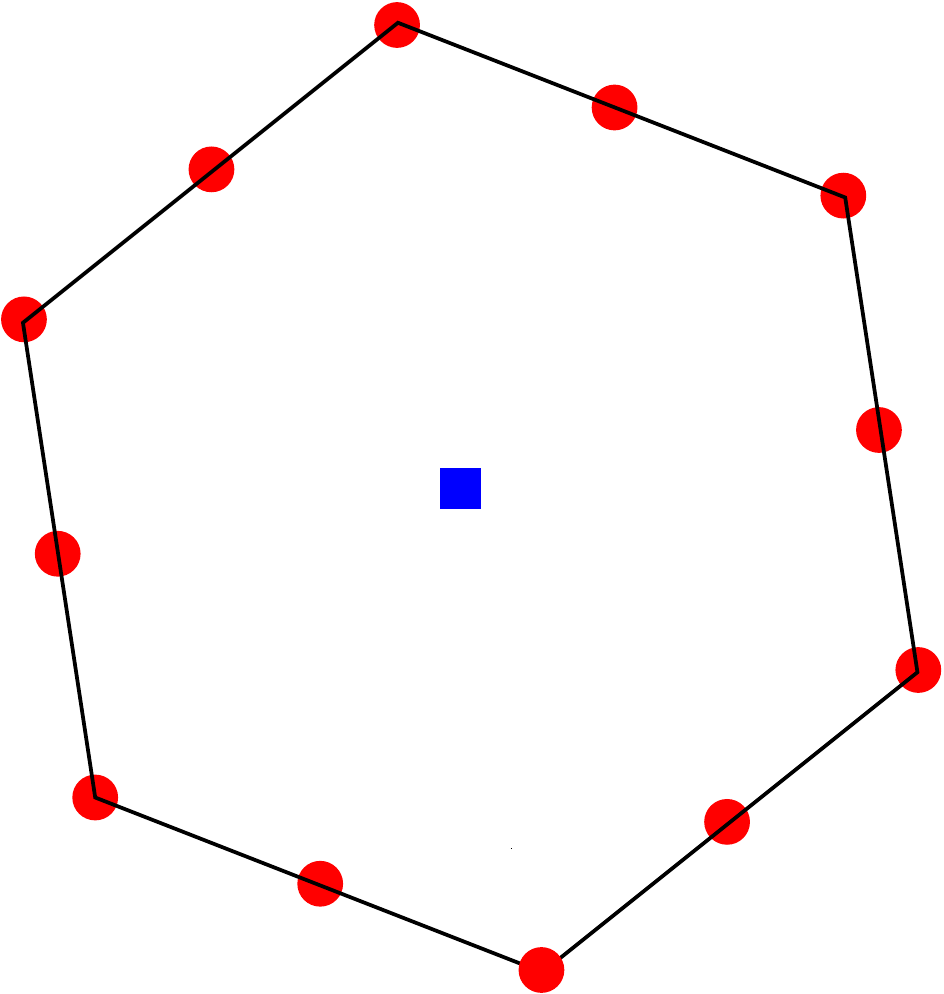} &\quad
    \includegraphics[width=0.2\textwidth]{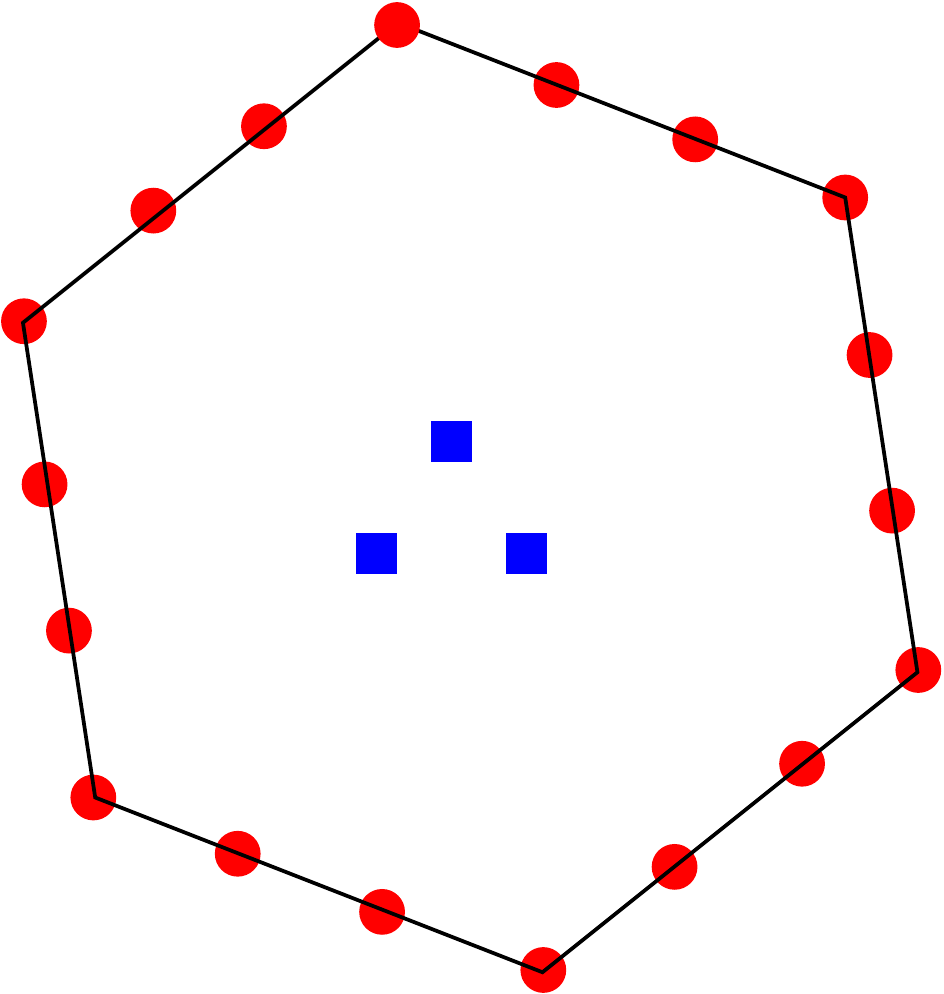} &\quad
    \includegraphics[width=0.2\textwidth]{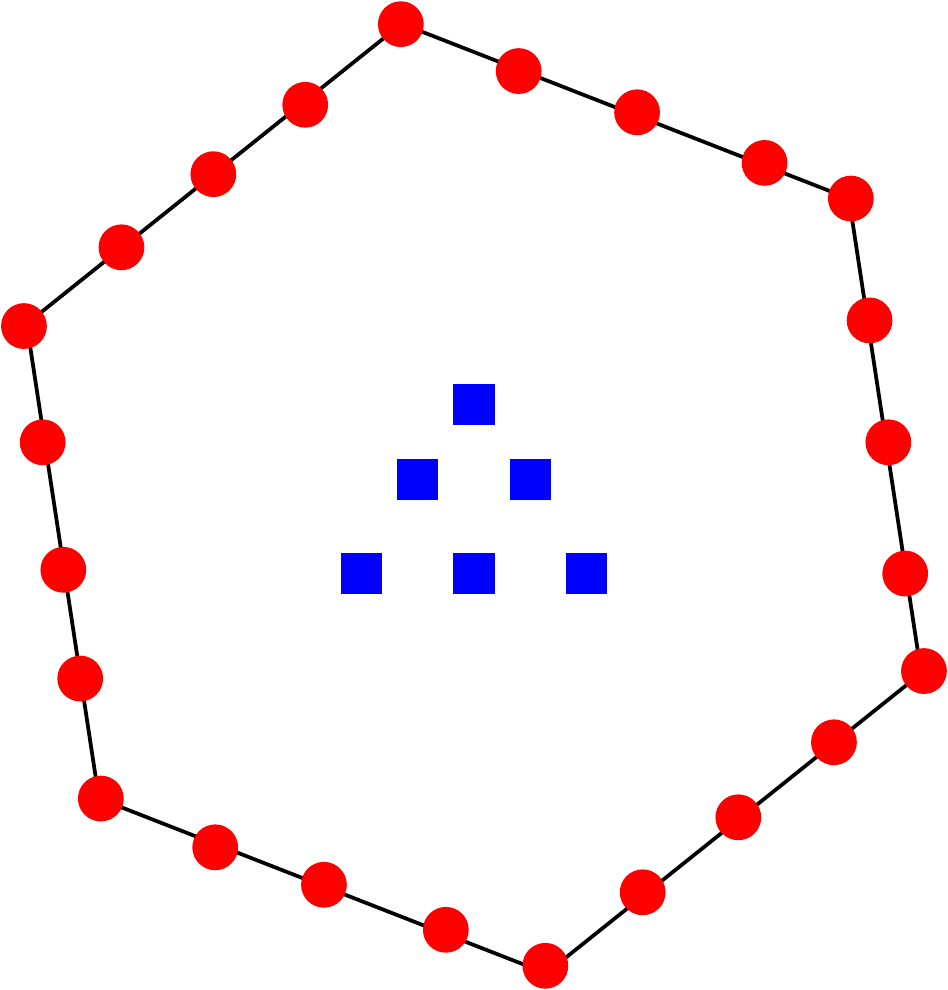} \\
    $\mathbf{\k=1}$ & $\mathbf{\k=2}$ & $\mathbf{\k=3}$ & $\mathbf{\k=4}$
  \end{tabular}
  \caption{The VEM degrees of freedom for a hexagonal mesh element for
    $\k=1,2,3,4$; nodal values and edge moments are marked by a
    circle; internal moments are marked by a square. The number of moments per edge is given by $N_{1,\k-2}=0,1,2,3$ and the number of internal moments is given by  $N_{2,\k-2}=0,1,3,6$ for $\k=1,2,3,4$, respectively.}
  \label{fig:dofs:hexa:conf-VEM}
\end{figure}
Further, we note that it is also possible to compute $\Po{\k} \vh$ and
$\Po{\k-1} \nabla \vh$ for each $\vh\in\VhE$
just from the reduced set of degrees of freedom since we can access the higher order moments through $\spaceProj{\k}$; we refer to~\cite[Section 4.1]{UnifiedVEM} for the details.

For the case $\spacedim=3$, we first (re-)define $\Vhs$ on each face
$s\in\dE$ to be the $2$-dimensional virtual element space given
by~\eqref{eq:localVEMSpace}. The construction of the local virtual
element space on $\E$ now follows by defining the
auxiliary space $\biglocalspace$ of~\eqref{eq:biglocalspace} and final
space $\VhE$ of~\eqref{eq:localVEMSpace} in exactly the same way as the 2-dimensional case.
In the 3-dimensional case, $\VhE$ is identified by the following set of degrees of freedom~\cite{UnifiedVEM}:
\begin{equation}\label{eq:dofs3d}
\text{DoF}(\VhE)
:=
\NO{\E}
\cup
\{\MO{\e}{\k-2} \text{ for each edge } \e\in\dE\}
\cup
\{\MO{\s}{\k-2} \text{ for each face } \s\in\dE\}
\cup
\MO{\E}{\k-2}.
\end{equation}
Therefore, the dimension of the
local space for $\spacedim = 3$ is $\NE = \nuE''+\nuE'N_{1,\k-2}+\nuE
N_{2,\k-2}+N_{3,\k-2}$ where $\nuE''$ and $\nuE'$ denote,
respectively, the number of vertices and edges of $\E$, cf.~\cite{UnifiedVEM}.
Representative examples are illustrated in
Fig.~\ref{fig:dofs:cube:conf-VEM}.
\begin{figure}[!t]
  \centering
  \begin{tabular}{cccc}
    \includegraphics[width=0.2\textwidth]{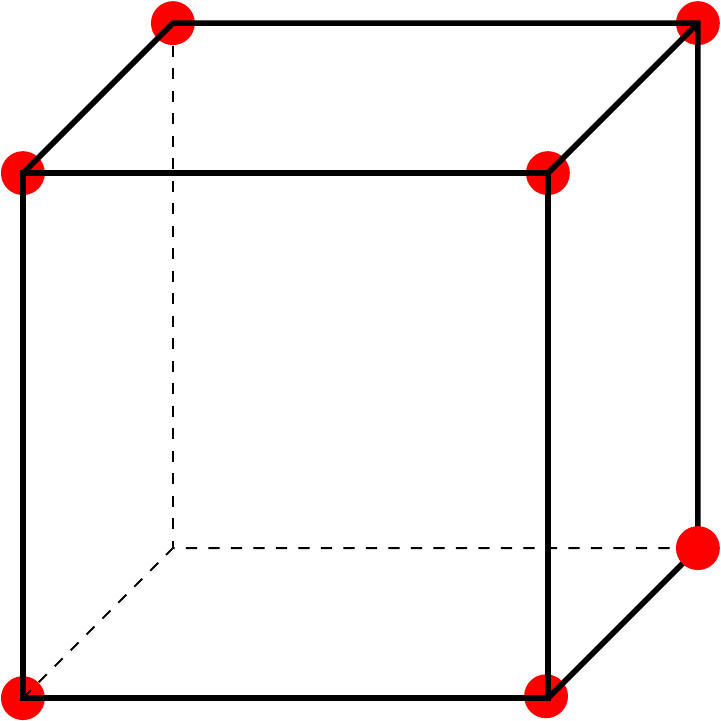} &\quad
    \includegraphics[width=0.2\textwidth]{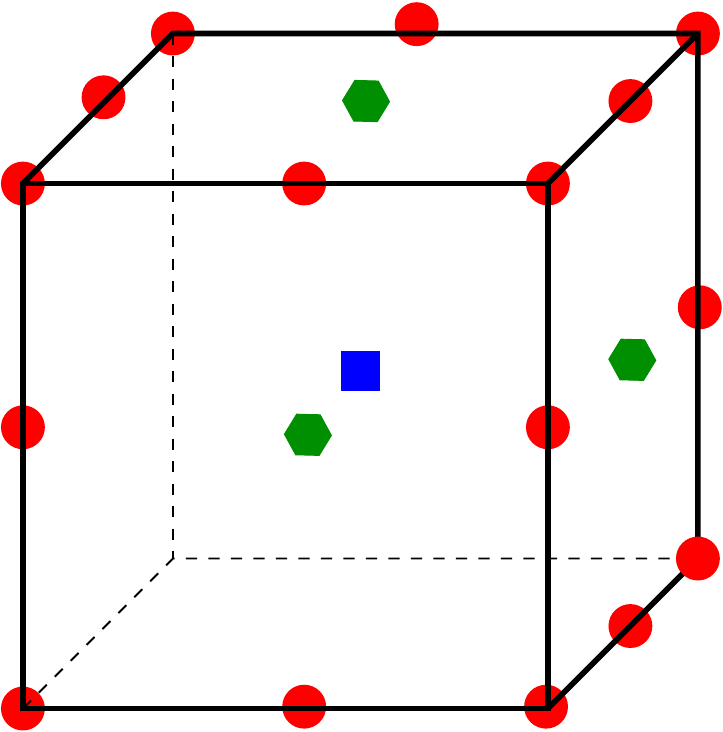} &\quad
    \includegraphics[width=0.2\textwidth]{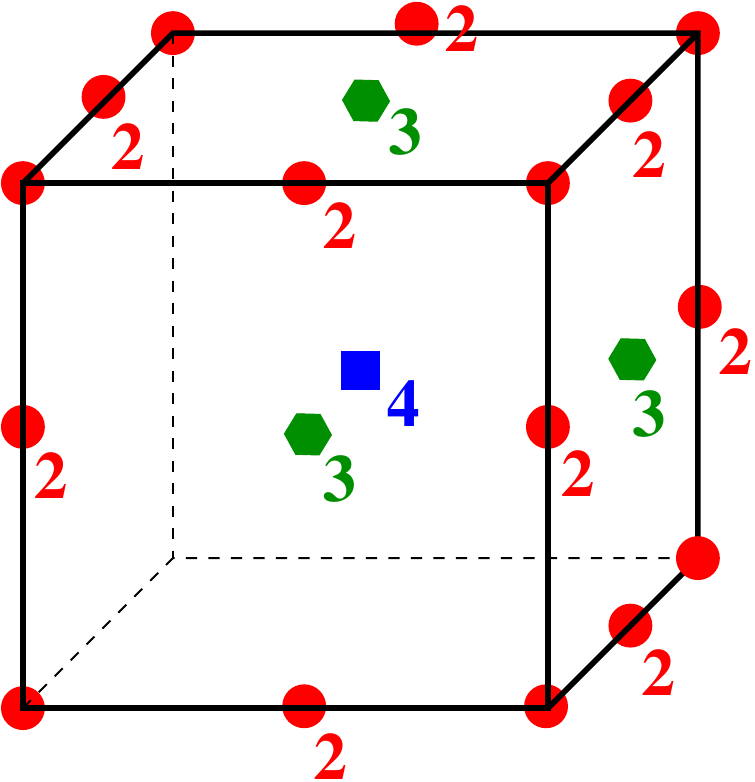} &\quad
    \includegraphics[width=0.2\textwidth]{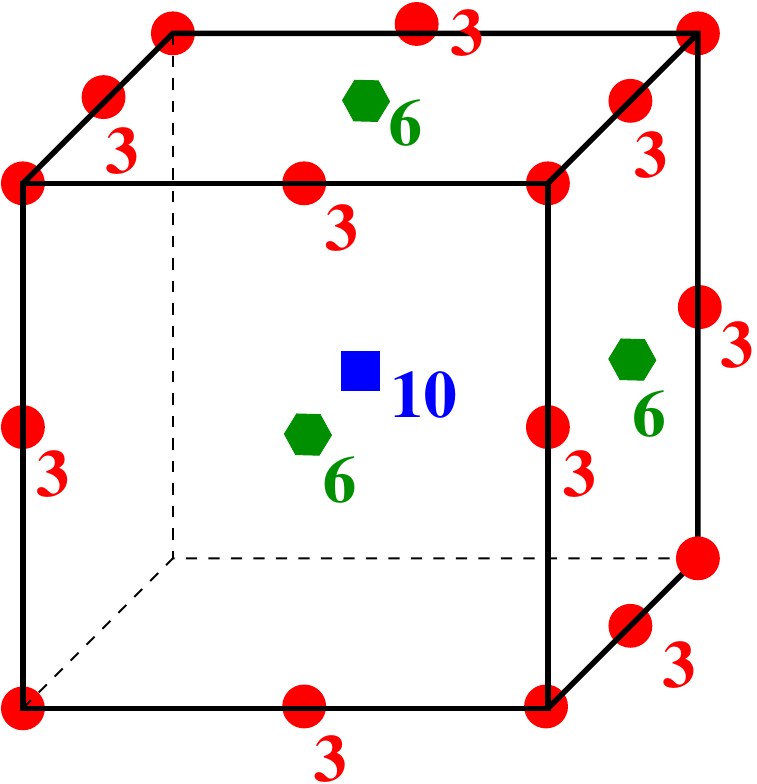} \\
    $\mathbf{\k=1}$ & $\mathbf{\k=2}$ & $\mathbf{\k=3}$ & $\mathbf{\k=4}$
  \end{tabular}
  \caption{The VEM degrees of freedom for a cubic mesh
	element for $\k=1,2,3,4$; nodal values and edge moments are marked
	by a circle; face moments are marked by a hexagon; internal moments are
    marked by a square. Only the internal degrees of freedom and those
	of the visible faces and edges are marked. The numeric labels
    indicate the number of degrees of freedom when there are more than
    $1$. The number of moments per edge is given by $N_{1,\k-2}=0,1,2,3$, the number of interface moments is given by  $N_{2,\k-2}=0,1,3,6$, and the number of internal moments is given by  $N_{3,\k-2}=0,1,4,10$ for $\k=1,2,3,4$, respectively.}
  \label{fig:dofs:cube:conf-VEM}
\end{figure}

Finally, the global space is constructed from these local spaces as
\begin{equation}\label{eq:GlobalSpace}
  \Vh := \left\{ \vh \in H^1_0(\D) : \vh |_{\E} \in \VhE \quad \forall \E \in \Th \right\},
\end{equation}
and the global degrees of freedom are obtained by
collecting the local ones, with the nodal and interface degrees of
freedom corresponding to internal entities counted only once.
Those on the boundary are fixed to be equal to zero in accordance with the
ambient space $H^1_0(\D)$.

\subsection{Discrete formulation}
We shall now recall the VEM for~\eqref{eq:newVariationalForm} introduced in~\cite{UnifiedVEM}.
For every $\E\in\Th$, let $\aE$ and $\bE$ be the elemental continuous
forms obtained by restriction of the forms in~\eqref{eq:a:def}
and~\eqref{eq:b:def} onto the element $\E$, respectively.  A virtual
bilinear form $\Ah: \Vh \times \Vh \rightarrow \Re$, is constructed
elementwise as
\begin{equation*}
  \Ah(\uh, \vh) = \sum_{\E \in \Th} \AhE(\uh, \vh)=\sum_{\E \in \Th}\ahE(\uh, \vh) + \bhE(\uh, \vh),
\end{equation*}
for any $\uh, \vh \in \Vh$. Here, $\AhE$ is a bilinear form over the
space $\VhE$, which is split into  the symmetric and
skew-symmetric discrete bilinear forms $\ahE$ and $\bhE$
corresponding to the continuous forms $\aE$ and $\bE$,
respectively.
To define $\AhE$ precisely, we begin by introducing the concept of \emph{admissible stabilising forms}.

\begin{definition}[Admissible stabilising forms]
  \label{def:admissibleStabilisingTerm}
  Let $\E\in\Th$. Two computable (in the sense of \dref{def:computable}),
  symmetric, and positive definite bilinear forms
  $\StaE_{1},\StaE_{0}:\VhE / \PE{\k}\times\VhE / \PE{\k}\rightarrow
  \Re$ are said to be local \emph{admissible} bilinear forms for
  stabilising the diffusion and reaction terms in~\eqref{eq:a:def} if
  they respectively satisfy
  \begin{align*}
    \Cstab^{-1} \intE |\sqrt{\diff} \nabla \vh|^2 \dx \le  &\StaE_{1}(\vh, \vh) \le \Cstab \intE |\sqrt{\diff} \nabla \vh|^2 \dx, \\
    \Cstab^{-1} \intE\reacSym \vh^2 \dx \le &\StaE_{0}(\vh, \vh)\le \Cstab \intE \reacSym \vh^2 \dx,
  \end{align*}
  for all $\vh\in \VhE / \PE{\k}$ for some constant $\Cstab$
  independent of $\E$ and $h$.
\end{definition}
A practical choice of admissible stabilising bilinear forms is given
in~\eqref{eq:scaling}. We note here a
trivial consequence of the above definition which will be useful in
the analysis: for all $\vh\in \VhE/ \PE{\k}$, we have
\begin{align}
  \norm{\nabla
    \vh }_{0,\E}^2
  &\le \frac{\Cstab}{\elipLowerE }  \StaE_{1}(
  \vh,
  \vh),
  \label{eq:bounddiff}
  \qquad
  \norm{
    \vh }_{0,\E}^2 \le \frac{\Cstab}{\reacLowerE }  \StaE_{0}(
  \vh,
  \vh),
\end{align}
where $\elipLowerE, \reacLowerE$ are the local counterparts of
$\elipLower, \reacLower$, respectively.

A \emph{virtual element stabilising term} $\StaE$ may then be defined as
the linear combination of any pair of admissible diffusion and
reaction stabilising forms:
\begin{align*}
  \StaE:= \stabConst_1 \StaE_{1} + \stabConst_0 \StaE_{0},
\end{align*}
with $\StaE_{1}$, $\StaE_{0}$ admissible stabilising bilinear forms
and $\stabConst_1$, $\stabConst_0$ positive constants. We prefer to
keep the dependence of the stabilising form on
the constants $\stabConst_1$ and $\stabConst_0$ explicit, to be able to study
their influence on the constants in the a posteriori bounds below.

For every $\E\in\Th$, the \emph{local symmetric} and
\emph{skew-symmetric discrete bilinear forms} $\ahE$ and $\bhE$ are
defined by
\begin{align}
  \ahE(\uh, \vh) &:= (\diff \Po{\k-1} \nabla
  \uh, \Po{\k-1} \nabla \vh)_{\E} + (\reacSym \Po{\k} \uh, \Po{\k}
  \vh) + \StaE((\Id - \Po{\k})\uh, (\Id - \Po{\k})\vh),
  \label{eq:ahE:def}
  \\
  \bhE(\uh, \vh) &:= \frac{1}{2} \left[ (\conv \cdot \Po{\k-1}\nabla \uh, \Po{\k}\vh) - (\Po{\k} \uh, \conv \cdot \Po{\k-1} \nabla \vh)\right],
  \label{eq:bhE:def}
\end{align}
respectively, for all $\uh, \vh\in \VhE$. Notice that all of the terms
in~\eqref{eq:ahE:def} and~\eqref{eq:bhE:def} are computable since
$\Po{\k}\vh$ and $\Po{\k-1}\nabla \vh$ are computable for any $\vh \in
\VhE$, and $\StaE$ is computable by assumption.

\begin{remark}[Polynomial consistency and stability]
If $p,q\in\PE{\k}\subset\VhE$, then $\ahE(p, q)=\aE(p, q)$
and $\bhE(p, q)=\bE(p, q)$. This property is referred to as
\emph{polynomial consistency} in the VEM
literature~\cite{BasicsPaper}. Furthermore,
Definition~\ref{def:admissibleStabilisingTerm} ensures the following
\emph{stability} property: there exists a positive constant $\Cstab$,
independent of $h$ and the mesh element $\E$, such that
\begin{equation}\label{eq:stability}
  (\Cstab)^{-1} \aE(\vh, \vh) \leq \, 
  \ahE(\vh, \vh) \leq \Cstab \aE(\vh, \vh),
\end{equation}
for all $\vh \in \VhE$. This, together with
the obvious identity $\bhE(\vh, \vh) = 0$ for all $\vh\in\VhE$ yields
the coercivity of $\Ah$. We refer to~\cite{UnifiedVEM} for the
details.
\end{remark}

The \emph{virtual element method} (VEM) then reads:
find $\uh \in \Vh$ such that
  \begin{equation}\label{eq:VEMproblem}
    \Ah(\uh, \vh)
	=(\forceh, \vh)\quad \forall \vh \in \Vh.
  \end{equation}
 where $\forceh := \Po{\k-1} \force$.

It may be shown that the problem~\eqref{eq:VEMproblem} possesses a
unique solution whenever~\eqref{eq:stability} is satisfied~\cite[Theorem 1]{UnifiedVEM}, along with optimal order a priori error bounds for the VEM solution in the $H^1$ and $\LTWO$ norms~\cite[Theorems 5 \& 6]{UnifiedVEM}.

\section{Approximation properties}
\label{subsec:approximation-properties}
The conforming virtual element space introduced above satisfies optimal
properties for approximating sufficiently smooth functions.
In particular, the theory in~\cite{BrennerScott} for star-shaped
domains may be used to prove the following theorem regarding the
approximation properties of the $L^2(\E)$-orthogonal projection to
polynomials.
\begin{theorem}[Approximation using polynomials]
  \label{thm:polynomialApproximation}
  Suppose that Assumption~\ref{ass:meshReg} is
  satisfied.
  Let $\E\in\Th$ and let $\Po{\ell} : \LTWO(\E) \rightarrow
  \PE{\ell}$, for $\ell\ge 0$, denote the $\LTWO(\E)$-orthogonal
  projection onto the polynomial space $\PE{\ell}$.
  Then, for any $\w \in H^{\reg}(\E)$, with $1 \leq \reg \leq \ell+1$,
  it holds
  \begin{equation*}
    \norm{\w - \Po{\ell}\w}_{0,\E} + h_{\E} \abs{\w - \Po{\ell}\w}_{1,\E} \leq \Cproj h_{\E}^\reg \abs{\w}_{\reg,\E}.
  \end{equation*}
  The positive constant $\Cproj$ depends only on the polynomial degree
  $\ell$ and the mesh regularity.
\end{theorem}

We shall make use of standard bubble functions on polygons/polyhedra below. A
bubble function $\psi_E\in H^1_0(E)$ for a polygon/polyhedron $\E$ can be
constructed piecewise as the sum of the (polynomial) barycentric
bubble functions (cf.~\cite{Verfurth:1996,Ainsworth-Oden:2000}) on
each $\spacedim$-simplex of the shape-regular sub-triangulation of the
mesh element $\E$ discussed in Remark~\ref{rem:shapeRegSubTri}.
\begin{lemma}[Interior bubble functions]
  \label{lem:bubEl}
  Let $\E\in\Th$ and let $\bubEl$ be the corresponding bubble
  function. There exists a constant $\Cbub$, independent of
  $h_{\E}$ such that for all $q \in \PE{\k}$
  \begin{equation*}
    \Cbub^{-1} \norm{q}_{0,\E}^2 \leq \intE \bubEl q^2 \dx \leq \Cbub \norm{q}_{0,\E}^2,
  \end{equation*}
  and
  \begin{equation*}
    \Cbub^{-1} \norm{q}_{0,\E} \leq \norm{\bubEl q}_{0,\E} + h_{\E} \norm{\nabla (\bubEl q)}_{0,\E} \leq \Cbub \norm{q}_{0,\E}.
  \end{equation*}
\end{lemma}

\begin{lemma}[Edge bubble functions]
  \label{lem:bubEdge}
  For $\E\in\Th$, let $\edge \subset \dE$ be a mesh interface and let
  $\bubEdge$ be the corresponding interface bubble function.  There exists
  a constant $\Cbub$, independent of
  $h_{\E}$ such that for
  all $q \in \P{\edge}{\k}$
  \begin{equation*}
    \Cbub^{-1} \norm{q}_{0,\edge}^2 \leq \edgeint \bubEdge q^2 \ds \leq \Cbub \norm{q}_{0,\edge}^2,
  \end{equation*}
  and
  \begin{equation*}
    h_{\E}^{-1/2}\norm{\bubEdge q}_{0,\E} + h_{\E}^{1/2} \norm{\nabla (\bubEdge q)}_{0,\E} \leq \Cbub \norm{q}_{0,\edge}.
  \end{equation*}
Here, with slight abuse of notation, the symbol $q$ is also used to denote the constant prolongation of $q$ in the direction normal to $\edge$.
\end{lemma}

We shall first use the above two results to prove an inverse inequality for virtual element functions, made possible by the fact that functions in $\biglocalspace$ and $\VhE$ have polynomial Laplacians.
\begin{lemma}[Inverse inequality]
  \label{lem:laplacianInverseEstimate}
  Suppose that Assumption~\ref{ass:meshReg} is satisfied. Let
  $\E\in\Th$ and let $\w \in H^1(\E)$ be such that $\Delta
  w\in\PE{\k}$.  There exists a constant $\Cinv$, independent of $\w$,
  $h$ and $\E$, such that
  \begin{equation*}
    \norm{\Delta \w}_{0,\E} \leq \Cinv h_{\E}^{-1} \abs{\w}_{1,\E}.
  \end{equation*}
\end{lemma}
\begin{proof}
  We first require an auxiliary polynomial inverse inequality
  $\norm{q}_{0,\E} \leq \Cinv h_{\E}^{-1} \norm{q}_{H^{-1}(\E)}$,
  valid for all $q \in \PE{\k}$.  This may be proven by selecting $\v
  = q \bubEl$ in the definition of the dual norm, viz.
  \begin{align}
    \norm{q}_{H^{-1}(\E)}
    :=
    \sup_{0\neq\v \in H^1_0(\E)} \frac{ \intE q \v \dx}{\norm{\nabla \v}_{0,\E}} \ge \frac{ \intE \bubEl q^2  \dx}{\norm{\nabla (\bubEl q)}_{0,\E}},
    \label{eq:inv:dualNormDef}
  \end{align}
  and using Lemma~\ref{lem:bubEl}. Applying this to
  $\Delta \w \in \PE{\k}$, we find that
  \begin{align*}
    \norm{\Delta \w}_{0,\E} &\leq \Cinv h_{\E}^{-1} \norm{\Delta \w}_{H^{-1}(\E)}.
  \end{align*}
  Now, using~\eqref{eq:inv:dualNormDef}, along with an integration by parts, we deduce
  \begin{align*}
    \norm{\Delta \w}_{H^{-1}(\E)} &= \sup_{0\neq\v \in H^1_0(\E)} \frac{ \intE\Delta \w  \v \dx}{\norm{\nabla \v}_{0,\E}}
    = \sup_{0\neq\v \in H^1_0(\E) } \frac{ -\intE \nabla \w \cdot \nabla \v \dx}{\norm{\nabla \v}_{0,\E}}.
  \end{align*}
  The result then follows by applying the Cauchy-Schwarz inequality.
\end{proof}

The above inverse estimate will be used to prove an approximation theorem (Theorem~\ref{thm:spaceApproximation} below) for the virtual element spaces considered in this work. The proof of Theorem~\ref{thm:spaceApproximation} is inspired by~\cite[Prop. 4.2]{SteklovVEM}, where a related result is obtained in the much simpler setting of the original virtual element space of~\cite{BasicsPaper} for $\spacedim=2$ only. As the construction in~\cite[Prop. 4.2]{SteklovVEM} does not appear to generalize to $d=3$, we use a different construction for the Cl\'ement-type interpolant below.

We begin by recalling some classical polynomial interpolation results
on simplicial triangulations. \assref{ass:meshReg} implies the existence of a globally
shape-regular sub-triangulation $\subtriang$ of $\Th$, cf. Remark~\ref{rem:shapeRegSubTri}.
We use this to define $\vClem$ as the classical Cl{\'e}ment
interpolant~\cite{Clement:1975} of $\v$ of degree $\k$ over the
sub-triangulation $\subtriang$.  Then, the following approximation
estimates hold~\cite{Clement:1975}
 for any
$\v \in H^1(\D)$:
\begin{align}
  \norm{\v - \vClem}_{0,T}+ h\abs{\v - \vClem}_{1,T} \leq
  \CclemT
  h \abs{\v}_{1,\patch{T}},
  \label{eq:subtriClementBound}
\end{align}
for all $T\in\subtriang$, with $\CclemT$ a positive constant depending only on the polynomial degree $\k$ and on the mesh regularity.  Here, $\patch{T}$ denotes the usual finite element patch relative to $T$.

\begin{theorem}[Approximation using virtual element functions]
	\label{thm:spaceApproximation}
	Suppose that Assumption~\ref{ass:meshReg} is satisfied and let $\Vh$ denote the virtual element space~\eqref{eq:GlobalSpace}.
	For $\v \in H^1(\D)$, there exists a $\vI \in \Vh$, such that, for all elements $\E\in\Th$, we have
	\begin{equation*}
		\norm{\v - \vI}_{0,\E} +h_{\E}\abs{\v - \vI}_{1,\E} \leq \Cclem h_{\E} \abs{\v}_{1,\patch{\E}},
	\end{equation*}
 $\Cclem$ being a positive constant, depending only on the polynomial degree $\k$ and the mesh regularity.
\end{theorem}
\begin{proof}
We denote by $\vClem$ the Cl{\'e}ment interpolant defined over a sub-triangulation $\subtriang$ and satisfying~\eqref{eq:subtriClementBound}. It is assumed that all edges of the polygonal/polyhedral mesh $\Th$ are also edges of the sub-triangulation $\subtriang$, cf. Remark~\ref{rem:shapeRegSubTri}.

	{\bf Case $\spacedim=2$.}
	We start by interpolating $\vClem$ into the enlarged virtual element space $\bigspace$. More specifically,
we define $\wI$ elementwise as the solution of the problem
	\begin{align}\label{eq:wIprob}
		\begin{cases}
			-\Delta \wI = -\Delta \Po{\k}\vClem &\text{ in } \E, \\
			\wI = \vClem &\text{ on } \partial \E.
		\end{cases}
	\end{align}
	Then, since $\Delta \Po{\k}\vClem  \in \PE{\k-2} \subset \PE{\k}$ and $\vClem$ is a polynomial of degree $\k$ on each edge of $\E$, we may conclude that $\wI|_{\E} \in \biglocalspace$.
	Moreover, since $\vClem$ is continuous on $\D$, it follows that $\wI \in \bigspace$.

	Arguing as in~\cite[Proposition 4.2]{SteklovVEM}, we may show that
	\begin{align}\label{eq:wIPbound}
		\abs{\wI - \Po{\k}\vClem}_{1,\E} &\leq \abs{\vClem - \Po{\k}\vClem }_{1,\E},
	\end{align}
	and, therefore,
	\begin{align}\label{eq:wIbound}
			\abs{\vClem - \wI}_{1,\E} \leq 2\abs{\vClem - \Po{\k}\vClem }_{1,\E}.
	\end{align}
	Now, $\wI$ allows us to construct an interpolant $\vI\in\Vh$ using the definition of $\VhE$ (given in~\eqref{eq:localVEMSpace}) on each $\E\in\Th$.
	By definition, the two interpolants $\vI$ and $\wI$ are equal on the mesh skeleton $\Edges$ and for all $\E\in\Th$,
	$\MO{\E}{\mindex}(\vI)=\MO{\E}{\mindex}(\wI)$ if $\abs{\mindex}\le \k-2$, while $\MO{\E}{\mindex}(\vI)=\MO{\E}{\mindex}(\spaceProj{\k}\wI)$ if $\k-1\le\abs{\mindex}\le \k$.
	Consider, now, $\abs{\wI-\vI}_{1,\E}$ on each $\E\in\Th$.  Integration by parts yields
	\begin{align}
		\abs{\wI - \vI}_{1,\E}^2
			&= -(\Delta (\wI - \vI), \wI - \vI)_{\E},
			\label{eq:clement:integratedByParts}
	\end{align}
	as $\wI$ and $\vI$ coincide on $\dE$.
	Since $\wI - \vI \in \biglocalspace$, we have $\Delta (\wI - \vI) \in \PE{\k}$. Let $q_{\k, \k-1} \in \PE{\k} / \PE{\k-2}$ be defined by $q_{\k, \k-1}=\Delta (\wI - \vI) - \Po{\k-2} \Delta (\wI - \vI)$.
	Identity~\eqref{eq:clement:integratedByParts} can then be rewritten as
	\begin{align*}
		\abs{\wI - \vI}_{1,\E}^2 &= -(q_{\k, \k-1}, \wI - \vI)_{\E} = -(q_{\k, \k-1}, \wI - \spaceProj{\k} \wI)_{\E},
	\end{align*}
	since $\vI$ and $\wI$ have the same moments of up to degree $\k-2$, while $\vI$ and $\spaceProj{\k}\wI$ share the same moments of degree $\k$ and $\k-1$.
	The Cauchy-Schwarz inequality then implies that
	\begin{align*}
		\abs{\wI - \vI}_{1,\E}^2 &\leq \norm{q_{\k, \k-1}}_{0,\E} \norm{\wI - \spaceProj{\k} \wI}_{0,\E}.
	\end{align*}
	Further, from the stability of the $L^2$ projection we get
	\begin{align*}
		\norm{q_{\k,\k-1}}_{0,\E} = \norm{(\Id - \Po{\k-2}) \Delta (\wI - \vI)}_{0,\E} \leq \norm{\Delta (\wI - \vI)}_{0,\E},
	\end{align*}
	where $\Id$ denotes the identity operator on the space $\PE{\k}$.
	Thus,
	\begin{align*}
		\abs{\wI - \vI}_{1,\E}^2 &\leq \norm{\Delta (\wI - \vI)}_{0,\E} \norm{\wI - \spaceProj{\k} \wI}_{0,\E}
		 \leq \Cinv h_{\E}^{-1} \abs{\wI - \vI}_{1,\E} \norm{\wI - \spaceProj{\k} \wI}_{0,\E},
	\end{align*}
	by Lemma~\ref{lem:laplacianInverseEstimate}.
	Further, adding and subtracting $\Po{\k}\wI$ and using the stability of $\spaceProj{\k}$ and then using the Poincar\'e inequality (either on each  $2$-simplex of the shape-regular sub-triangulation $\E$  or directly on $\E$, cf.~\cite{veeser_verfurth:2012}), we obtain
	\begin{align*}
		\abs{\wI - \vI}_{1,\E} \leq \Cinv (1+C_0^*) h_{\E}^{-1} \norm{\wI - \Po{\k} \wI}_{0,\E}\leq \Cp \Cinv (1+C_0^*) \abs{\wI - \Po{\k} \wI}_{1,\E},
	\end{align*}
	for some uniform constants $C_0^*$ and $\Cp>0$ which depend on the shape regularity constant.
	Then, the triangle inequality, the stability of $\Po{\k}$ with constant, say, $C_0$, and~\eqref{eq:wIbound} imply that
	\begin{align*}
		\abs{\wI - \vI}_{1,\E}
			&\leq \Cp\Cinv (1+C_0^*)
			(\abs{\wI - \Po{\k}\vClem}_{1,\E} + \abs{\Po{\k}\vClem - \Po{\k} \wI}_{1,\E}) \\
			&= \Cp\Cinv (1+C_0^*)
			(\abs{\wI - \Po{\k}\vClem}_{1,\E} + C_0\abs{\vClem - \wI}_{1,\E}) \\
			&\leq \Cp\Cinv(1+C_0^*)(1+2C_0)
			\abs{\vClem - \Po{\k}\vClem}_{1,\E}.
	\end{align*}
	Finally, the triangle inequality, the above bound, and~\eqref{eq:wIbound}, imply
	\begin{align}
		\abs{\vClem - \vI}_{1,\E} \leq \abs{\vClem - \wI}_{1,\E}+\abs{\wI - \vI}_{1,\E}
		\leq C_1\abs{\vClem - \Po{\k}\vClem}_{1,\E},
		\label{eq:vCbound}
	\end{align}
	with $C_1:= (2+\Cp\Cinv(1+C_0^*)(1+2C_0))$.
	Since $\vI$ and $\vClem$ are equal on $\dE$, we may apply the Poincar\'e inequality to this to obtain a bound on $\norm{\vClem - \vI}_{0,\E}$, with an extra power of $h_\E$.

	The required bounds of $\abs{\v - \vI}_{r,\E}$, $r=0,1$, now follow by the triangle inequality,  adding and subtracting $\v$ and $\Po{\k}\v$ to the right-hand side of~\eqref{eq:vCbound}, using once again the triangle inequality, and applying the bounds~\eqref{eq:subtriClementBound} and Theorem~\ref{thm:polynomialApproximation}.

	{\bf Case $\spacedim=3$.}
    The proof in this case is based on using on each face $\s\in\dE$ the construction just considered for $\spacedim=2$ and then extending this inside $\E$.

	Let $\biglocalspaceS$ and $\Vhs$ be the interface spaces respectively defined by~\eqref{eq:biglocalspace} and~\eqref{eq:localVEMSpace} applied to the interface.
	For each $\s\in\dE$, we consider $\wIs\in\biglocalspaceS$ as the solution of the $2$-dimensional boundary value problem~\eqref{eq:wIprob} set on $\s$, with $\vClem$ representing the three-dimensional Cl\'ement interpolant of $\v$ with respect to the 3-dimensional sub-triangulaiton $\subtriang$.

	Further, from $\wIs$ we may use the $2$-dimensional construction to obtain $\vIs\in\Vhs$. This face interpolant satisfies~\eqref{eq:vCbound} with $\s$ in place of $\E$, namely:
	\begin{align}
		\abs{\vClem - \vIs}_{1,\s}\leq C_1\abs{\vClem - \Pos{\k}\vClem}_{1,\s},
		\label{eq:vCbound_face}
	\end{align}
	with $\Pos{\k}\vClem$ denoting the $L^2$-projection of the restriction of $\vClem$ to $\s$.
	Collecting the face-wise definitions we obtain a continuous interpolant $\vIdE$ on $\dE$. With this, we first construct $\wI$ on $\E$ as the solution of the problem
		\begin{align*}
		\begin{cases}
			-\Delta \wI = -\Delta \Po{\k}\vClem \text{ in } \E, \\
			\wI = \vIdE \text{ on } \partial \E,
		\end{cases}
	\end{align*}
	so that $\wI\in\biglocalspace$ by definition, as in the case $\spacedim = 2$ (cf.~\eqref{eq:biglocalspace}).

	In view of bounding $\abs{\wI - \Po{\k}\vClem}_{1,\E}$, it is convenient to first split the trace $(\wI- \Po{\k}\vClem)|_{\dE}=(\vIdE-\vClem|_{\dE})+(\vClem -\Po{\k}\vClem)|_{\dE}$.
	Recall that, for all $\s\in\dE$, we have $(\vIs-\vClem)|_{\partial\s}=0$. Moreover, by Assumption~\ref{ass:meshReg}, over $s$ we may construct  a shape-regular pyramid  $P_s\subset\E$  with $|P_s|\ge \meshReg |\E|$.   By the Trace Theorem applied to $\s\in\partial P_s$, there exists ${\varphi}_s\in H^1(P_s)$ with ${\varphi}_s |_{\partial P_s \setminus s} = 0$ and a constant $C_{{\rm T}}>0$ such that
	\begin{align*}
	\abs{{\varphi}_s}_{1,P_s}\le C_{{\rm T}}\norm{\vIs-\vClem}_{1/2,s}.
	\end{align*}
	The constant ${C}_{\rm T}$ can be bounded uniformly over all $\s$ by a generalised scaling argument, cf.~\cite{UnifiedVEM} and the references therein.
	Hence, defining ${\varphi}=\sum_{\s\in\dE}{\varphi}_s+\vClem-\Po{\k}\vClem$, where each  ${\varphi}_s$ should be interpreted as its extension to zero on $\E$, we have by construction that ${\varphi}|_{\dE}=(\wI- \Po{\k}\vClem)|_{\dE}$. Thus, as in the case $\spacedim=2$, we have
	\begin{align}\label{eq:iccc_bound}
	\abs{\wI - \Po{\k}\vClem}_{1,\E} &\le \abs{{\varphi}}_{1,\E}\le \sum_{\s\in\dE}\abs{{\varphi}^s}_{1,P_s}+\abs{\vClem-\Po{\k}\vClem}_{1,\E}\le C_{{\rm T}}\sum_{\s\in\dE}\norm{\vIs-\vClem}_{1/2,s}+\abs{\vClem-\Po{\k}\vClem}_{1,\E}.
	\end{align}
	It just remains to bound the first term on the right-hand side. To this end, we use the Sobolev Interpolation Theorem and Poincar\'e inequality (facewise, cf. the case $\spacedim=2$ above):
	\begin{align}
	\nonumber
	\norm{\vIs - \vClem }_{1/2,\s}^2&\le \norm{\vIs - \vClem }_{0,\s}^2+C_{\rm S}\norm{\vIs - \vClem }_{0,\s}\abs{\vIs - \vClem }_{1,\s}\\
	\nonumber
	&\le (1+C_{\rm S}h_\E^{-1})\norm{\vIs - \vClem }_{0,\s}^2+C_{\rm S}h_\E\abs{\vIs - \vClem }_{1,\s}
	\\
	\nonumber
	&\le (\Cp^2 (h_\E+C_{\rm S})+ C_{\rm S}) h_\E \abs{\vIs - \vClem }_{1,\s}^2
	\\&
	\label{eq:vic_bound}
	\le (\Cp^2 (h_\E+C_{\rm S})+ C_{\rm S}) C_1 h_\E \abs{\vClem - \Pos{\k}\vClem}_{1,\s}^2,
	\end{align}
	for some constant ${C}_{\rm S}>0$
	which, again, can be bounded uniformly over all $s$ by a generalised scaling argument.
	To obtain the last bound above  we used~\eqref{eq:vCbound_face} applied to $\s\in\dE$.
	The interface terms above are further bounded by applying Theorem~\ref{thm:polynomialApproximation}, yielding
	\begin{align*}
	\abs{\vClem - \Pos{\k}\vClem}_{1,\s}
	\le
	\Cproj \abs{\vClem}_{1,\s}\le
	 \Cproj h_{\E}^{-1/2}\abs{\vClem}_{1,\E}.
	\end{align*}
	Using this bound in~\eqref{eq:vic_bound} and the latter in~\eqref{eq:iccc_bound}, we finally obtain
	\begin{align}
	\label{eq:wI_bound3}
	\abs{\wI - \Po{\k}\vClem}_{1,\E}&\leq C_2 \abs{\vClem}_{1,\E},
	\end{align}
	with $C_2>0$ depending on the (uniformly bounded) number $\nuE$ of interfaces of $\E$  and on the constants $C_{{\rm T}}$, $\Cp, C_{{\rm S}}, \Cinv$, and  $\Cproj$.

	Now, given $\wI$, we can construct an interpolant $\vI\in\Vh$ exactly as in the $2$-dimensional case and following the same (dimension-independent) argument derive the bound~\eqref{eq:vCbound}. This latter bound, combined with~\eqref{eq:wI_bound3}, yields
		\begin{align}
		\abs{\vClem - \vI}_{1,\E} \leq 2\abs{\vClem-\Po{\k}\vClem}_{1,\E}+\abs{\wI - \Po{\k}\vClem}_{1,\E}
		\leq C_3 \abs{\vClem }_{1,\E},
		\label{eq:vCbound3}
	\end{align}
	for some $C_3>0$ depending on $C_1$ and $C_2$.

	From~\eqref{eq:vCbound3} we can derive the required bound in the $L^2$-norm by resorting to the scaled Poincar\'e-Friedrichs inequality~\cite{Brenner-PF} and recalling~\eqref{eq:vCbound3}:
	\begin{align}
		\nonumber
		\norm{\vClem - \vI}_{0,\E} &\leq \Cp\left(h_\E\abs{\vClem - \vI}_{1,\E}+h_\E^{-1/2}\abs{\int_{\dE} (\vClem - \vI)ds} \right)\\
		 &\leq \Cp(C_3 h_\E\abs{\vClem }_{1,\E}+h_\E^{1/2}\sum_{s\in\dE}\norm{\vClem - \vIs}_{0,\s})
		\label{eq:L2bound1}
	\end{align}
	The interface terms on the right-hand side can be further bounded using the Poincar\'e inequality once more and~\eqref{eq:vCbound_face}:
	\begin{align*}
		\nonumber
		\norm{\vClem - \vIs}_{0,\s} &\leq \Cp h_\E\abs{\vClem - \vIs}_{1,\s} \leq \Cp C_1 h_\E \abs{\vClem - \Pos{\k}\vClem}_{1,\s}\\
		&\leq  \Cp C_1 \Cproj h_\E \abs{\vClem}_{1,\s}\leq \Cp C_1 \Cproj h_\E^{1/2}\abs{\vClem}_{1,\E}.
	\end{align*}
	Finally, combining this bound with~\eqref{eq:L2bound1} yields
	\begin{align*}
		\norm{\vClem - \vI}_{0,\E} &\leq C_4 h_\E \abs{\vClem}_{1,\E},
	\end{align*}
	with $C_4>0$ depending on $\Cp, C_1,$ $C_3$, $\Cproj$, and $\nuE$.

	The statement of the theorem now follows, as in the case $\spacedim=2$.
\end{proof}

\begin{remark}\label{rem:constants}
	For $\spacedim=3$, the proof of the above VEM approximation result makes use of both the Trace Theorem and Sobolev Interpolation Theorem applied to each mesh interface. This was necessitated by the hierarchical construction of the local virtual element spaces with respect to spactial dimension. The associated constants are uniformly bounded but depend on the polygonal shape of the mesh interfaces, and as such are not easily accessible  in general.
	However, if the mesh interfaces are triangular or the method is constructed on the sub-triangulation of each mesh interface, the proof does only depend on easily computable quantities. 
\end{remark}

\section{A posteriori error analysis}
\label{sec:aposteriori}
We shall now derive a residual-type a posteriori error bound for the
error in the standard energy norm:
\begin{equation}
  \label{eq:energyNorm}
  \triplenorm{\v}_{\omega}^2
  :=
  \norm{\sqrt{\diff} \nabla \v}_{0,\omega}^2 + \norm{\sqrt{\reacSym} \v}_{0,\omega}^2,
\end{equation}
for $\v \in H^1(\omega)$, for any $\omega \subseteq \D$.  The
coercivity and continuity of the bilinear form $\A$ in this
norm follow from the assumptions on the coefficients $\diff$ and $\reacSym$, which
ensure that for $\v \in H^1_0(\D)$,
\begin{align*}
  (\Cequiv)^{-1} \norm{\v}_{1,\D} \leq  \triplenorm{\v}_{\D} \leq  \CequivUp \norm{\v}_{1,\D},
\end{align*}
where $\Cequiv := \sqrt{(1 + \Cpf)/{\elipLower}}$, with
$\Cpf$ the Poincar\'e-Friedrichs constant, and $\CequivUp :=\sqrt{
\max\{\elipUpper, \norm{\reacSym}_{\infty}\}}$,
cf.~\cite{UnifiedVEM}.
The coercivity and continuity of $\Ah$ in this norm are then inherited from $\A$ through the virtual element stability property~\eqref{eq:stability}.

To account for the effects of data oscillation, we
introduce the following piecewise-polynomial approximations of the
PDE coefficients:
\begin{align}
  \diffh \approx \diff, \qquad \convh \approx \conv, \qquad \reach \approx \reac.
  \label{eq:dataApproximation}
\end{align}

For quantities which may be discontinuous across the mesh skeleton, we
define the jump operator $\jump{\cdot}$ across a mesh interface $\s\in
\Edges$ as follows. If $\s \in \InternalEdges$, then there exist
$\E^+$ and $\E^-$ such that $\s\subset \partial\E^+\cap\partial\E^-$.
Denote by $\vec{v}^{\pm}$ the trace of the vector-valued function $\vec{v}|_{\E^{\pm}}$ on $\s$ from within
$\E^{\pm}$ and by $\n_\s^{\pm}$ the unit outward normal on $\s$ from
$\E^{\pm}$. Then, $\jump{\vec{v}}:=\vec{v}^+ \cdot \n_{\s}^{+} + \vec{v}^- \cdot \n_{\s}^{-}$.  If, on
the other hand, $\s \in \BoundaryEdges$, then $\jump{\vec{v}}:=\vec{v} \cdot \n_{\s}$,
with $\vec{\v}$ representing the trace of $\vec{\v}$ from within the element $\E$
having $\s$ as an interface and $\n_{\s}$ is the unit outward normal
on $\s$ from $\E$.

\subsection{The residual equation}
Define $\err := \u - \uh \in H^1_0(\D)$, and let $\v\in H^1_0(\D)$.
Then, we have, respectively,
\begin{align}
	\A(\err, \v) &=(\force, \v) - \A(\uh, \chi)-\A(\uh, v-\chi)\notag\\
&=(\force, \v)- ( \forceh, \chi) + \A_h(\uh, \chi)- \A(\uh, \chi)-\A(\uh, v-\chi)	\notag\\
&=(\force-\forceh, \chi)+  (\force, \v-\chi) +\A_h(\uh, \chi)-\A(\uh, \chi)- \A(\uh, \v-\chi)\label{eq:realibility:originalAPostSplitting}
\end{align}
for any $\chi\in \Vh$, since $\u$ satisfies the weak form of the PDE
problem and $\uh$ is the virtual element solution. Notice that, in contrast to a posteriori bounds for standard
finite element approximations, additional terms appear in the virtual
element residual equation. These terms represent the \emph{virtual inconsistency} of the VEM.

\subsection{A posteriori error bound}
We shall estimate each term on the right-hand side
of~\eqref{eq:realibility:originalAPostSplitting} separately, to arrive
to a computable error bound. To this end, an integration by parts and
straightforward manipulation yields the identity
\begin{align*}
  \A(\uh, w)
  &=
  \left( - \nabla \cdot \diff \Po{\k-1} \nabla \uh +\conv \cdot\Po{\k-1}\nabla \uh + \reac \Po{\k} \uh, w \right) + \sum_{\s \in \Edges} \ints \jump{ \diff \Po{\k-1} \nabla \uh} w \ds +
  \\
  &\qquad +
  \left(\diff (\Id - \Po{\k-1}) \nabla \uh, \nabla w) + (\conv \cdot (\Id - \Po{\k-1}) \nabla \uh, w) + (\reac (\Id - \Po{\k}) \uh, w \right),
\end{align*}
for any $w \in H^1_0(\D)$. Using this and the data approximations introduced
in~\eqref{eq:dataApproximation},
\eqref{eq:realibility:originalAPostSplitting} may be rewritten as
\begin{align}
  \A(\err, \v)
  =&
  \sum_{\E \in \Th} \left(
 	\left( \elresid, \v-\chi \right)
  	+ \left( \eldataresid, \v-\chi \right)
  	+ \nonPolyResid^{\E}(\uh,\v-\chi)
  	\right)
  - \sum_{\s \in \Edges} \left(
  	\left( \edgeresid, \v -\chi\right)_{0,\s}
  	+ \left( \edgedataresid, \v-\chi \right)_{0,\s}
  	\right) \notag\\
  &\qquad\qquad + (\force-\forceh, \chi)
  +\A_h(\uh, \chi)-\A(\uh, \chi)\label{eq:residualEquation}
\end{align}
for any $\v \in H^1_0(\D)$, $\chi \in \Vh$, where
\begin{align}
  \elresid
  &:=
  (\forceh + \nabla \cdot \diffh \Po{\k-1} \nabla \uh -\convh \cdot \Po{\k-1} \nabla \uh - \reach \Po{\k} \uh)|_{\E}, \label{eq:elResidual}\\
  \edgeresid
  &:=
  \left.\jump{ \diffh \Po{\k-1} \nabla \uh}\right|_{\s}, \label{eq:edgeResidual}\\
  \eldataresid
  &:=
  (\force - \forceh + \nabla \cdot (\diff - \diffh) \Po{\k-1} \nabla \uh - (\conv - \convh) \cdot \Po{\k-1} \nabla \uh - (\reac - \reach) \Po{\k} \uh)|_{\E}, \label{eq:elDataResidual}\\
  \edgedataresid
  &:=
  \left.\jump{ (\diff - \diffh) \Po{\k-1} \nabla \uh}\right|_{\s}, \label{eq:edgeDataResidual}
\end{align}
are the element and edge residuals, and the element and edge data
oscillation terms, respectively, and
\begin{align*}
  \nonPolyResid^{\E}(\wh,\v)
  :=
  (\diff (\Id - \Po{\k-1}) \nabla \wh, \nabla \v)_{\E} + (\conv \cdot (\Id - \Po{\k-1}) \nabla \wh, \v)_{\E} + (\reac (\Id - \Po{\k}) \wh, \v )_{\E},
\end{align*}
is the `virtual' residual.

\begin{theorem}[Upper bound]
	\label{thm:reliability}
	Let $\uh \in \Vh$ be the virtual element solution to problem~\eqref{eq:VEMproblem}.
	Then, there exists a constant $C$, independent of $h$, $\u$ and $\uh$, such that
	\begin{align*}
		\triplenorm{\u - \uh}^2_{\D} \leq C \sum_{\E \in \Th} (\estimator^{\E} + \dataest^{\E} + \stabest^{\E} + \virtualosc^{\E}),
	\end{align*}
	where
	\begin{align*}
		\estimator^{\E}&:= h_{\E}^2 \norm{\elresid}^2_{0,\E}+ \sum_{\s \subset \dE} h_{\s}\norm{\edgeresid}_{0,\s}^2,\\
		\dataest^{\E} &:=h_{\E}^2 \norm{\eldataresid}^2_{0,\E} + h_{\E}^2 \norm{\force - \forceh}_{0,\E}^2 + \sum_{\s \subset \dE} h_{\s}\norm{\edgedataresid}_{0,\s}^2, \\
		\stabest^{\E} &= \StaE((\Po{\k} - \Id) \uh,(\Po{\k} - \Id)\uh ),
	\end{align*}
	and $\virtualosc^{\E}$ encompasses the \emph{virtual inconsistency} terms, defined as the sum of
	\begin{align*}
		\virtualosc_1^{\E} &= \norm{(\Po{\k-1} - \Id) (\diff \Po{\k-1} \nabla \uh)}_{0,\E}^2, \\
		\virtualosc_2^{\E} &= h_{\E}^2 \norm{(\Po{\k} - \Id) (\conv \cdot \Po{\k-1} \nabla \uh)}_{0,\E}^2, \\
		\virtualosc_3^{\E} &= \norm{(\Po{\k-1} - \Id) (\conv \Po{\k} \uh)}_{0,\E}^2, \\
		\virtualosc_4^{\E} &= h_{\E}^2 \norm{(\Po{\k} - \Id) (\reacSym \Po{\k} \uh)}_{0,\E}^2.
	\end{align*}
\end{theorem}

\begin{proof}
	Let $\errClem \in \Vh$ be the interpolant of $\err$ satisfying the bounds of Theorem~\ref{thm:spaceApproximation}.
	Then, upon setting $\v=\err$ and $\chi=\errClem$ in~\eqref{eq:residualEquation}, coercivity yields
	\begin{align*}
		\triplenorm{\err}_{\D}^2 &= \sum_{\E \in \Th} \Big( \left( \elresid, \err - \errClem \right)_{\E} + \left( \eldataresid, \err - \errClem \right)_{\E} + (\force-\forceh, \errClem)_{\E}+ \nonPolyResid^{\E}(\uh, \err - \errClem)  \\
			&\qquad + \left( \AhE( \uh, \errClem) - \AE( \uh, \errClem) \right)  \Big) - \sum_{\s \in \Edges} \left(\left( \edgeresid, \err - \errClem \right)_{0,\s} + \left( \edgedataresid, \err - \errClem \right)_{0,\s} \right)\\
			&=: \sum_{\E \in \Th} \Big( I +II+III+IV+V\Big) - \sum_{\s \in \Edges} \Big(VI+VII\Big).
	\end{align*}

	For $I$ and $II$, we use the Cauchy-Schwarz inequality and the bounds of Theorem~\ref{thm:spaceApproximation} to find that
	\begin{align*}
		I= \left(\elresid, \err - \errClem \right)_{\E}
			&\leq \Cclem h_{\E}\norm{\elresid}_{0,\E} \norm{\err}_{1,\E}, \\
				II= \left(\eldataresid, \err - \errClem \right)_{\E}
			&\leq \Cclem h_{\E}\norm{\eldataresid}_{0,\E} \norm{\err}_{1,\E}.
	\end{align*}
		For $III$, we use the properties of the $L^2$-projection to find that
	\begin{align*}
		III &= (\force - \forceh, \errClem - \Po{\k-1} \errClem)_{\E}
			\leq \Cproj h_{\E} \norm{\err}_{1,\E} \norm{\force - \forceh}_{0,\E}.
	\end{align*}

	Bounding the edge terms $VI$ and $VII$ requires the use of the scaled trace inequality $\norm{v}_{0,\s}^2 \leq \Ctrace ( h_{\E}^{-1} \norm{\v}_{0,\E}^2 + h_{\E} \norm{\nabla \v}_{0,\E}^2)$, for $\v \in H^1(\E)$,
	along with the bounds of Theorem~\ref{thm:spaceApproximation} and the mesh regularity assumption, yielding
	\begin{align*}
		VI &\leq \norm{\edgeresid}_{0,\s} \norm{\err - \errClem}_{0,\s}
			\leq \Ctrace\Cclem \meshReg^{-\frac{1}{2}} \norm{\err}_{1,\edgepatch} h_{\s}^{\frac{1}{2}} \norm{\edgeresid}_{0,\s},\\
			VII &\leq \norm{\edgedataresid}_{0,\s} \norm{\err - \errClem}_{0,\s}
			\leq \Ctrace\Cclem \meshReg^{-\frac{1}{2}} \norm{\err}_{1,\edgepatch} h_{\s}^{\frac{1}{2}} \norm{\edgedataresid}_{0,\s},
	\end{align*}
	where $\edgepatch = \E^+ \cup \E^-$ with $\E^+$ and $\E^-$ the elements meeting at the edge $\s$.

Noting that
    \begin{align}
    \label{eq:externalToInternalProjection}
    \norm{(\Id - \Po{\k-1}) \nabla \uh}_{0,\E}=\norm{(\Id - \Po{\k-1}) \nabla (\Id - \Po{\k})\uh }_{0,\E}\le
    \norm{ \nabla (\Id - \Po{\k})\uh }_{0,\E},
    \end{align}
we can bound $IV$ as
	\begin{align}
		IV \label{eq:nonPolyResidUpperBound}
			&\leq \Cclem\left(\norm{\diff (\Id - \Po{\k-1}) \nabla \uh }_{0,\E} + h_{\E}\norm{\conv \cdot (\Id - \Po{\k-1}) \nabla \uh }_{0,\E} + h_{\E}\norm{\reac(\Id - \Po{\k}) \uh }_{0,\E} \right) \norm{\err}_{1,\E} \notag\\
			&\leq \Cclem\left(\left(\elipUpperE + h_{\E}\norm{\conv}_{\infty,\E} \right)\norm{(\Id - \Po{\k-1}) \nabla \uh }_{0,\E} + h_{\E}\norm{\reac}_{\infty,\E} \norm{(\Id - \Po{\k}) \uh }_{0,\E} \right) \norm{\err}_{1,\E} \notag\\
			&\leq \Cclem\left(\left(\elipUpperE + h_{\E}\norm{\conv}_{\infty,\E} \right)\norm{\nabla(\Id - \Po{\k})  \uh }_{0,\E} + h_{\E}\norm{\reac}_{\infty,\E} \norm{(\Id - \Po{\k}) \uh }_{0,\E} \right) \norm{\err}_{1,\E}.
	\end{align}
Using now~\eqref{eq:bounddiff},
and introducing the mesh Pecl\'et number by $\pecletE := h_{\E}\norm{\conv}_{\infty,\E}/\elipLowerE$, we arrive to
\begin{align*}
		IV &\leq \sqrt{2}\Cclem \Cstab \norm{\err}_{1,\E} \bigg( \left(\frac{\elipUpperE}{\elipLowerE} + \pecletE \right)^2 \elipLowerE \StaE_{1}((\Id - \Po{\k})\uh,(\Id - \Po{\k})\uh)  \\
				&\qquad\qquad\qquad\qquad\qquad + \left( \frac{\norm{\reac}_{\infty,\E}}{\reacLowerE} \right)^2 h_{\E}^2 \reacLowerE \StaE_{0}((\Id - \Po{\k})\uh,(\Id - \Po{\k})\uh)  \bigg)^{1/2}.
	\end{align*}

	Focussing on $V$, we begin by observing the identity (due to the properties of the $L^2$-projection)
	\begin{align*}
		\aE(\uh, \errClem) - \ahE(\uh, \errClem)
			&= (\diff (\Id - \Po{\k-1}) \nabla \uh, \nabla \errClem)_{\E} + \left((\Id - \Po{\k})\uh, \reacSym \errClem \right)_{\E}  \\
				&\quad + ((\Id - \Po{\k-1}) \diff \Po{\k-1} \nabla \uh, \nabla \errClem)_{\E} + \left((\Id - \Po{\k}) \reacSym \Po{\k} \uh, (\Id - \Po{\k}) \errClem \right)_{\E}  \\
				&\quad - \StaE((\Id - \Po{\k})\uh, (\Id - \Po{\k})\errClem)\\
				&\leq \norm{\errClem}_{1,\E} \left( \elipUpperE \norm{ \nabla (\Id - \Po{\k})\uh }_{0,\E} + \norm{\reacSym}_{\infty,\E}\norm{(\Id - \Po{\k})\uh}_{0,\E} \right.
			\\
				&\quad\quad\quad\quad\quad \left.+ \norm{(\Id - \Po{\k-1}) \diff \Po{\k-1} \nabla \uh}_{0,\E} + \Cproj h_{\E} \norm{(\Id - \Po{\k}) \reacSym \Po{\k} \uh}_{0,\E} \right) \\
				&\quad\quad  - \StaE((\Id - \Po{\k})\uh, (\Id - \Po{\k})\errClem),
	\end{align*}
	with the last bound resulting from the Cauchy-Schwarz inequality and Theorem~\ref{thm:polynomialApproximation}.
	From Definition~\ref{def:admissibleStabilisingTerm} and Theorem~\ref{thm:polynomialApproximation}, we may bound the final term by
	\begin{align*}
		\StaE&((\Id - \Po{\k})\uh, (\Id - \Po{\k})\errClem)  \\
			&\leq  ( \stabConst_1^2 \StaE_1((\Id - \Po{\k})\uh, (\Id - \Po{\k})\uh))^{1/2} \Big( \Cstab \intE \diff  \nabla ((\Id - \Po{\k}) \errClem) \cdot  \nabla ((\Id - \Po{\k}) \errClem) \dx \Big)^{1/2}  \\
				&\qquad +  ( \stabConst_0^2 \StaE_0((\Id - \Po{\k})\uh, (\Id - \Po{\k})\uh))^{1/2} \Big(\Cstab \intE \reacSym ((\Id - \Po{\k}) \errClem)^2 \dx \Big)^{1/2} \\
			&\leq \Cstab^{1/2} \norm{\errClem}_{1,\E} \left( (\stabConst_1^2 \elipUpperE \StaE_1((\Id - \Po{\k})\uh, (\Id - \Po{\k})\uh))^{1/2}   \right. \\
				&\qquad\qquad\qquad\qquad \left. + (\stabConst_0^2 \norm{\reacSym}_{\infty,\E} h_{\E}^2 \StaE_0((\Id - \Po{\k})\uh, (\Id - \Po{\k})\uh))^{1/2} \right).
	\end{align*}
Combining the last two bounds, along with~\eqref{eq:bounddiff}, we conclude that
	\begin{align*}
		\aE(\uh, \errClem) - \ahE(\uh, \errClem)
			&\leq \norm{\errClem}_{1,\E} \Big(\norm{(\Id - \Po{\k-1}) \diff \Po{\k-1} \nabla \uh}_{0,\E} + \Cproj h_{\E} \norm{(\Id - \Po{\k}) \reacSym \Po{\k} \uh}_{0,\E}   \\
				&\qquad\qquad + \Cstab^{1/2}  \bigg( \Big( \Big( \frac{\elipUpperE}{\elipLowerE}  + \stabConst_1^2 \Big) \elipUpperE \StaE_1((\Id - \Po{\k})\uh, (\Id - \Po{\k})\uh) \Big)^{1/2}  \\
				&\qquad\qquad  + \Big( \Big( \frac{\norm{\reacSym}_{\infty,\E}}{\reacLowerE} + \stabConst_0^2 h_{\E}^2 \Big) \norm{\reacSym}_{\infty,\E} \StaE_0((\Id - \Po{\k})\uh, (\Id - \Po{\k})\uh) \Big)^{1/2} \bigg)\bigg).
	\end{align*}
	The skew-symmetric terms can be treated completely analogously, yielding
	\begin{align*}
		\bE(\uh, \errClem) - \bhE(\uh, \errClem) &\leq \frac{1}{2} \left( \left( \Cpf h_{\E} \norm{\conv}_{\infty,\E} + \Cproj h_{\E} \norm{\conv}_{1,\infty,\E} \right) \norm{\nabla (\Id - \Po{\k}) \uh}_{0,\E}  \right. \\
				&+ \left. \Cproj h_{\E}  \norm{(\Id - \Po{\k}) \conv \Po{\k-1} \nabla \uh}_{0,\E}+ \norm{(\Id - \Po{\k-1}) \conv \Po{\k} \uh}_{0,\E} \right) \norm{\nabla \errClem}_{0,\E},
	\end{align*}
	as $ \norm{(\Id - \Po{\k}) \uh}_{0,\E} \le \Cpf h_{\E} \norm{\nabla(\Id - \Po{\k}) \uh}_{0,\E}$ since $(\Id - \Po{\k}) \uh$ has zero average, and using~\eqref{eq:externalToInternalProjection}.
	The stability bound~\eqref{eq:bounddiff}, further implies
	\begin{align*}
		\bE(\uh, \errClem) - \bhE(\uh, \errClem) &\leq \frac{1}{2} \left( \Cstab^{1/2} ( \Cpf + \Cproj ) h_{\E} \norm{\conv}_{1,\infty,\E} \left(  \StaE_1((\Id - \Po{\k}) \uh, (\Id - \Po{\k}) \uh) \right)^{1/2}  \right. \\
				&+ \left. \Cproj h_{\E}  \norm{(\Id - \Po{\k}) \conv \Po{\k-1} \nabla \uh}_{0,\E}+ \norm{(\Id - \Po{\k-1}) \conv \Po{\k} \uh}_{0,\E} \right) \norm{\nabla \errClem}_{0,\E}.
	\end{align*}
	 Combining the bounds for the symmetric and skew-symmetric terms above, we deduce
	\begin{align*}
		\AE( \uh, \errClem) - \AhE(\uh, \errClem)
			&\leq \norm{\errClem}_{1,\E} \bigg(\norm{(\Id - \Po{\k-1}) \diff \Po{\k-1} \nabla \uh}_{0,\E} + \frac{1}{2} \norm{(\Id - \Po{\k-1}) \conv \Po{\k} \uh}_{0,\E}  \\
				& \qquad+ \Cproj h_{\E} \Big( \frac{1}{2} \norm{(\Id - \Po{\k}) \conv \Po{\k-1} \nabla \uh}_{0,\E} + \norm{(\Id - \Po{\k}) \reacSym \Po{\k} \uh}_{0,\E} \Big)  \\
				&\qquad+ \Cstab^{1/2}  \Big(\tilde{s}_1 \StaE_1((\Id - \Po{\k})\uh, (\Id - \Po{\k})\uh) \Big)^{1/2}  \\
				&\qquad + \Big( \Big( \frac{\norm{\reacSym}_{\infty,\E}}{\reacLowerE} + \stabConst_0^2 h_{\E}^2 \Big) \norm{\reacSym}_{\infty,\E} \StaE_0((\Id - \Po{\k})\uh, (\Id - \Po{\k})\uh) \Big)^{1/2} \bigg),
	\end{align*}
	with
	\begin{align*}
	\tilde{s}_1&:=  \Big( \frac{\elipUpperE}{\elipLowerE} \Big)^2  + \stabConst_1^2 \frac{\elipUpperE}{\elipLowerE} + ( \Cpf + \Cproj )^2 h_{\E}^2 \norm{\conv}_{1,\infty,\E}^2 .
	\end{align*}

	The result then follows by combining the individual bounds above and using the equivalence of the energy- and $H^1$-norms.
\end{proof}

The terms of $\estimator^\E$ echo the standard element and edge residual terms while $\dataest^\E$ are data oscillation terms, familiar from the residual a posteriori analysis of finite element methods~\cite{Verfurth:1996,Ainsworth-Oden:2000,Verfurth:2005}.
In the present virtual context, however, these terms involve only the polynomial part of $\uh$, as they would not be computable for $\uh$ itself.
As a result, remainder terms also appear in the estimator, collected in $\virtualosc^\E$.
The term $\stabest^\E$, on the other hand, provides a computable estimate for the quality of the approximation $\Po{\k} \uh$ of $\uh$.

\begin{remark}
We note that the term $\virtualosc_3^{\E}$ does not vanish when the PDE coefficients are constant, as $\virtualosc_i^{\E}$, $i=1,2,4$ do. It is possible to circumvent this quite easily by modifying the skew-symmetric bilinear form $\bhE$ to use the degree $\k$ projection of the gradient.
The resulting method and the respective (modified) estimators are still computable in just the same manner as the current method (cf.~\cite{UnifiedVEM}), since the virtual element functions are polynomials of degree $\k$ on each edge.
\end{remark}

The estimator of Theorem~\ref{thm:reliability} is also an estimator for the error between $\u$ and the projection of $\uh$, and we have the following result.
\begin{corollary}[Bound for the projected solution]
\label{cor:projectedSolutionBound}
	Let $\eta^{\E}, \virtualosc^{\E}, \stabest^{\E}$ and $\dataest^{\E}$ be the terms of the estimator in Theorem~\ref{thm:reliability}.
	Then,
	\begin{equation*}
		\triplenorm{\u - \Po{\k} \uh}_{\D}^2 \leq C \sum_{\E \in \Th} (\eta^{\E} + \virtualosc^{\E} + \stabest^{\E} + \dataest^{\E}).
	\end{equation*}
\end{corollary}
\begin{proof}
	Using the triangle inequality and the definition of the stabilising term, we have
	\begin{align*}
		\triplenorm{\u - \Po{\k} \uh}_{\D}^2
			&\leq 2 \triplenorm{\u - \uh}_{\D}^2 + \sum_{\E \in \Th} \frac{2 \Cstab}{ \min\{ \stabConst_1, \stabConst_0 \}} \StaE(\uh - \Po{\k} \uh, \uh - \Po{\k} \uh).
	\end{align*}
	The result follows by Theorem~\ref{thm:reliability}.
\end{proof}

\subsection{Lower bound}
We now prove local lower bounds of the error in the energy norm by the a posteriori error estimate.
To this end, we make use of element and edge bubble functions satisfying the bounds of Lemma~\ref{lem:bubEl} and Lemma~\ref{lem:bubEdge} respectively.

\begin{theorem}[Local lower bound]
	\label{thm:localLowerBound}
	Let $\estimator^{\E}, \stabest^{\E}$ and $\dataest^{\E}$ be as in Theorem~\ref{thm:reliability}. Then,
	\begin{align*}
		\estimator^{\E} \leq C\sum_{\E' \in \elpatch} \left(\triplenorm{\u - \uh}_{\E'}^2 + \stabest^{\E'} + \dataest^{\E'} \right),
	\end{align*}
	where $\elpatch := \{\E' \in \Th : \mu_{\spacedim-1}(\dE' \cap \dE) \neq 0\}$ is the patch consisting of the element $\E$ and its neighbours, and $\mu_{\spacedim-1}$ denotes the $(\spacedim-1)$-dimensional measure.
	The constant $C$ depends on $\Cstab, \CequivUp, \Cbub$, $\meshReg$, $\D$ and the PDE coefficients, but is independent of $h$, $\u$ and $\uh$.
	\end{theorem}

\begin{proof}
	First observe that $\elresid \in \PE{\k+q}$ for some $q\in \mathbb{N}\cup\{0\}$ representing the degree of the polynomials used for the data approximations in~\eqref{eq:dataApproximation}.
	From~\eqref{eq:residualEquation} with $\chi=0$ and the fact that $\bubEl|_{\dE} = 0$, we deduce
	\begin{align*}
		\A(\err, \bubEl \elresid) &= ( \elresid, \bubEl \elresid )_\E + ( \eldataresid, \bubEl \elresid )_\E + \nonPolyResid^{\E}(\uh, \bubEl \elresid).
	\end{align*}
	Arguing as in~\eqref{eq:nonPolyResidUpperBound}, with $\bubEl \elresid$ in place of $\err-\errClem$, we find that
	\begin{align*}
		\nonPolyResid^{\E}(\uh, \bubEl \elresid) \leq C (\stabest^{\E})^{\frac{1}{2}}\norm{\bubEl \elresid}_{1, \E},
	\end{align*}
  	and consequently, using the properties of the interior bubble functions given in Lemma~\ref{lem:bubEl},
	\begin{align}
		C\norm{\elresid}^2_{0,\E} &\leq ( \elresid, \bubEl \elresid )_{\E} \notag\\
		&= \AE(\err, \bubEl \elresid)
		- ( \eldataresid, \bubEl \elresid )_{\E}
		- \nonPolyResid^{\E}(\uh, \bubEl \elresid) \notag \\
			&\leq C \left(\triplenorm{\err}_{\E}
			+ (\stabest^{\E})^{\frac{1}{2}}\right) \norm{\bubEl \elresid}_{1, \E}
			+\norm{\eldataresid}_{0,\E} \norm{\bubEl \elresid}_{0,\E}. \notag
	\end{align}
	Using Lemma~\ref{lem:bubEl} again this becomes
	\begin{align*}
		C \norm{\elresid}^2_{0,\E} &\leq h_{\E}^{-1}
		 \left( \triplenorm{\err}_{\E}
		+ (\stabest^{\E})^{\frac{1}{2}} \right) \norm{\elresid}_{0,\E}
		+ \norm{\eldataresid}_{0,\E} \norm{\elresid}_{0,\E},
	\end{align*}
	and therefore we arrive at
	\begin{align*}
		Ch_{\E}^2 \norm{\elresid}_{0,\E}^2 &\leq
		\triplenorm{\err}_{\E}^2
		+ \stabest^{\E}
		+ h_{\E}^2\norm{\eldataresid}_{0,\E}^2.
	\end{align*}

For the face residual, we start by extending $\edgeresid$ into $\edgepatch$ through a constant prolongation in the direction normal to the face $s$, yielding $\edgeresid \in \P{\edgepatch}{\k} \subset V_h^{\edgepatch} := V_h^{\E^+} \cup V_h^{\E^-}$ with $\E^+ \cap \E^- = \edge$.
	Then,~\eqref{eq:residualEquation} gives
	\begin{align*}
		\A(\err, \bubEdge \edgeresid ) =& \sum_{\E'\in\edgepatch} \left[
		( \elresidd, \bubEdge \edgeresid )_{\E'} + ( \eldataresidd, \bubEdge \edgeresid )_{\E'} + \nonPolyResid^{\E}(\uh, \bubEdge \edgeresid) \right]
		- \left( \edgeresid, \bubEdge \edgeresid \right)_{0,\s} - \left( \edgedataresid, \bubEdge \edgeresid \right)_{0,\s}.
	\end{align*}
	Arguing as before and using Lemma~\ref{lem:bubEdge}, we find that
	\begin{align*}
		C\norm{\edgeresid}_{0,\edge}^2 &\leq \norm{\edgedataresid}_{0,\edge} \norm{\bubEdge \edgeresid}_{0,\edge} + \sum_{\E' \in \edgepatch} \Big[ \left( \triplenorm{\err}_{\E'}
			 + (\stabest^{\E'})^{\frac{1}{2}} \right) \norm{\bubEdge \edgeresid}_{1,\E'} +  \left( \norm{\elresid}_{0,\E'} + \norm{\eldataresid}_{0,\E'} \right) \norm{\bubEdge \edgeresid}_{0,\E'}  \Big].
	\end{align*}
	Applying Lemma~\ref{lem:bubEdge} again, using the bound for the element residual, and multiplying by $h_s^{1/2}$, we obtain
	\begin{align*}
		C h_s\norm{\edgeresid}_{0,\edge}^2
					&\leq h_s\norm{\edgedataresid}_{0,\edge}^2 + \sum_{\E'\in\edgepatch} \left[ (h_s/h_{\E'})^{1/2} \left( \triplenorm{\err}_{\E'}^2 + \stabest^{\E'} \right) + (h_s h_{\E'})^{1/2}  \norm{\eldataresidd}_{0,\E'}^2 \right] .
	\end{align*}
	Using Assumption~\ref{ass:meshReg} and putting these bounds together completes the proof.
\end{proof}

This local lower bound then immediately provides a corresponding
global lower bound, by simply summing the local estimates over the
whole of $\Th$.
Furthermore, Theorem~\ref{thm:localLowerBound} and triangle inequality also provide us with the following lower bound on the error between the solution $\u$ and the projected virtual element solution $\Po{\k} \uh$.
\begin{corollary}[Lower bound for the projected solution]
	Let $\estimator^{\E}, \stabest^{\E}$ and $\dataest^{\E}$ be defined as in Theorem~\ref{thm:reliability}. Then,
	\begin{align*}
		\estimator^{\E} \leq C\sum_{\E' \in \elpatch} \left(\triplenorm{\u - \Po{\k} \uh}_{\E'}^2 + \stabest^{\E'} + \dataest^{\E'} \right),
	\end{align*}
	where $\elpatch := \{\E' \in \Th : \mu_{\spacedim-1}(\dE' \cap \dE) \neq 0\}$ is the patch consisting of the element $\E$ and its neighbours, and $\mu_{\spacedim-1}$ denotes the $(\spacedim-1)$-dimensional measure.
\end{corollary}

In addition to a lower bound for the residual part of the estimator,
$\estimator^\E$, we have the following control on the virtual inconsistency terms $\virtualosc^E$ indicating that these are also of optimal order up to data oscillation.

\begin{lemma}[Lower bound for the inconsistency terms]\label{lemma_lower_bound_incon}
 We have
  \begin{equation*}
    \begin{split}
      \virtualosc^{\E}
      &\leq
      C \bigg(
      \triplenorm{\u - \uh}^2_{\E}
      +
      \stabest^{\E}      +
      \norm{(\Po{\k -1} - \Id)\diff \nabla u}^2_\E
      +
      h_E^2 \norm{(\Po{\k} - \Id)\conv \cdot \nabla u}^2_\E
      \\
      &\qquad\qquad
      +
      \norm{(\Po{\k-1} - \Id)\conv u}^2_\E
      +
      h_E^2 \norm{(\Po{\k} - \Id)\reacSym  u}^2_\E
      \bigg).
  \end{split}
  \end{equation*}
\end{lemma}
\begin{proof}
  We have, respectively,
  \begin{equation*}
    \begin{split}
      \virtualosc_1^E
      &\leq
      2
      \bigg(
      \norm{\Po{\k-1} (\diff (\Po{\k-1} \nabla \uh - \nabla u))}_{0,\E}^2
      +
      \norm{\Po{\k-1} (\diff \nabla u) - \diff \nabla u}_{0,\E}^2
      +
      \norm{\diff  \nabla u - \diff \Po{\k-1} \nabla u_h}_{0,\E}^2
      \bigg)
      \\
      &\leq
      2\norm{\Po{\k-1} (\diff \nabla u) - \diff \nabla u}_{0,\E}^2
      +
      4\norm{\diff ( \nabla u -  \Po{\k-1} \nabla u_h)}_{0,\E}^2\\
      &\leq
      2\norm{\Po{\k-1} (\diff \nabla u) - \diff \nabla u}_{0,\E}^2
      +
      8{\elipUpper}^2 \norm{\nabla u - \nabla \uh}_{0,\E}^2
      +
      8{\elipUpper}^2 \norm{\nabla \uh - \Po{\k-1} \nabla \uh}_{0,\E}^2,
    \end{split}
  \end{equation*}
  using the stability of the $\LTWO$ projection operator and~\eqref{eq:diffEllipticity}.
  Using~\eqref{eq:externalToInternalProjection} and~\eqref{eq:bounddiff} the
  final term can be controlled by the stabilising term, resulting in the required bound.
  A completely analogous argument can be applied to each of the remaining
  terms of $\virtualosc^{\E}$.
\end{proof}

\section{Numerical Results}
\label{sec:numerics}

\pgfplotsset{width=11cm}

\input{figures/plot_data}
\pgfplotstableread{
Cells1	nDofs1	val1	Rate1	Cells2	nDofs2	val2	Rate2	Cells3	nDofs3	val3	Rate3	Cells4	nDofs4	val4	Rate4
25		96	0.5970825683600859	0.000	25		241	0.22803898240329123	0.000	25		411	0.06071527467766062	0.000	25		606	0.014027429025538742	0.000
100		341	0.2980074283745808	-0.548	100		881	0.06414104480749161	-0.979	100		1521	0.008914257507766936	-1.466	100		2261	0.0009081701459616874	-2.079
400		1281	0.14869039962568673	-0.525	400		3361	0.016719356076655884	-1.004	400		5841	0.0011678487533517995	-1.511	400		8721	5.7274609608352595e-5	-2.047
1600		4961	0.07425016473262251	-0.513	1600		13121	0.004247146533771301	-1.006	1600		22881	0.00014787039819526821	-1.514	1600		34241	3.5882645282711764e-6	-2.025
6400		19521	0.03709807265650263	-0.507	6400		51841	0.0010690061703284312	-1.004	6400		90561	1.855516845182898e-5	-1.509	6400		135681	2.2441751525824653e-7	-2.013
25600		77441	0.018541748225475073	-0.503	25600		206081	0.0002680797907016505	-1.002	25600		360321	2.3224601497147242e-6	-1.505	25600		540161	1.4211831839495209e-8	-1.997
}{\UniformHOneErrorPlotData}

\pgfplotstableread{
Cells1	nDofs1	val1	Rate1	Cells2	nDofs2	val2	Rate2	Cells3	nDofs3	val3	Rate3	Cells4	nDofs4	val4	Rate4
25		96	3.307487067977175	0.000	25		241	0.7976643983434503	0.000	25		411	0.13706555672384996	0.000	25		606	0.018150049420494654	0.000
100		341	1.6673579024409115	-0.540	100		881	0.20187842255684177	-1.060	100		1521	0.017330767640168627	-1.580	100		2261	0.0011517387955884142	-2.094
400		1281	0.836232259242347	-0.521	400		3361	0.050556816869279195	-1.034	400		5841	0.002181246127144656	-1.540	400		8721	7.29447150250526e-5	-2.044
1600		4961	0.4186545281664705	-0.511	1600		13121	0.012639508273383452	-1.018	1600		22881	0.00027335496872372294	-1.521	1600		34241	4.5933240896582405e-6	-2.022
6400		19521	0.20944927940058664	-0.506	6400		51841	0.00315937949922911	-1.009	6400		90561	3.4198969883811136e-5	-1.511	6400		135681	2.8811178052558463e-7	-2.011
25600		77441	0.10475354153104582	-0.503	25600		206081	0.0007897536606941648	-1.005	25600		360321	4.27618691078763e-6	-1.506	25600		540161	1.8216565867854837e-8	-1.998
}{\UniformEstimatorConvergencePlotData}

\input{figures/jumping_coeff_plot_data}

We present a series of numerical experiments aimed at
testing the practical behaviour of the estimator derived in Theorem
\ref{thm:reliability}. In addition, we propose an adaptive algorithm
based on the estimator which is applied to a variety of test problems.

The above analysis is valid by only requiring a set of abstract assumptions for the stabilisation forms $\StaE_{1}$, $\StaE_{0}$ and for the projector $\spaceProj{\k}$ (which in turn defines the space $\VhE$ through~\eqref{eq:localVEMSpace}), giving rise to a number of possibilities.
Here, we focus on a specific scheme by providing precise choices for $\spaceProj{\k}$, $\StaE_{1}$ and $\StaE_{0}$.
Define the bilinear form $\altStaE: \VhE \times \VhE \rightarrow \Re$ by
\begin{equation*}
	\altStaE(\vh,\wh) := \sum_{r = 1}^{\NE} \dof_r(\vh) \dof_r(\wh),
\end{equation*}
with $\dof_r(\wh)$ denoting the value of the $r^{\text{th}}$ local degree of freedom of $\wh$ with respect to an arbitrary but fixed ordering of the degrees of freedom on the element $\E$.
This bilinear form corresponds to the Euclidean inner product on the space $\Re^{\NE}$ consisting of vectors of degrees of freedom.
Following~\cite[Section 4.1]{UnifiedVEM}, we define $\spaceProj{\k}$ to be the orthogonal projection onto the polynomial space $\PE{\k}$ with respect to $\altStaE(\cdot, \cdot)$, and
we fix
\begin{equation}\label{eq:scaling}
\begin{aligned}
\StaE_{1}(\uh,\vh) := \overline{\diff}_\E h_{\E}^{\spacedim - 2} \altStaE(\uh,\vh), \qquad
\StaE_{0}(\uh,\vh) := \overline{\reacSym}_\E h_{\E}^{\spacedim} \altStaE(\uh,\vh),
\end{aligned}
\end{equation}
where $\overline{\diff}_\E$, and $\overline{\reacSym}_\E$ are some constant approximations of $\diff$, and $\reacSym$ over $\E$ (e.g., local averages), respectively, resulting in
\begin{equation*}
\StaE(\uh,\vh):=h^{\spacedim-2}_\E \left(\stabConst_0 \overline{\diff}_\E + h_{E}^2 \stabConst_1 \overline{\reacSym}_\E \right)
\,
\altStaE(\uh,\vh).
\end{equation*}
\begin{remark}Note that the internal degrees of freedom of $(\Id - \Po{\k}) \vh$
  are equal to zero, and hence the above stablilising term reduces to a term active
  only on the mesh skeleton.
\end{remark}

\subsection{Uniformly generated meshes}
\label{sec:ugm}
As a first test to verify the asymptotic behaviour of the estimator, we consider the test problem
\begin{equation}\label{example_smooth}
  \begin{split}
  -\Delta \u = \force \text{ in } \D,
  \qquad
  u = 0 \text{ on } \partial\D,
  \end{split}
\end{equation}
for $\D = (0,1)^2$, and fix $\force$ such
that the exact solution is given by
$
	u(x,y) = \sin(\pi x) \sin(\pi y),
$
on a uniformly generated sequence of meshes consisting of \emph{non-convex}
polygonal elements.
The first two meshes in the uniform sequence are shown in
Fig.~\ref{fig:num:uniform:meshes}.
Fig.~\ref{fig:num:uniform:errorPlots}
depicts the convergence history of the $H^1(\D)$-seminorm error and of the estimator on this sequence of meshes, indicating that both converge at the optimal rate for polynomial degrees $\k=1,2$ and $3$.
The effectivity of the estimator is defined by
\begin{align}
\text{effectivity}:=	\frac{
          \left(
            \sum_{\E \in \Th}
            \estimator^{\E}
            +
            \dataest^{\E}
            +
            \virtualosc^{\E}
            +
            \stabest^{\E}
          \right)^{\frac{1}{2}}}
        {\norm{\nabla(\u - \Po{\k}\uh)}_{0,\D}},
	\label{eq:num:effectivity}
\end{align}
with $\estimator^\E, \dataest^\E, \stabest^{\E}$ and $\virtualosc^\E$ as in Theorem~\ref{thm:reliability}. Asymptotically the effectivity becomes constant throughout the mesh sequence, tending to approximately 5.7 for $\k=1$, 3 for $\k=2$, and 1.84 for $\k=3$.

\begin{figure}[h!]
\subcaptionbox{The first mesh, with 25 elements. \label{fig:num:uniform:nonconvex5}}[.45\textwidth]{%
	\includegraphics[width=.3\textwidth]{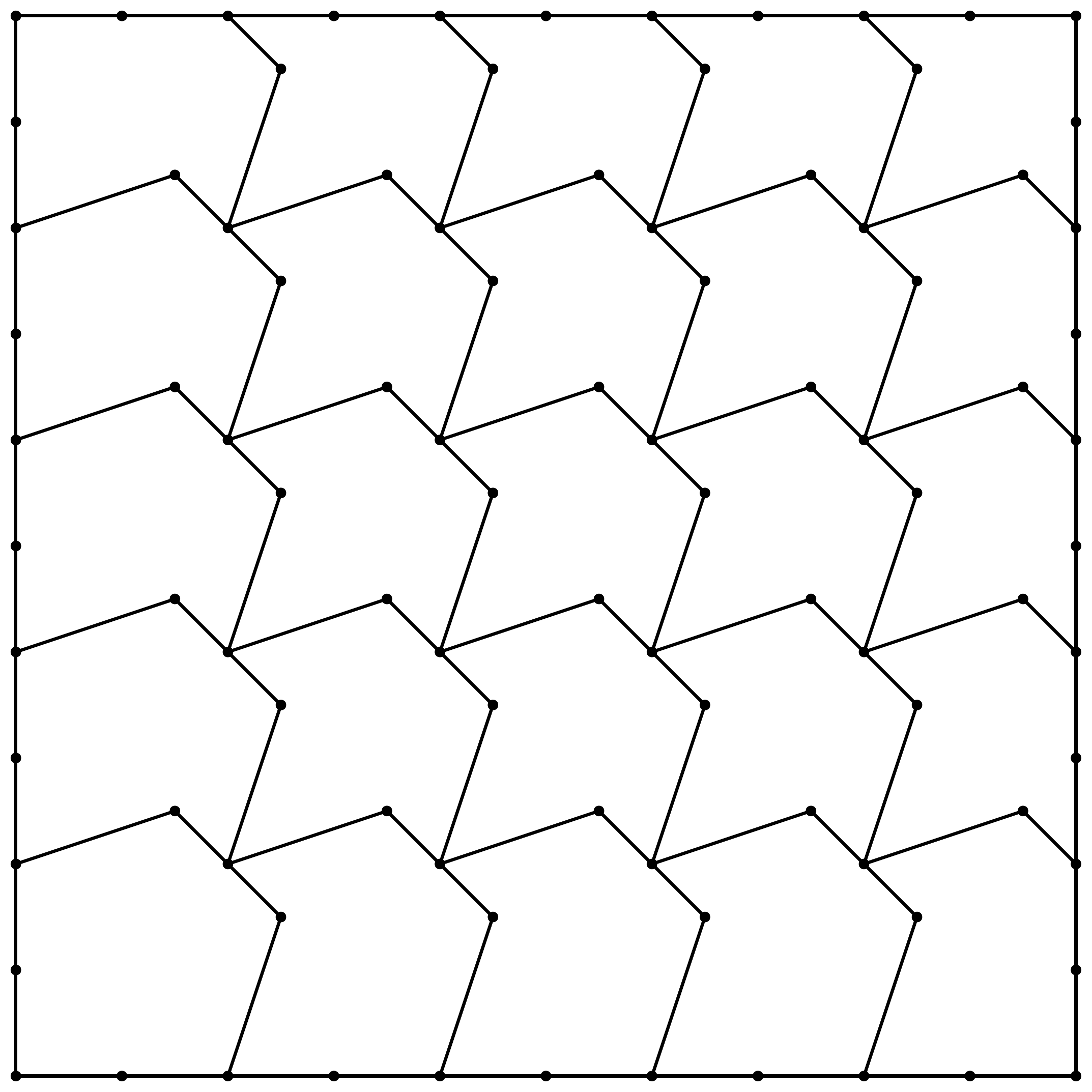}%
}\hfill%
\subcaptionbox{The second mesh with 100 elements. \label{fig:num:uniform:nonconvex10}}[.45\textwidth]{%
\includegraphics[width=.3\textwidth]{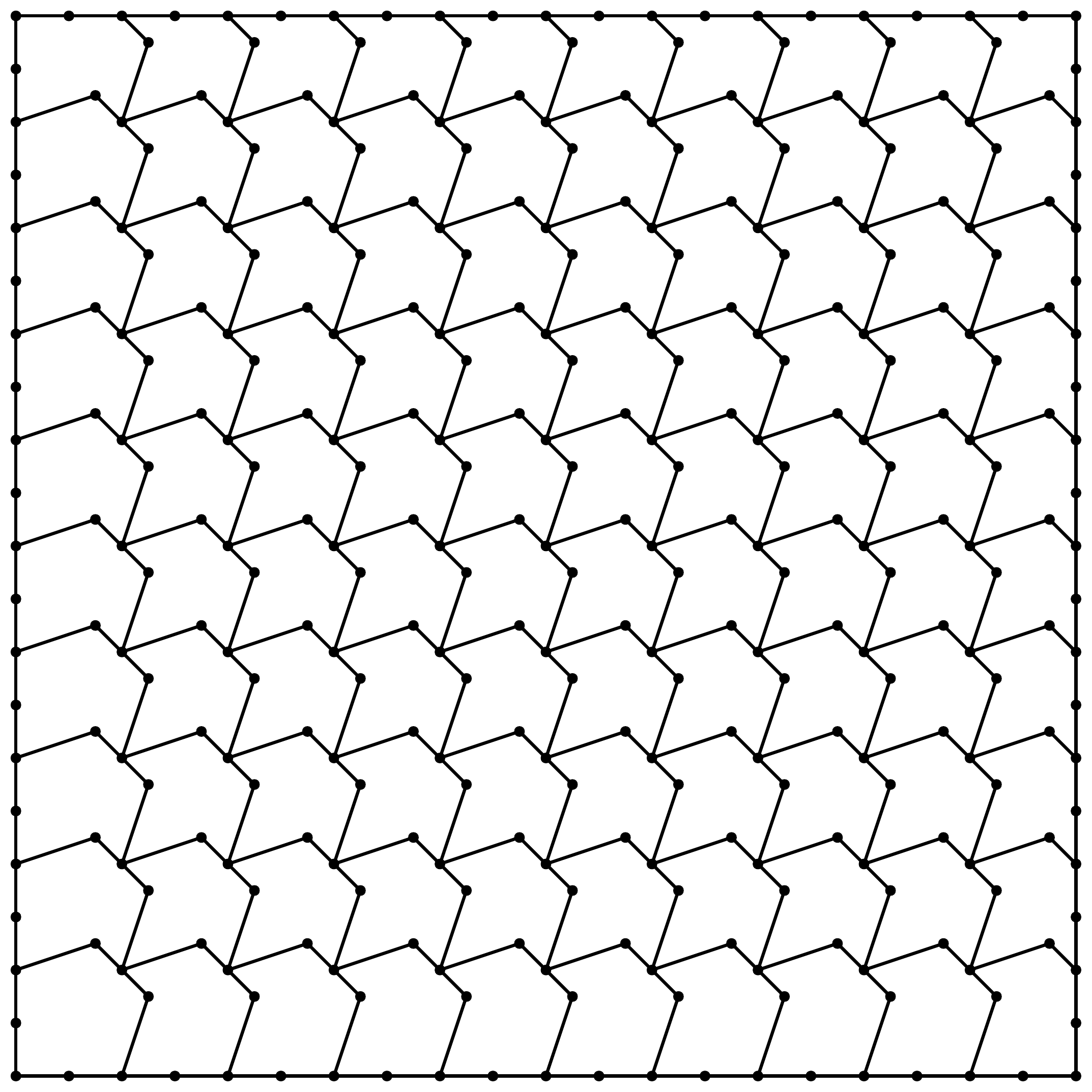}%
}
\caption{
  The first two non-convex, in general, meshes used in the uniform sequence described in \S\ref{sec:ugm}. Vertices are marked with a dot, and may separate coplanar
  edges.}
  \label{fig:num:uniform:meshes}
\end{figure}

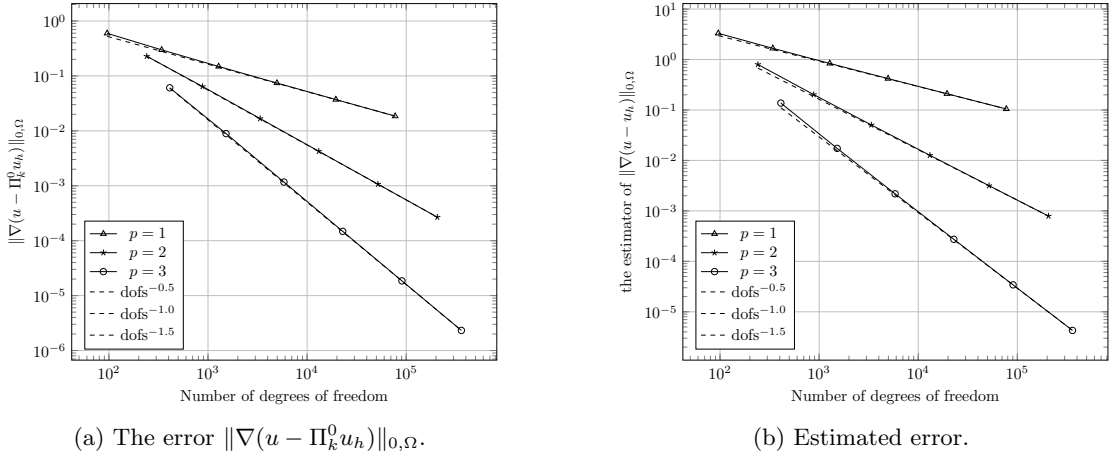
\begin{figure}[h!]
  \subcaptionbox{The error $\|\nabla(u - \Pi^0_k u_h ) \|_{0,\Omega}$.\label{fig:num:uniform:HOneError}}[.5\textwidth]{%
\begin{tikzpicture}[scale=0.6]
	\begin{loglogaxis}[xlabel={Number of degrees of freedom},ylabel={$\|\nabla(u - \Pi^0_k u_h ) \|_{0,\Omega}$},grid=major,legend entries={$\k = 1$, $\k = 2$, $\k = 3$, dofs$^{-0.5}$, dofs$^{-1.0}$, dofs$^{-1.5}$, }, legend pos=south west]
		\addplot[mark=triangle] table[x index = 1, y index = 2] from \UniformHOneErrorPlotData;
		\addplot[mark=star] table[x index = 5, y index = 6] from \UniformHOneErrorPlotData;
		\addplot[mark=o] table[x index = 9, y index = 10] from \UniformHOneErrorPlotData;
		\addplot[dashed] coordinates {(96, 0.5266238688096305) (77441.0, 0.018541748225475073)};
		\addplot[dashed] coordinates {(241, 0.22923714252110736) (206081.0, 0.0002680797907016505)};
		\addplot[dashed] coordinates {(411, 0.060286476132828566) (360321.0, 2.3224601497147242e-6)};
\end{loglogaxis}
\end{tikzpicture}}%
  \hfill
\subcaptionbox{Estimated error.}[.5\textwidth]{%
\begin{tikzpicture}[scale=0.6]
	\begin{loglogaxis}[xlabel={Number of degrees of freedom},ylabel={the estimator of $\|\nabla(u - u_h ) \|_{0,\Omega}$},grid=major,legend entries={$\k = 1$, $\k = 2$, $\k = 3$, dofs$^{-0.5}$, dofs$^{-1.0}$, dofs$^{-1.5}$}, legend pos=south west]
		\addplot[mark=triangle] table[x index = 1, y index = 2] from \UniformEstimatorConvergencePlotData;
		\addplot[mark=star] table[x index = 5, y index = 6] from \UniformEstimatorConvergencePlotData;
		\addplot[mark=o] table[x index = 9, y index = 10] from \UniformEstimatorConvergencePlotData;
		\addplot[dashed] coordinates {(96, 2.9752165028751603) (77441.0, 0.10475354153104582)};
		\addplot[dashed] coordinates {(241, 0.6753245815332544) (206081.0, 0.0007897536606941648)};
		\addplot[dashed] coordinates {(411, 0.11100136214108064) (360321.0, 4.27618691078763e-6)};
\end{loglogaxis}
\end{tikzpicture}}
  \caption{\label{fig:num:uniform:errorPlots}
    Convergence history of the $H^1(\D)$-seminorm error and estimator for the example~\eqref{example_smooth} on the meshes shown in Fig.~\ref{fig:num:uniform:meshes}.}
\end{figure}

\subsection{Adaptive refinement}
We shall use a typical adaptive algorithm for elliptic problems reading: \emph{solve$\to$estimate$\to$mark$\to$refine}. In this context, given a polygonal subdivision of $\Omega$, one \emph{solves} the VEM problem, \emph{estimates} the error using the a posteriori error bound (Theorem~\ref{thm:reliability}), \emph{marks} a subset of elements for refinement and, subsequently, \emph{refines}. The  D\"orfler/bulk marking strategy is used below for the \emph{mark} step,
marking the subset of mesh elements $\mathcal{M} \subset \Th$ with the largest estimated errors such that
\begin{equation}
	\Big( \sum_{\E \in \mathcal{M}} \estimator^\E + \dataest^\E + \virtualosc^{\E} + \stabest^{\E} \Big)^{\frac{1}{2}} \leq \theta \Big( \sum_{\E \in \Th} \estimator^\E + \dataest^\E + \virtualosc^{\E} + \stabest^{\E} \Big)^{\frac{1}{2}},
	\label{eq:dorflermarking}
\end{equation}
for some $\theta \in (0,1)$. Here, we pick $\theta = 0.4$.

To refine a polygonal element we divide elements by connecting the midpoint of each planar element face to its barycentre; see  Fig.~\ref{fig:num:polygonRefinement} for an illustration for a hexagon. Note that this strategy simply reduces to the standard refinement strategy for a square element.
\begin{figure}[h!]
  \centering
  \begin{tikzpicture}[rotate=-45, scale=0.7]
    \node[shape=coordinate] (barycentre) at (1,3.3) {};
		\node[circle, fill, scale=0.4, label=right:{$P$}, rotate=-45] (v1) at (0.5,0.5) {};
		\node[circle, fill, scale=0.4, label=right:{$Q$}, rotate=-45] (v2) at (1.75,1.35) {};
		\node[circle, fill, scale=0.4, label=right:{$R$}, rotate=-45] (v3) at (3,2.2) {};
		\node[circle, fill, scale=0.4] (v4) at (4,5) {};
		\node[circle, fill, scale=0.4] (v5) at (1,6) {};
		\node[circle, fill, scale=0.4] (v6) at (-1,4) {};
		\node[circle, fill, scale=0.4] (v7) at (-1.3,2) {};
		\node[shape=coordinate] (o1) at (0.5,0) {};
		\node[shape=coordinate] (o2) at (1.95,0.95) {};
		\node[shape=coordinate] (o3) at (3.4,1.9) {};
		\node[shape=coordinate] (o4) at (4.4,5.2) {};
		\node[shape=coordinate] (o5) at (0.9,6.4) {};
		\node[shape=coordinate] (o6) at (-1.4,4.2) {};
		\node[shape=coordinate] (o7) at (-1.7,1.8) {};
		%
		\draw (v1) -- (v2) -- (v3) -- (v4) -- (v5) -- (v6) -- (v7) -- (v1);
		%
		\draw (v1) -- (o1);
		\draw (v2) -- (o2);
		\draw (v3) -- (o3);
		\draw (v4) -- (o4);
		\draw (v5) -- (o5);
		\draw (v6) -- (o6);
		\draw (v7) -- (o7);
		%
		\draw[dashed] (v2) -- (barycentre);
		\draw[dashed] ($0.5*(v3)+0.5*(v4)$) -- (barycentre);
		\draw[dashed] ($0.5*(v4)+0.5*(v5)$) -- (barycentre);
		\draw[dashed] ($0.5*(v5)+0.5*(v6)$) -- (barycentre);
		\draw[dashed] ($0.5*(v6)+0.5*(v7)$) -- (barycentre);
		\draw[dashed] ($0.5*(v7)+0.5*(v1)$) -- (barycentre);
	\end{tikzpicture}
	\caption{An illustration of the refinement strategy used for polygonal elements.
	Edges and vertices of the original element are shown by solid lines and points.
	To refine the element, the midpoint of each \emph{planar face} of its boundary is connected to its barycentre.
	Note that the two edges $PQ$ and $QR$ are not bisected as the refinement treats $PR$ as a single planar face, adding only a new edge from $Q$ to the barycentre.
	Consequently, the result of refining two neighbours in the mesh is independent of the order in which they are refined.}
	\label{fig:num:polygonRefinement}
\end{figure}
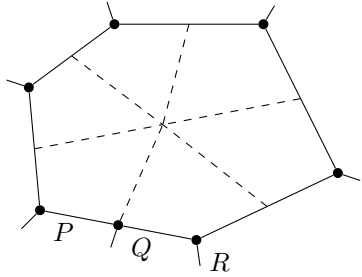
By refining in this fashion, hanging nodes may be introduced. Nevertheless, this is trivially accounted for in the VEM setting as the method is able to handle polygonal elements with an arbitrary number of faces.
This is a flexibility which we take advantage of in these examples by imposing no restriction on the number of hanging nodes allowed on each face.
In this extreme mesh flexibility, more exotic refinement strategies are certainly possible, but we leave the development of these for future work.

\begin{remark}[On the mesh assumptions]
	By imposing no restriction on the number of hanging nodes per face, we are at risk of violating Assumption~\ref{ass:meshReg} by producing meshes which contain very small faces.
	However, this requirement does not seem to be necessary for the virtual element method to remain accurate and stable in practice.
	This is demonstrated in Section~\ref{sec:num:kellogg}, where the effect of limiting the number of hanging nodes allowed per edge is also studied, and the results in either case are found to be very similar. 
\end{remark}

We consider the general convection-reaction-diffusion problem
\begin{equation*}
	-\nabla \cdot (\diff \nabla \u) + \conv \cdot \nabla \u + \reac \u = \force,
\end{equation*}
with coefficients
\begin{align*}
	\diff &= \begin{bmatrix} 1 & 0 \\ 0 & 1 \end{bmatrix}, \quad
	\conv = \begin{bmatrix}\cos(x) \exp(y) \\ \exp(x) \sin(y) \end{bmatrix}, \quad
	\reac = \sin(2\pi x) \sin(2\pi y),
\end{align*}
and forcing function $\force$ chosen in accordance with two different benchmark solutions:
\begin{enumerate}[Problem 1:]
\item posed over an L-shaped domain contained within $[-1,1]\times[-1,1]$ (depicted in Fig.~\ref{fig:num:LDomainGaussian:initialMesh}) and exhibiting low regularity at the reentrant corner located at the origin, along with a sharp Gaussian at the point $(0.5,0.5)$ which initially is not resolved by the mesh.
  This problem has the solution
  \begin{equation}
    \u(x, y) = r^{2/3} \sin(2\theta/3) + \exp(-(1000(x-0.5)^2+1000(y-0.5)^2)),
    \label{eq:num:LDomainGaussian}
  \end{equation}
  where $(r,\theta)$ are the usual polar coordinates centred around the point $(x,y) = (0,0)$, depicted in Fig.~\ref{fig:num:LDomainGaussian:solution}; \label{enum:num:LDomainGaussianProblem}
\item posed over $\D = (0,1)^2$ with a sharp layer in the interior of the domain and solution
  \begin{equation}
    u(x, y) = 16x(1-x)y(1-y)\arctan(25x-100y+50),
    \label{eq:num:InternalLayer}
  \end{equation}
  depicted in Fig.~\ref{fig:num:InternalLayer:solution}.
  \label{enum:num:InternalLayerProblem}
\end{enumerate}

The behaviour of the error and estimator under adaptive refinement for Problem~\ref{enum:num:LDomainGaussianProblem} and a representative set of the meshes obtained are shown in Fig.~\ref{fig:num:LDomainGaussian:errorPlots} and~\ref{fig:num:LDomainGaussianMeshes}, respectively. The same results are shown for Problem~\ref{enum:num:InternalLayerProblem} in Fig.~\ref{fig:num:InternalLayer:errorPlots} and~\ref{fig:num:InternalLayer:meshes}.
We first observe that, once the asymptotic regime is reached in each case, the error measured in the $H^1(\D)$-seminorm (shown in Fig.~\ref{fig:num:LDomainGaussian:HOneError} for Problem~\ref{enum:num:LDomainGaussianProblem} and in Fig.~\ref{fig:num:InternalLayer:HOneError} for Problem~\ref{enum:num:InternalLayerProblem} converges with the theoretical optimal rate of $N^{-\k/2}$, despite the low regularity of the true solution around the reentrant corner for Problem~\ref{enum:num:LDomainGaussianProblem}.

The initial rapid drop-off in error for Problem~\ref{enum:num:LDomainGaussianProblem} is explained by examining the magnitudes of the various components of the estimator for $\k=1$, given in Fig.~\ref{fig:num:LDomainGaussian:estimatorComponents}.
In particular, it is clear that the data oscillation term initially dominates the estimator and, comparing with the mesh after 28 iterations, shown in Fig.~\ref{fig:num:LDomainGaussian:mesh25}, it appears to be driving the refinement around the Gaussian centred at $(0.5, 0.5)$.
Once this is sufficiently resolved, the element and face residual terms begin to dominate, resulting in the expected refinement around the singularity at the reentrant corner. This is shown in Fig.~\ref{fig:num:LDomainGaussian:mesh39}, after 40 iterations.

The key difficulty of Problem~\ref{enum:num:InternalLayerProblem} is the presence of an interior sharp layer which is completely unresolved by the initial mesh. To test the resilience of the estimator in this challenging context, the initial mesh is chosen to consist of warped hexagons which are not aligned with the interior layer; see Fig.~\ref{fig:num:InternalLayer:meshOriginal} for an illustration.
As with Problem~\ref{enum:num:LDomainGaussianProblem}, the data oscillation terms initially dominate the estimator until the the mesh starts to resolve the layer.
After this point, the element and edge residuals become the dominant terms of the estimator.

For both problems, the effectivity plots in Fig.~\ref{fig:num:LDomainGaussian:efficiency} and~\ref{fig:num:InternalLayer:efficiency}, calculated as in~\eqref{eq:num:effectivity}, indicate a good level of agreement between the estimated and calculated error.

\subsection{Jumping diffusion coefficient}
\label{sec:num:kellogg}

We now consider the Kellogg problem~\cite{Kellogg:1974},
in which the diffusion coefficient $\diff$ is piecewise constant across the domain $\D = (0,1)^2$, such that
\begin{equation*}
	\diff(x,y) = \begin{cases} b \qquad \text{ for } (x - a)(y - a) \geq 0, \\ 1 \qquad \text{ otherwise,} \end{cases}
\end{equation*}
for fixed $0 < a < 1$ and $b > 0$, and no reaction or convection terms. This problem has weak solution
	$u(r, \theta) = r^{\alpha} g(\theta)$,
where $(r,\theta)$ denote the polar coordinates centred at the point $(a,a)$, and
\begin{align*}
	g(\theta) := \begin{cases}
					\cos((\frac{\pi}{2} - \sigma) \alpha) \cos((\theta - \frac{\pi}{4})\alpha) \qquad &\text{ for } 0 \leq \theta < \frac{\pi}{2}, \\
					\cos(\frac{\pi}{4}\alpha) \cos((\theta - \pi + \sigma) \alpha) \qquad &\text{ for } \frac{\pi}{2} \leq \theta < \pi, \\
					\cos(\sigma\alpha) \cos((\theta - \frac{5\pi}{4})\alpha) \qquad &\text{ for } \pi \leq \theta < \frac{3\pi}{2}, \\
					\cos(\frac{\pi}{4}\alpha) \cos((\theta - \frac{3\pi}{2} - \sigma)\alpha) \qquad &\text{ for } \frac{3\pi}{2} \leq \theta < 2\pi.
				\end{cases}
\end{align*}
The parameters $\sigma, \alpha$ and $b$ are required to satisfy a certain set of nonlinear relations~\cite{Kellogg:1974}, and following~\cite{Bonito-Devore-Nochetto:2013} we take the approximate values
$\sigma = −5.49778714378214$, $\alpha = 0.25$, $b = 25.27414236908818$.

The Kellogg problem is a common example used to test a posteriori estimators on a problem with pathological coefficients and a known weak solution.
Typically, this problem is studied in the case when $\kappa$ is piecewise constant with respect to the initial mesh, see, e.g.~\cite{Morin-Nochetto-Siebert:2002,Mekchay-Nochetto:2005,Chen-Shibin:2002,Bernardi-Verfurth:2000}. Recently, the case in which the diffusion jumps are not aligned with the initial mesh has been studied in~\cite{Bonito-Devore-Nochetto:2013} in the context of adaptive FEM.
To really test the applicability of our estimator, we consider both cases here on a variety of different meshes.

Whether the mesh is aligned with the problem or not is dictated by the parameter $a$.
We first consider $a=\frac{2}{5}$ on a square mesh, so the discontinuities of $\diff$ are matched by the initial mesh.
The behaviour of the error and estimator for this problem are shown in Fig.~\ref{fig:num:AlignedLimitedUnlimited:errorPlots}. Moreover, for this problem, we also compare the effect of limiting the mesh to have just one hanging node per edge, or allowing an unlimited number of hanging nodes to be produced.
In both cases, we use the D\"orfler strategy from~\eqref{eq:dorflermarking} with $\theta = 0.6$ to select the subset of elements to be refined.
For either a limited or unlimited number of hanging nodes per edge, the error under adaptive refinement eventually decays at the theoretical optimal rate of $N^{-1/2}$, where $N$ is the number of degrees of freedom.
It may also be seen that the $H^1$-seminorm error is slightly lower for the case of a limited number of hanging nodes, although the estimated error is approximately the same for both cases.
Consequently, the effectivity of the estimator is slightly better for the method with no limit on the number of hanging nodes.

Next, we consider $a = \frac{2 \sqrt{2}}{5}$.
In this case, it is not possible for the discontinuities of $\diff$ to align with any mesh in the sequence.
In the spirit of keeping the mesh fully unfitted from the discontinuities in $\diff$, we also test the method on a Voronoi mesh and a randomised quadrilateral mesh alongside a more standard square mesh.
For brevity, we only report here the results when an unlimited number of hanging nodes were allowed in the mesh, as limiting the number of hanging nodes leads to almost identical results in terms of convergence.
There are, however, differences in the final meshes obtained in each case: illustrations of the initial meshes and the final meshes for both limited and unlimited hanging nodes are given in Fig.~\ref{fig:num:Unaligned:mesh}. Fig.~\ref{fig:num:UnalignedUnlimited:errorPlots} shows the behaviour of the error and estimator under adaptive refinement on the three sequences of meshes: the error in the $H^1(\D)$-seminorm (Fig.~\ref{fig:num:UnalignedUnlimited:HOneError}) appears to reach the theoretical optimal convergence rate of $N^{-1/2}$ on the square and randomised quadrilateral meshes, and maintains a near-optimal rate of approximately $N^{-0.35}$ on the Voronoi mesh. These rates are also reflected by the convergence of the estimator, shown in Fig.~\ref{fig:num:UnalignedUnlimited:estimator}, resulting in good effectivities (Fig.~\ref{fig:num:UnalignedUnlimited:efficiency}) which remain roughly constant on the Voronoi and randomised quadrilateral meshes.

We note the sudden jump in the magnitude of the estimated error after $7$ iterations of adaptive refinement starting from the square mesh in Fig.~\ref{fig:num:Unaligned:mesh:initialSquares}.
Comparing with Figure~\ref{fig:num:UnalignedUnlimited:estimatorComponents}, which shows the relative magnitudes of the various terms comprising the estimator on the square mesh, it is apparent that this jump is caused by a jump in the value of the data oscillation term $\dataest$.
Noting that for $\k=1$ the coefficient approximation $\diffh$ is piecewise constant, we conclude that it is in fact only the edge data oscillation term which is non-zero and thus responsible for this effect.
Further investigation indicates that this jump occurs in situations such as that illustrated in Fig.~\ref{fig:num:kellogg:refiningSquare}, and is due to the fact that although the mesh cannot exactly align with the discontinuities of $\diff$, it is possible for it to get arbitrarily close.
\begin{figure}
	\centering
	\subcaptionbox{The lines along which $\diff$ is discontinuous pass close to the centre of the element. \label{fig:num:kellogg:refiningSquare:before}}[0.45\textwidth]{%
	\begin{tikzpicture}
		\node[shape=coordinate] (origin) at (0,0) {};
		\node[shape=coordinate] (isv1) at (-1,-1) {};
		\node[shape=coordinate] (isv2) at (1,-1) {};
		\node[shape=coordinate] (isv3) at (1,1) {};
		\node[shape=coordinate] (isv4) at (-1,1) {};
		\node[shape=coordinate] (ism12) at (0,-1) {};
		\node[shape=coordinate] (ism23) at (1,0) {};
		\node[shape=coordinate] (ism34) at (0,1) {};
		\node[shape=coordinate] (ism41) at (1,0) {};
		\coordinate (outerEdgeOffsetX) at (0.5,0);
		\coordinate (outerEdgeOffsetY) at (0,0.5);
		\node[shape=coordinate] (osv1) at ($(isv1) - (outerEdgeOffsetY)$) {};
		\node[shape=coordinate] (osv2) at ($(ism12) - (outerEdgeOffsetY)$) {};
		\node[shape=coordinate] (osv3) at ($(isv2) - (outerEdgeOffsetY)$) {};
		\node[shape=coordinate] (osv4) at ($(isv2) + (outerEdgeOffsetX)$) {};
		\node[shape=coordinate] (osv5) at ($(ism23) + (outerEdgeOffsetX)$) {};
		\node[shape=coordinate] (osv6) at ($(isv3) + (outerEdgeOffsetX)$) {};
		\node[shape=coordinate] (osv7) at ($(isv3) + (outerEdgeOffsetY)$) {};
		\node[shape=coordinate] (osv8) at ($(ism34) + (outerEdgeOffsetY)$) {};
		\node[shape=coordinate] (osv9) at ($(isv4) + (outerEdgeOffsetY)$) {};
		\node[shape=coordinate] (osv10) at ($(isv4) - (outerEdgeOffsetX)$) {};
		\node[shape=coordinate] (osv11) at ($(ism41) - (outerEdgeOffsetX)$) {};
		\node[shape=coordinate] (osv12) at ($(isv1) - (outerEdgeOffsetX)$) {};
		\coordinate (kappaoffsetx) at (0.1,0);
		\coordinate (kappaoffsety) at (0,0.1);
		\node[shape=coordinate] (kd1) at ($(0,-1.5) + (kappaoffsetx)$) {};
		\node[shape=coordinate] (kd2) at ($(0,1.5) + (kappaoffsetx)$) {};
		\node[shape=coordinate] (kd3) at ($(-1.5,0) + (kappaoffsety)$) {};
		\node[shape=coordinate] (kd4) at ($(1.5,0) + (kappaoffsety)$) {};
		\draw (isv1) -- (isv2) -- (isv3) -- (isv4) -- (isv1);
		\draw (osv1) -- (isv1) -- (osv12);
		\draw (osv3) -- (isv2) -- (osv4);
		\draw (osv6) -- (isv3) -- (osv7);
		\draw (osv9) -- (isv4) -- (osv10);
		\draw[dotted] (kd1) -- (kd2);
		\draw[dotted] (kd3) -- (kd4);
	\end{tikzpicture}}\hfill%
	\subcaptionbox{After refinement, the discontinuities of $\diff$ are suddenly very close to mesh edges. \label{fig:num:kellogg:refiningSquare:after}}[0.45\textwidth]{%
	\begin{tikzpicture}
		\node[shape=coordinate] (origin) at (0,0) {};
		\node[shape=coordinate] (isv1) at (-1,-1) {};
		\node[shape=coordinate] (isv2) at (1,-1) {};
		\node[shape=coordinate] (isv3) at (1,1) {};
		\node[shape=coordinate] (isv4) at (-1,1) {};
		\node[shape=coordinate] (ism12) at (0,-1) {};
		\node[shape=coordinate] (ism23) at (1,0) {};
		\node[shape=coordinate] (ism34) at (0,1) {};
		\node[shape=coordinate] (ism41) at (-1,0) {};
		\coordinate (outerEdgeOffsetX) at (0.5,0);
		\coordinate (outerEdgeOffsetY) at (0,0.5);
		\node[shape=coordinate] (osv1) at ($(isv1) - (outerEdgeOffsetY)$) {};
		\node[shape=coordinate] (osv2) at ($(ism12) - (outerEdgeOffsetY)$) {};
		\node[shape=coordinate] (osv3) at ($(isv2) - (outerEdgeOffsetY)$) {};
		\node[shape=coordinate] (osv4) at ($(isv2) + (outerEdgeOffsetX)$) {};
		\node[shape=coordinate] (osv5) at ($(ism23) + (outerEdgeOffsetX)$) {};
		\node[shape=coordinate] (osv6) at ($(isv3) + (outerEdgeOffsetX)$) {};
		\node[shape=coordinate] (osv7) at ($(isv3) + (outerEdgeOffsetY)$) {};
		\node[shape=coordinate] (osv8) at ($(ism34) + (outerEdgeOffsetY)$) {};
		\node[shape=coordinate] (osv9) at ($(isv4) + (outerEdgeOffsetY)$) {};
		\node[shape=coordinate] (osv10) at ($(isv4) - (outerEdgeOffsetX)$) {};
		\node[shape=coordinate] (osv11) at ($(ism41) - (outerEdgeOffsetX)$) {};
		\node[shape=coordinate] (osv12) at ($(isv1) - (outerEdgeOffsetX)$) {};
		\coordinate (kappaoffsetx) at (0.05,0);
		\coordinate (kappaoffsety) at (0,0.05);
		\node[shape=coordinate] (kd1) at ($(0,-1.5) + (kappaoffsetx)$) {};
		\node[shape=coordinate] (kd2) at ($(0,1.5) + (kappaoffsetx)$) {};
		\node[shape=coordinate] (kd3) at ($(-1.5,0) + (kappaoffsety)$) {};
		\node[shape=coordinate] (kd4) at ($(1.5,0) + (kappaoffsety)$) {};
		\draw (isv1) -- (isv2) -- (isv3) -- (isv4) -- (isv1);
		\draw (osv1) -- (isv1) -- (osv12);
		\draw (osv3) -- (isv2) -- (osv4);
		\draw (osv6) -- (isv3) -- (osv7);
		\draw (osv9) -- (isv4) -- (osv10);
		\draw (ism12) -- (ism34);
		\draw (ism23) -- (ism41);
		\draw (osv2) -- (ism12);
		\draw (osv5) -- (ism23);
		\draw (osv8) -- (ism34);
		\draw (osv11) -- (ism41);
		\draw[dotted] (kd1) -- (kd2);
		\draw[dotted] (kd3) -- (kd4);
	\end{tikzpicture}}
	\caption{The element of the mesh containing the intersection of the lines along which $\diff$ is discontinuous, before and after refinement. Solid lines indicate edges in the mesh, dotted lines indicate the lines along which $\diff$ jumps.}
	\label{fig:num:kellogg:refiningSquare}
\end{figure}
This is a highly desirable trait from the point of view of generating a well-adapted mesh. Nonetheless, the standard (isotropic) refinement strategy used on squares produces a mesh with edges close to the diffusion discontinuity, such as the ones depicted in Fig.~\ref{fig:num:kellogg:refiningSquare:after}, only if the previous iteration contains elements as in Fig.~\ref{fig:num:kellogg:refiningSquare:before} with the lines of discontinuity of $\diff$ passing close to its centre.
This is problematic because the roughly equal distribution of the central element in Figure~\ref{fig:num:kellogg:refiningSquare:before} and its four neighbours among the different zones of $\diff$ mean that the approximation $\diffh$ will be very similar on each of the five elements, and thus the edge term of the data oscillation indicator will be very small.
However, once this parent is refined, each child is almost entirely in a single zone of $\diff$, so the approximations $\diffh$ will be very different on each of the children.
This will, then, cause the reported error to dramatically increase.
Moreover, since the discontinuities of $\diff$ lie along lines with irrational coordinates, it is clear that this situation could occur an arbitrary number of times in the refinement sequence, causing problems with the effectivity of the estimator.
Clearly the real culprit here is the symmetry of the situation and, consequently, a way to prevent such problems occurring is to use unstructured meshes.
This claim can be substantiated by the fact that the same difficulty does not occur with the randomised quadrilateral or Voronoi meshes.

\begin{figure}[p]
\subcaptionbox{The error $\|\nabla(u - \Pi^0_k u_h ) \|_{0,\Omega}$.\label{fig:num:LDomainGaussian:HOneError}}[.3\textwidth]{%
\begin{tikzpicture}[scale=0.5]
  \begin{loglogaxis}[xlabel={Number of degrees of freedom},grid=major,legend entries={$\k = 1$, $\k = 2$, $\k = 3$, dofs$^{-0.5}$, dofs$^{-1.0}$, dofs$^{-1.5}$, },]
\addplot[mark=triangle] table[x index = 1, y index = 2] from \LDomainGaussianHOneErrorPlotData;
\addplot[mark=star] table[x index = 5, y index = 6] from \LDomainGaussianHOneErrorPlotData;
\addplot[mark=o] table[x index = 9, y index = 10] from \LDomainGaussianHOneErrorPlotData;
\addplot[dashed] coordinates {(21, 1.338260275343705) (875.0, 0.2073223903718849)};
\addplot[dashed] coordinates {(65, 2.3493773532399134) (1837.0, 0.08312984646738944)};
\addplot[dashed] coordinates {(121, 6.67500879129396) (2008.0, 0.09873800198297408)};
\end{loglogaxis}
\end{tikzpicture}}%
\hfill%
\subcaptionbox{Estimated error.}[.3\textwidth]{%
\begin{tikzpicture}[scale=0.5]
	\begin{loglogaxis}[xlabel={Number of degrees of freedom},grid=major,legend entries={$\k = 1$, $\k = 2$, $\k = 3$, dofs$^{-0.5}$, dofs$^{-1.0}$, dofs$^{-1.5}$}]
		\addplot[mark=triangle] table[x index = 1, y index = 2] from \LDomainGaussianEstimatorConvergencePlotData;
		\addplot[mark=star] table[x index = 5, y index = 6] from \LDomainGaussianEstimatorConvergencePlotData;
		\addplot[mark=o] table[x index = 9, y index = 10] from \LDomainGaussianEstimatorConvergencePlotData;
		\addplot[dashed] coordinates {(21, 6.218465102783021) (875.0, 0.9633604712820257)};
		\addplot[dashed] coordinates {(65, 7.912198067205785) (1837.0, 0.2799634591009124)};
		\addplot[dashed] coordinates {(121, 13.292411507802909) (2008.0, 0.19662388393072483)};
\end{loglogaxis}
\end{tikzpicture}}%
\hfill%
\subcaptionbox{Effectivity of the estimator.\label{fig:num:LDomainGaussian:efficiency}}[.3\textwidth]{%
\begin{tikzpicture}[scale=0.5]
	\begin{semilogxaxis}[xlabel={Number of degrees of freedom},grid=major,ymin=-10,ymax=70,legend entries={$\k = 1$, $\k = 2$, $\k = 3$, },]
		\addplot[mark=triangle] table[x index = 1, y index = 2] from \LDomainGaussianEstimatorEfficiencyPlotData;
		\addplot[mark=star] table[x index = 5, y index = 6] from \LDomainGaussianEstimatorEfficiencyPlotData;
		\addplot[mark=o] table[x index = 9, y index = 10] from \LDomainGaussianEstimatorEfficiencyPlotData;
\end{semilogxaxis}
\end{tikzpicture}}

\subcaptionbox{The estimator components for $\k=1$.\label{fig:num:LDomainGaussian:estimatorComponents}}[.5\textwidth]{%
\begin{tikzpicture}[scale=0.6]
	\begin{loglogaxis}[xlabel={Number of degrees of freedom},grid=major,legend entries={$\left(\sum \estimator^{\E}\right)^{1/2}$, $\left(\sum \virtualosc^E\right)^{1/2}$, $\left(\sum \dataest^\E\right)^{1/2}$, $\left(\sum \stabest^\E\right)^{1/2}$},]
		\addplot[mark=triangle] table[x index = 1, y index = 2] from \LDomainGaussianEstimatorComponentsPlotData;
		\addplot[mark=star] table[x index = 5, y index = 6] from \LDomainGaussianEstimatorComponentsPlotData;
		\addplot[mark=o] table[x index = 9, y index = 10] from \LDomainGaussianEstimatorComponentsPlotData;
		\addplot[mark=10-pointed star] table[x index = 13, y index = 14] from \LDomainGaussianEstimatorComponentsPlotData;
\end{loglogaxis}
\end{tikzpicture}}%
\hfill%
\subcaptionbox{Adaptive approximation.\label{fig:num:LDomainGaussian:solution}}[.5\textwidth]{%
\includegraphics[width=.5\textwidth]{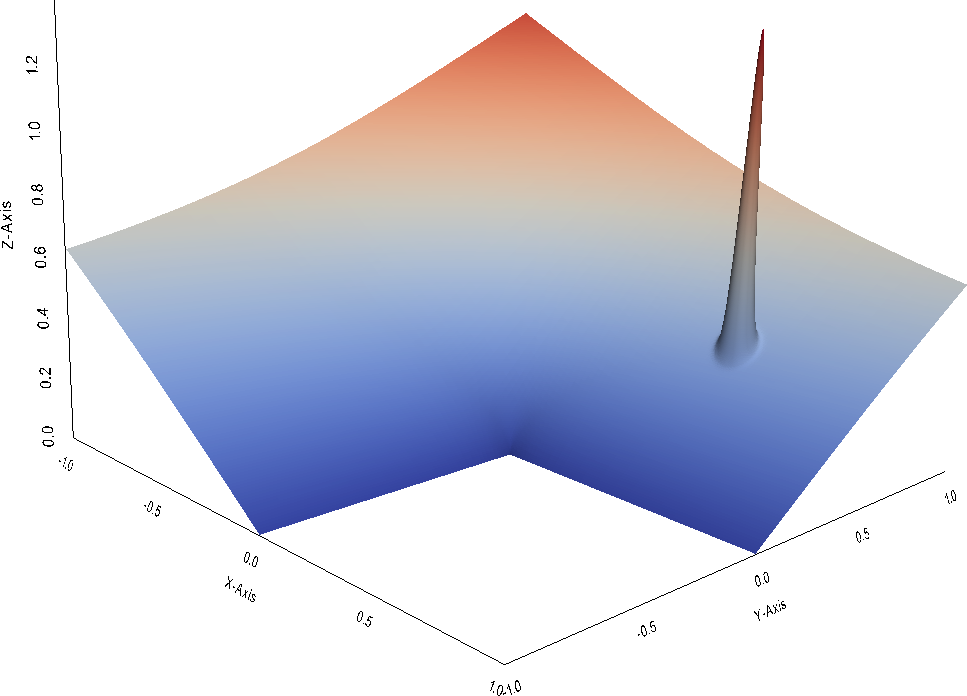}
}
\caption{The behaviour of the error and estimator when applied to Problem~\ref{enum:num:LDomainGaussianProblem} with solution~\eqref{eq:num:LDomainGaussian} under adaptive refinement, each plotted against the number of degrees of freedom.}
\label{fig:num:LDomainGaussian:errorPlots}
\end{figure}

\begin{figure}[p]
\subcaptionbox{The initial mesh.\label{fig:num:LDomainGaussian:initialMesh}}[.3\textwidth]{%
	\includegraphics[width=.3\textwidth]{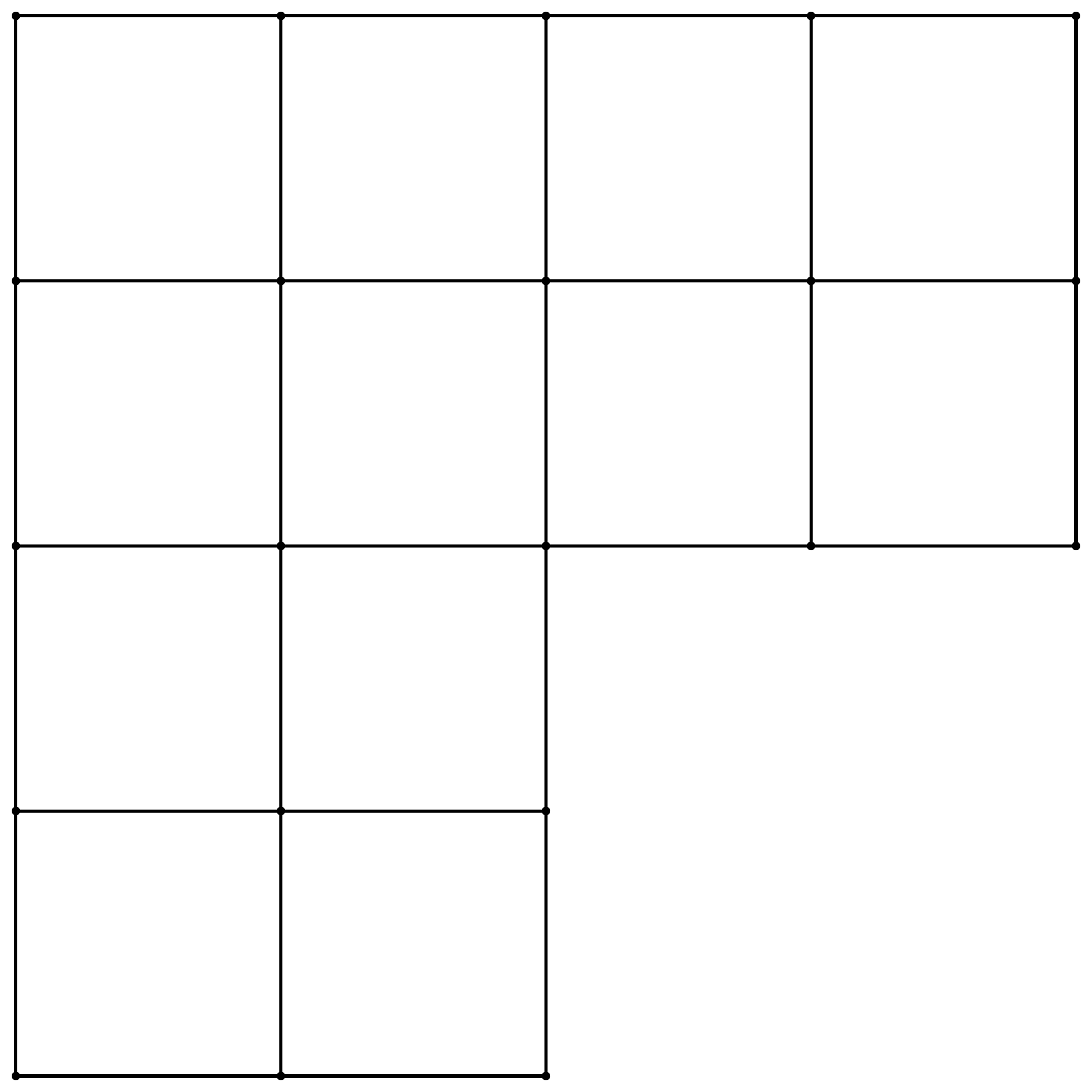}%
}\hfill%
\subcaptionbox{After 28 refinement steps.\label{fig:num:LDomainGaussian:mesh25}}[.3\textwidth]{%
\includegraphics[width=.3\textwidth]{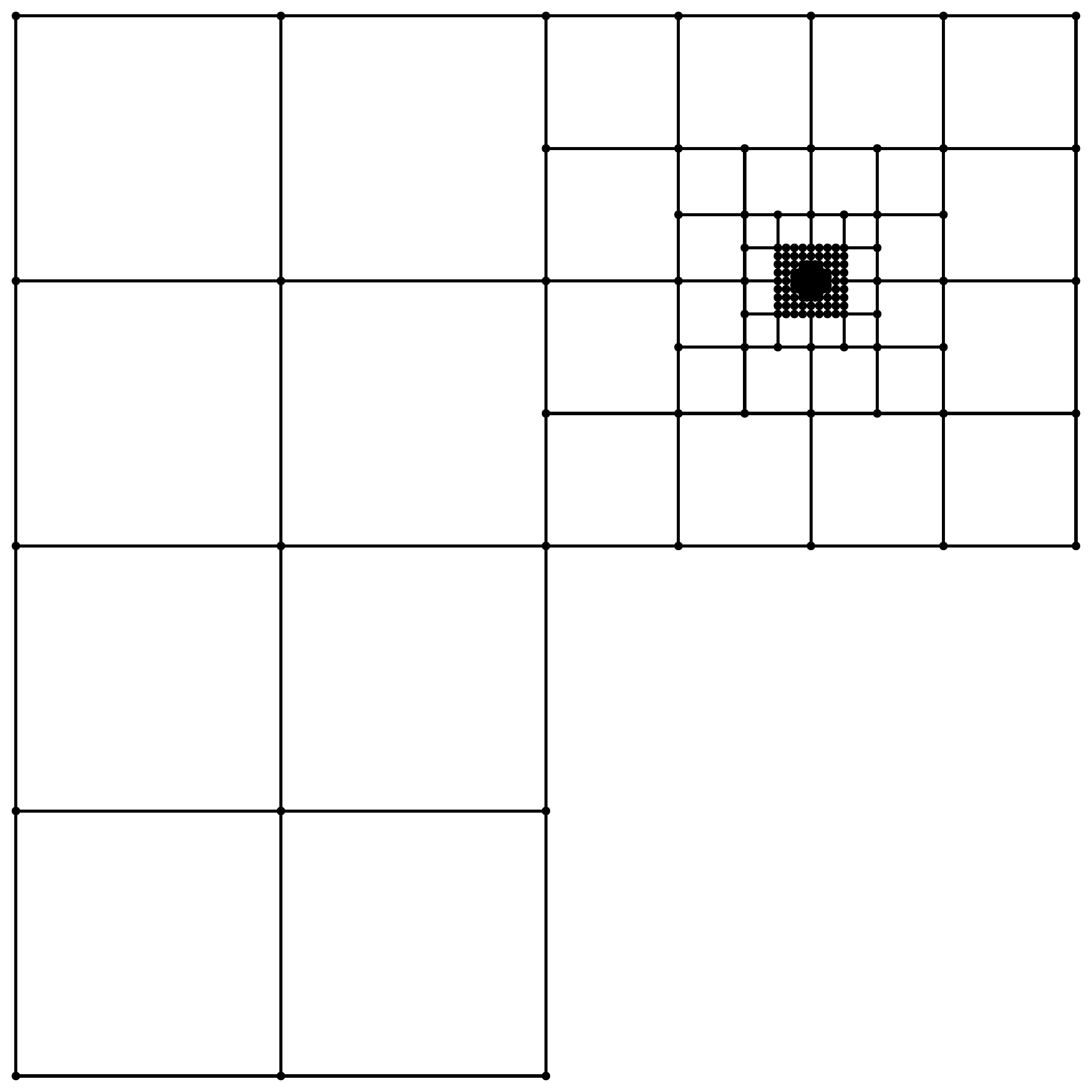}%
}%
\hfill%
\subcaptionbox{After 40 refinement steps.\label{fig:num:LDomainGaussian:mesh39}}[.3\textwidth]
{%
\includegraphics[width=.3\textwidth]{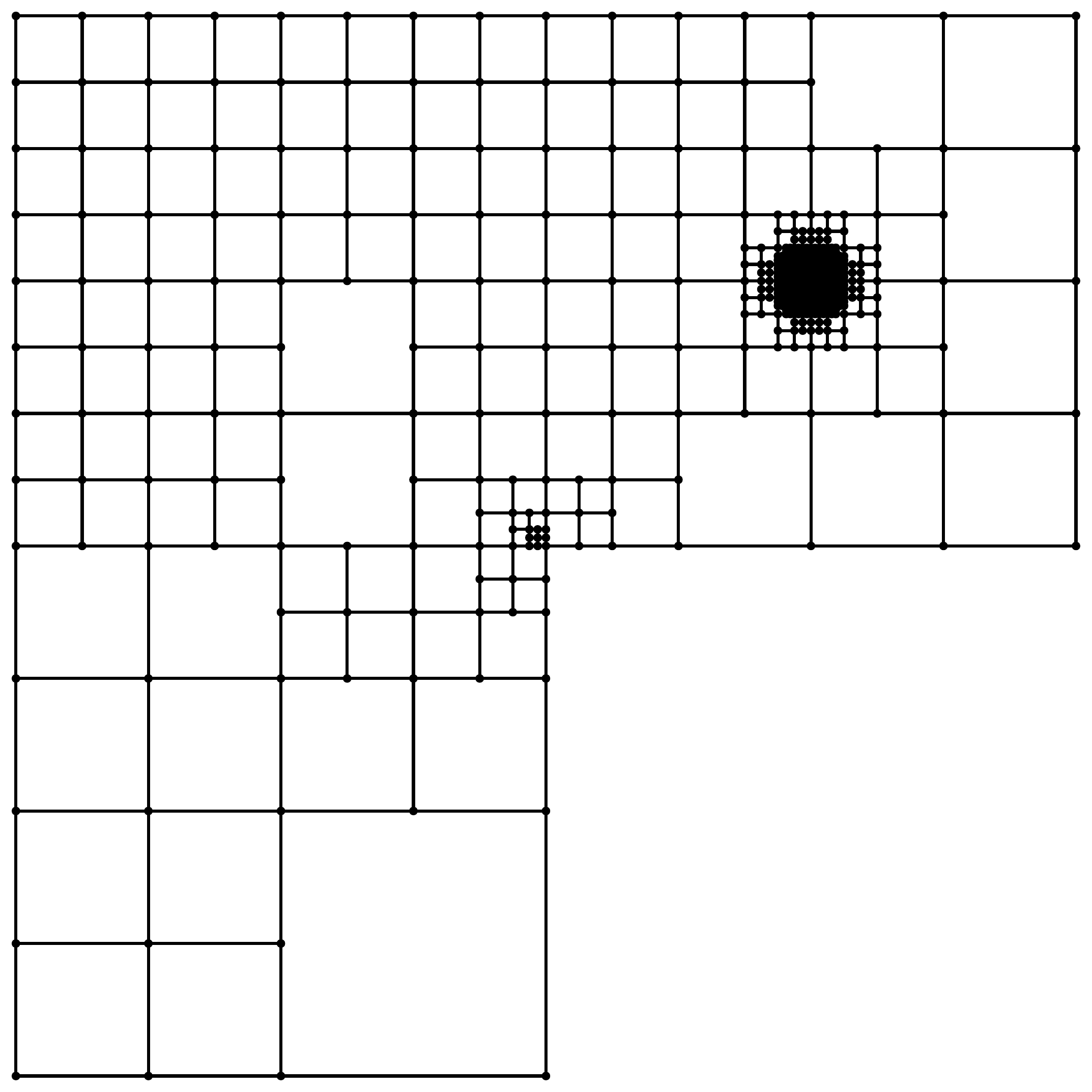}%
}
\caption{Some representative meshes from the adaptive refinement sequence for $\k = 1$ when solving Problem~\ref{enum:num:LDomainGaussianProblem} with solution~\eqref{eq:num:LDomainGaussian} together with the adaptive approximation on the final mesh.}
\label{fig:num:LDomainGaussianMeshes}
\end{figure}

\begin{figure}
\subcaptionbox{The error $\|\nabla(u - \Pi^0_k u_h ) \|_{0,\Omega}$.\label{fig:num:InternalLayer:HOneError}}[.3\textwidth]{%
\begin{tikzpicture}[scale=0.5]
	\begin{loglogaxis}[xlabel={Number of degrees of freedom},grid=major,legend entries={$\k = 1$, $\k = 2$, $\k = 3$, dofs$^{-0.5}$, dofs$^{-1.0}$, dofs$^{-1.5}$, },]
		\addplot[mark=triangle] table[x index = 1, y index = 2] from \InternalLayerHOneErrorPlotData;
		\addplot[mark=star] table[x index = 5, y index = 6] from \InternalLayerHOneErrorPlotData;
		\addplot[mark=o] table[x index = 9, y index = 10] from \InternalLayerHOneErrorPlotData;
		\addplot[dashed] coordinates {(280, 3.6894471073148543) (3479.0, 1.0466779855555364)};
		\addplot[dashed] coordinates {(801, 5.193354226522696) (4939.0, 0.8422508069335257)};
		\addplot[dashed] coordinates {(1443, 12.39881122269069) (6729.0, 1.2312712869593867)};
\end{loglogaxis}
\end{tikzpicture}}%
\hfill%
\subcaptionbox{Estimated error.}[.3\textwidth]{%
\begin{tikzpicture}[scale=0.5]
	\begin{loglogaxis}[xlabel={Number of degrees of freedom},grid=major,legend entries={$\k = 1$, $\k = 2$, $\k = 3$, dofs$^{-0.5}$, dofs$^{-1.0}$, dofs$^{-1.5}$}]
		\addplot[mark=triangle] table[x index = 1, y index = 2] from \InternalLayerEstimatorConvergencePlotData;
		\addplot[mark=star] table[x index = 5, y index = 6] from \InternalLayerEstimatorConvergencePlotData;
		\addplot[mark=o] table[x index = 9, y index = 10] from \InternalLayerEstimatorConvergencePlotData;
		\addplot[dashed] coordinates {(280, 25.320748448910212) (3479.0, 7.183371710823147)};
		\addplot[dashed] coordinates {(801, 27.244730811880945) (4939.0, 4.418511719035566)};
		\addplot[dashed] coordinates {(1443, 34.79074871460481) (6729.0, 3.4549158927202366)};
\end{loglogaxis}
\end{tikzpicture}}%
\hfill%
\subcaptionbox{Effectivity of the estimator.\label{fig:num:InternalLayer:efficiency}}[.3\textwidth]{%
\begin{tikzpicture}[scale=0.5]
	\begin{semilogxaxis}[xlabel={Number of degrees of freedom},grid=major,ymin=1,ymax=25,legend entries={$\k = 1$, $\k = 2$, $\k = 3$, },]
		\addplot[mark=triangle] table[x index = 1, y index = 2] from \InternalLayerEfficiencyPlotData;
		\addplot[mark=star] table[x index = 5, y index = 6] from \InternalLayerEfficiencyPlotData;
		\addplot[mark=o] table[x index = 9, y index = 10] from \InternalLayerEfficiencyPlotData;
\end{semilogxaxis}
\end{tikzpicture}}%

\subcaptionbox{The estimator components for $\k=1$.}[.5\textwidth]{%
\begin{tikzpicture}[scale=0.6]
	\begin{loglogaxis}[xlabel={Number of degrees of freedom}, grid=major, ymin=0.01, legend entries={$\left( \sum \estimator^{\E}\right)^{1/2}$, $\left( \sum \virtualosc^E\right)^{1/2}$, $\left( \sum \dataest^\E\right)^{1/2}$, $\left( \sum \stabest^\E\right)^{1/2}$},]
		\addplot[mark=triangle] table[x index = 1, y index = 2] from \InternalLayerEstimatorComponentsPlotData;
		\addplot[mark=star] table[x index = 5, y index = 6] from \InternalLayerEstimatorComponentsPlotData;
		\addplot[mark=o] table[x index = 9, y index = 10] from \InternalLayerEstimatorComponentsPlotData;
		\addplot[mark=10-pointed star] table[x index = 13, y index = 14] from \InternalLayerEstimatorComponentsPlotData;
\end{loglogaxis}
\end{tikzpicture}}%
\hfill%
\subcaptionbox{Adaptive approximation after 40 refinement steps.\label{fig:num:InternalLayer:solution}}[.5\textwidth]{%
  \includegraphics[width=.2\textwidth]{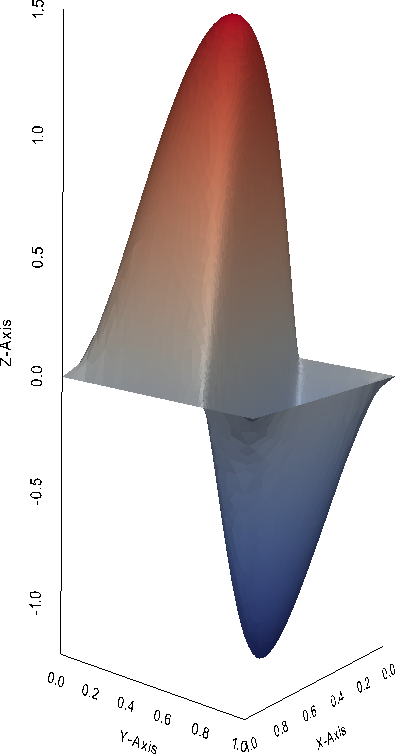}
}
\hfill

\caption{The behaviour of the error and estimator when applied to Problem~\ref{enum:num:InternalLayerProblem} with solution~\eqref{eq:num:InternalLayer} under adaptive refinement.}
\label{fig:num:InternalLayer:errorPlots}
\end{figure}

\begin{figure}
\subcaptionbox{The initial mesh.\label{fig:num:InternalLayer:meshOriginal}}[.3\textwidth]{%
\includegraphics[width=.3\textwidth]{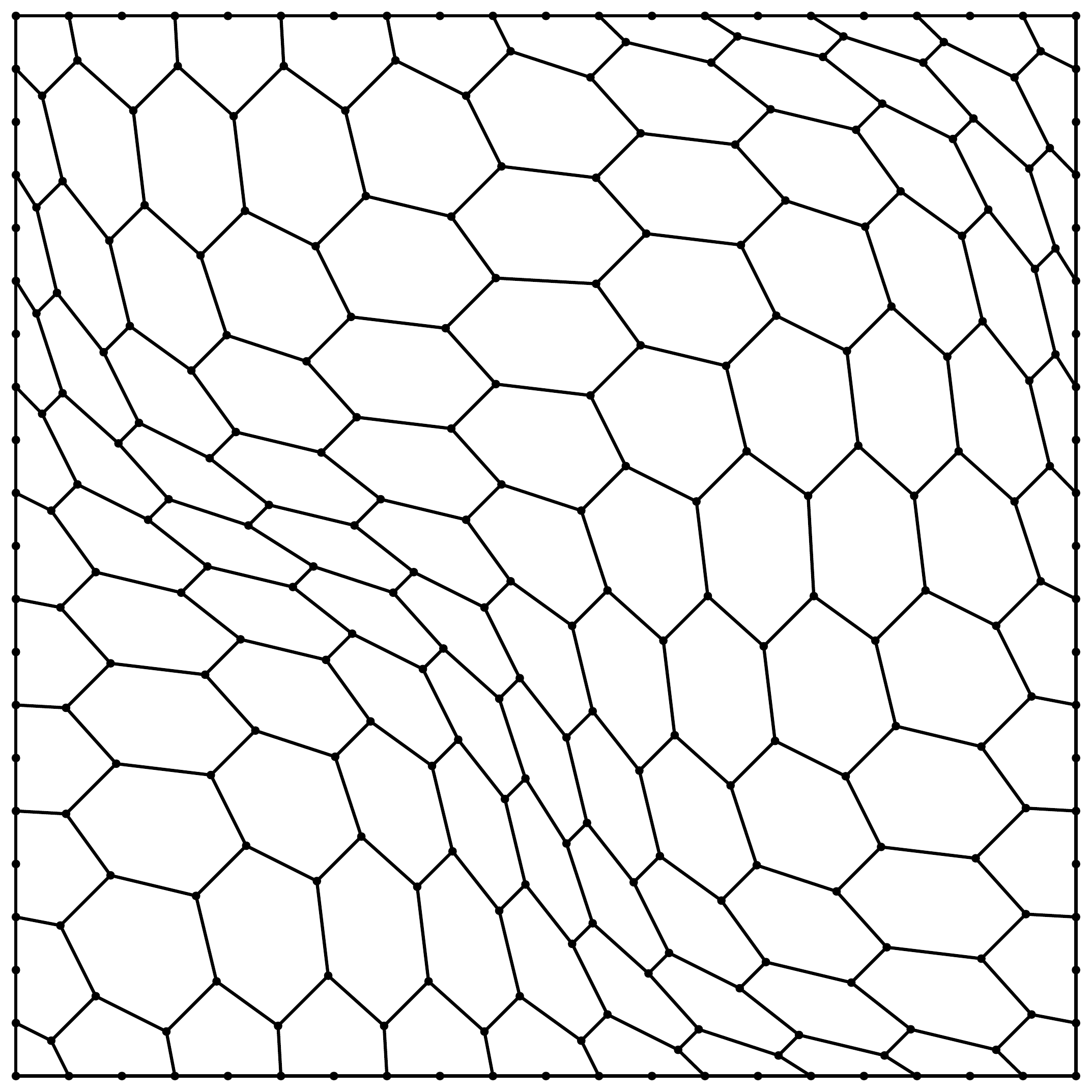}%
}\hfill%
\subcaptionbox{After 25 refinement steps.}[.3\textwidth]{%
\includegraphics[width=.3\textwidth]{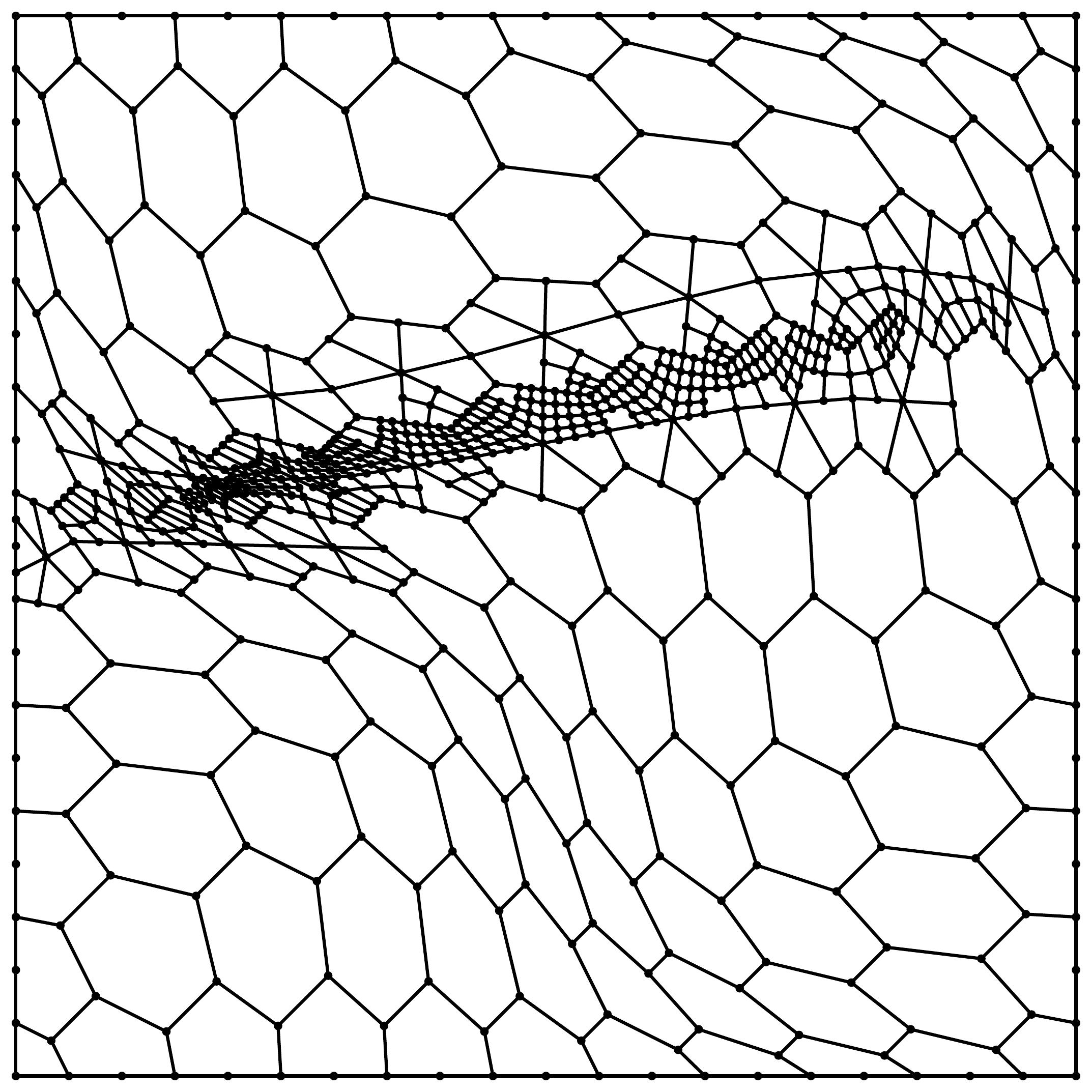}%
}\hfill%
\subcaptionbox{After 40 refinement steps.}[.3\textwidth]{%
\includegraphics[width=.3\textwidth]{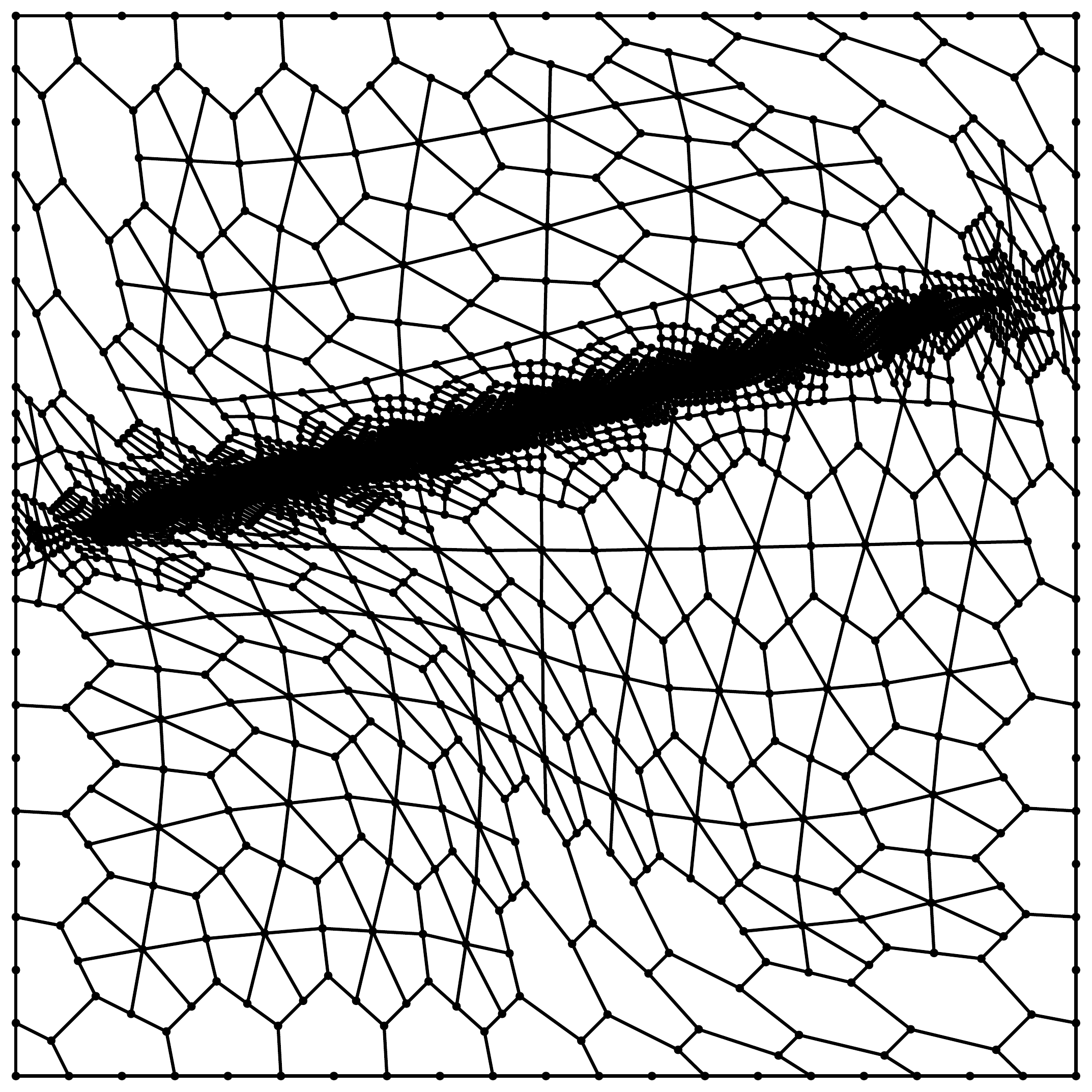}%
}
\caption{Some representative meshes from the adaptive refinement sequence for $\k = 1$ when solving Problem~\ref{enum:num:InternalLayerProblem} with solution~\eqref{eq:num:InternalLayer}.}
\label{fig:num:InternalLayer:meshes}
\end{figure}

\begin{figure}
\subcaptionbox{The error $\|\nabla(u - \Pi^0_k u_h ) \|_{0,\Omega}$.\label{fig:num:AlignedLimitedUnlimited:HOneError}}[.3\textwidth]{%
\begin{tikzpicture}[scale=0.45]
	\begin{loglogaxis}[xlabel={Number of degrees of freedom},grid=major,legend entries={limited, unlimited, dofs$^{-0.5}$ },]
		\addplot[mark=triangle] table[x index = 1, y index = 2] from \AlignedLimitedUnlimitedHOneErrorPlotData;
		\addplot[mark=star] table[x index = 5, y index = 6] from \AlignedLimitedUnlimitedHOneErrorPlotData;
		\addplot[dashed] coordinates {(36, 0.3606099826683983) (1338.0, 0.059150843031971184)};
\end{loglogaxis}
\end{tikzpicture}}%
\hfill%
\subcaptionbox{Estimated error.}[.3\textwidth]{%
\begin{tikzpicture}[scale=0.45]
	\begin{loglogaxis}[xlabel={Number of degrees of freedom},grid=major,legend entries={limited, unlimited, dofs$^{-0.5}$}]
		\addplot[mark=triangle] table[x index = 1, y index = 2] from \AlignedLimitedUnlimitedEstimatorConvergencePlotData;
		\addplot[mark=star] table[x index = 5, y index = 6] from \AlignedLimitedUnlimitedEstimatorConvergencePlotData;
		\addplot[dashed] coordinates {(36, 3.327349055826757) (1338.0, 0.545784950980598)};
\end{loglogaxis}
\end{tikzpicture}}%
\hfill%
\subcaptionbox{Effectivity of the estimator. \label{fig:num:AlignedLimitedUnlimited:efficiency}}[.3\textwidth]{%
\begin{tikzpicture}[scale=0.45]
	\begin{semilogxaxis}[xlabel={Number of degrees of freedom},grid=major,ymin=3.981071706,ymax=20,legend entries={limited, unlimited, },]
		\addplot[mark=triangle] table[x index = 1, y index = 2] from \AlignedLimitedUnlimitedEfficiencyPlotData;
		\addplot[mark=star] table[x index = 5, y index = 6] from \AlignedLimitedUnlimitedEfficiencyPlotData;
	\end{semilogxaxis}
\end{tikzpicture}}

\subcaptionbox{The non-zero components of the estimator. \label{fig:num:AlignedLimitedUnlimited:estimatorComponents}}[.5\textwidth]{%
\begin{tikzpicture}[scale=0.6]
	\begin{loglogaxis}[xlabel={Number of degrees of freedom},grid=major,ymax=20,legend entries={$\left( \sum \norm{\edgeresid}_{0,\s}^2 \right)^{1/2}$ (limited), $\left( \sum \norm{\edgeresid}_{0,\s}^2 \right)^{1/2}$ (unlimited), $\left( \sum \stabest^\E \right)^{1/2}$ (limited), $\left( \sum \stabest^\E \right)^{1/2}$ (unlimited) },]
		\addplot[mark=star] table[x index = 1, y index = 2] from \AlignedLimitedUnlimitedEstimatorComponentsPlotData;
		\addplot[mark=o] table[x index = 17, y index = 18] from \AlignedLimitedUnlimitedEstimatorComponentsPlotData;
		\addplot[mark=diamond] table[x index = 13, y index = 14] from \AlignedLimitedUnlimitedEstimatorComponentsPlotData;
		\addplot[mark=+] table[x index = 29, y index = 30] from \AlignedLimitedUnlimitedEstimatorComponentsPlotData;
\end{loglogaxis}
\end{tikzpicture}}%
\hfill%
\subcaptionbox{Adaptive approximation.}[.5\textwidth]{%
	\includegraphics[width=.5\textwidth]{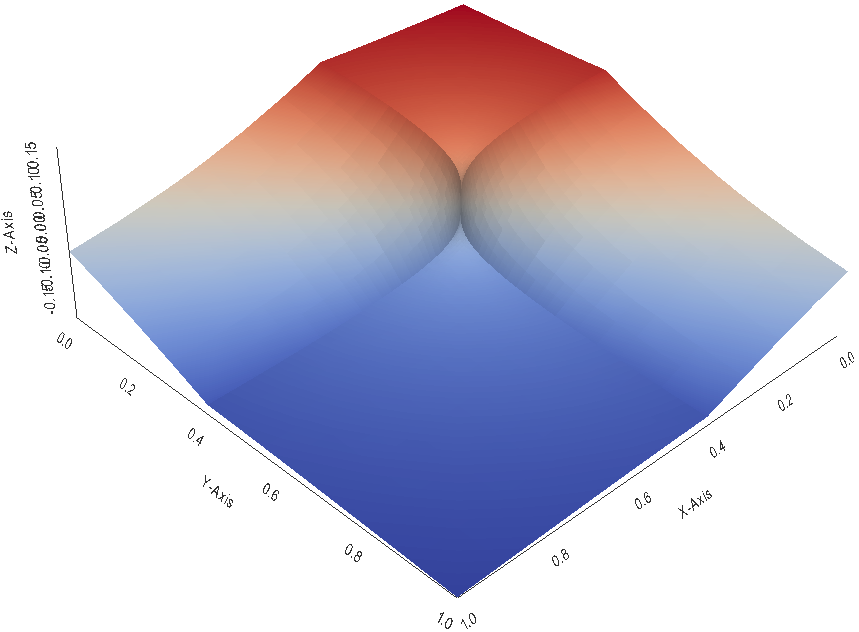}
}
\caption{The behaviour of the estimator for the Kellogg problem (\S\ref{sec:num:kellogg}) when the discontinuities of $\diff$ are matched by the initial mesh, with $\k=1$ and either one (`limited') or an unlimited number of hanging nodes per edge (`unlimited').}
\label{fig:num:AlignedLimitedUnlimited:errorPlots}
\end{figure}

\begin{figure}
\subcaptionbox{The initial mesh.\label{fig:num:AlignedLimitedUnlimited:initialMesh}}[.3\textwidth]{%
	\includegraphics[width=.3\textwidth]{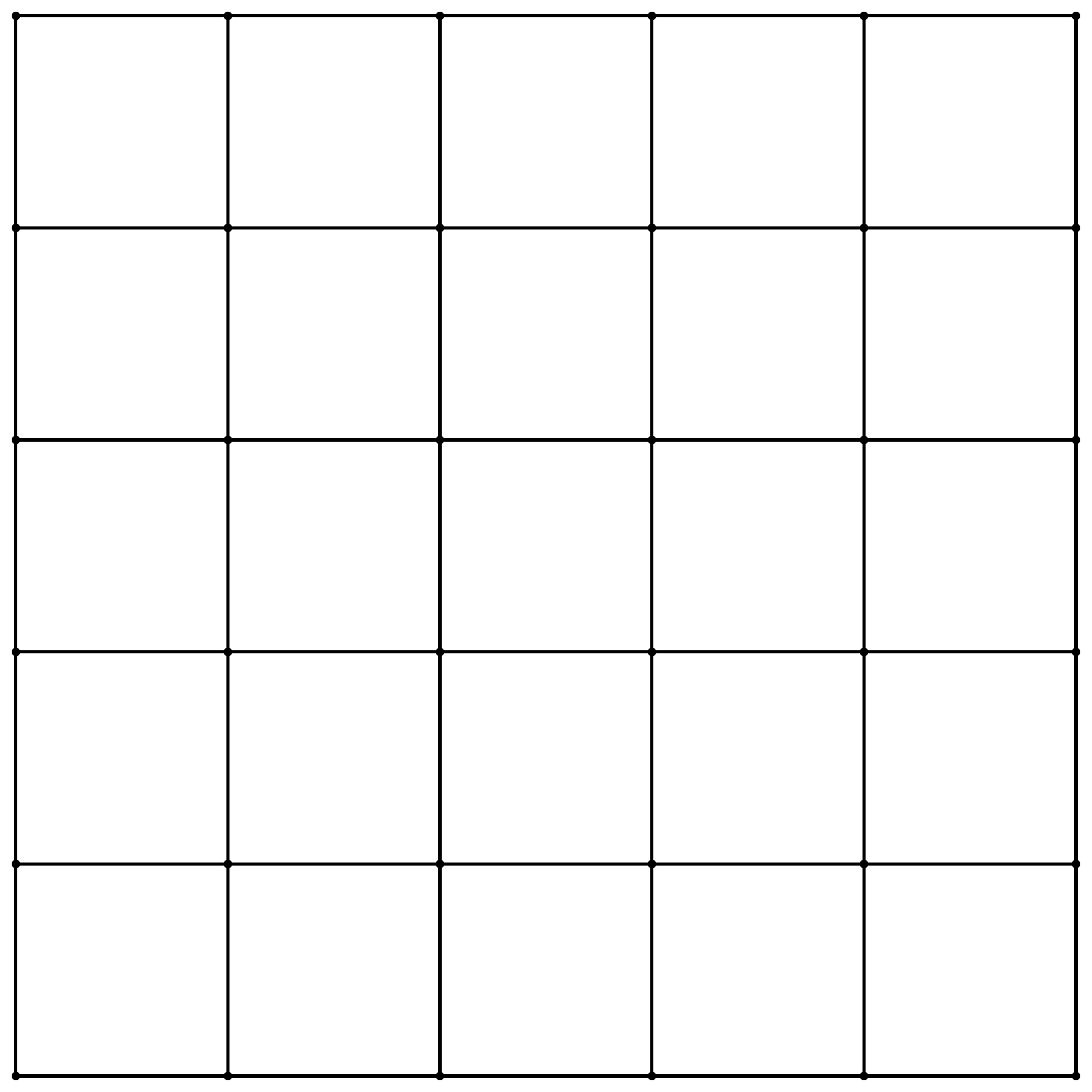}%
}\hfill%
\subcaptionbox{The final mesh with at most one hanging node per edge. \label{fig:num:AlignedLimitedUnlimited:finalLimitedMesh}}[.3\textwidth]{%
\includegraphics[width=.3\textwidth]{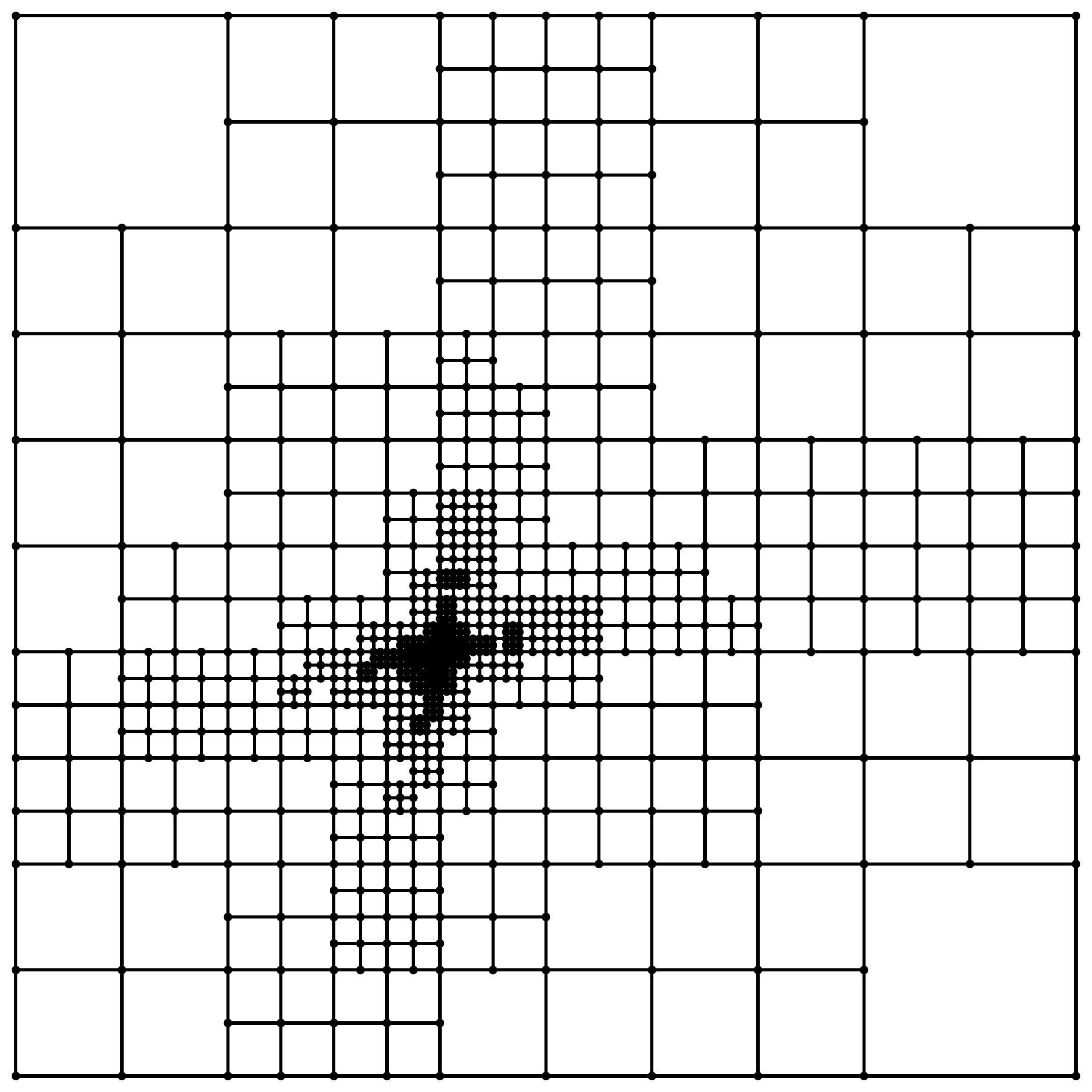}%
}\hfill%
\subcaptionbox{The final mesh with no limit on the number of hanging nodes. \label{fig:num:AlignedLimitedUnlimited:finalUnlimitedMesh}}[.3\textwidth]{%
\includegraphics[width=.3\textwidth]{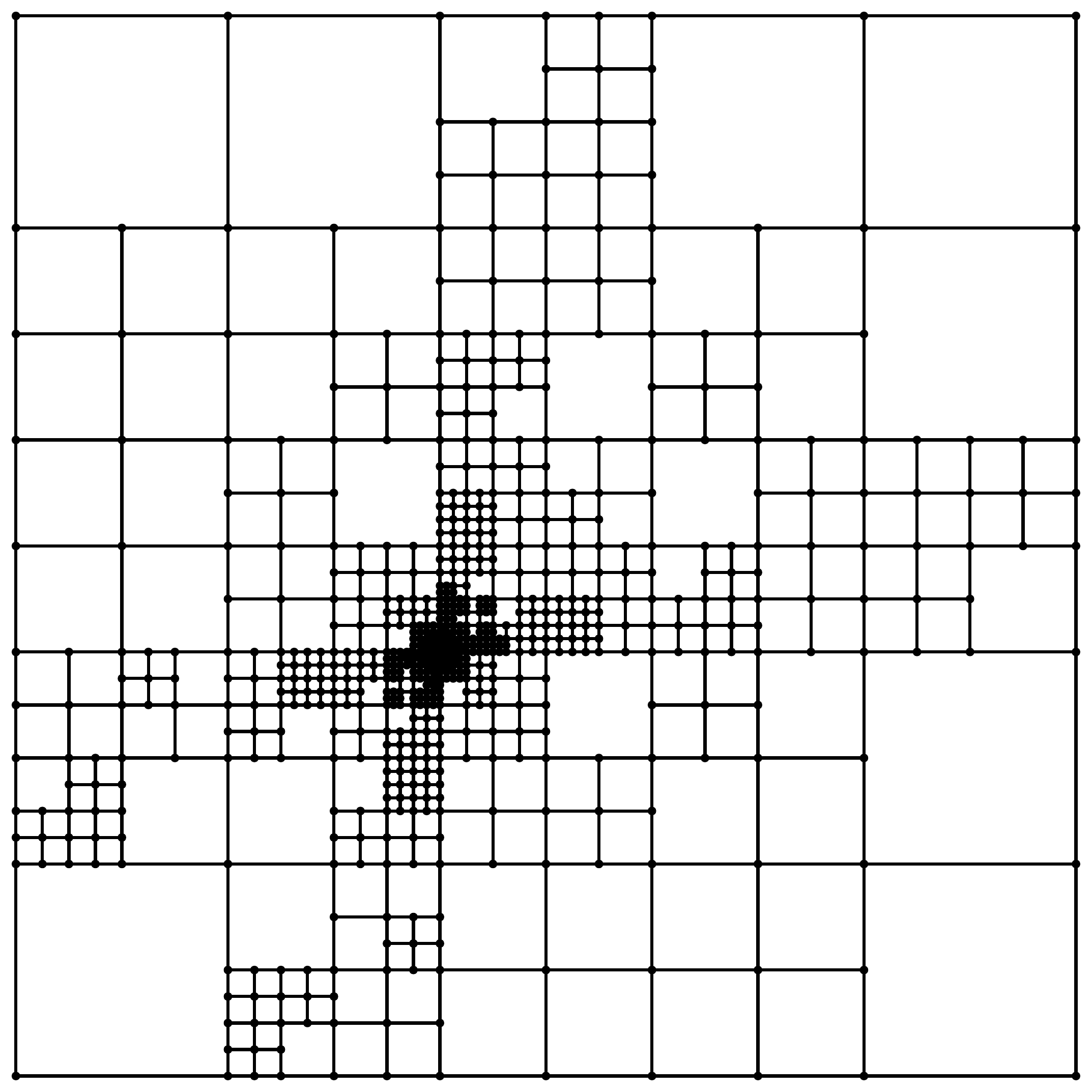}%
}
\caption{The initial and final adapted meshes when solving the Kellogg problem (see \S\ref{sec:num:kellogg}) when the discontinuity in $\diff$ is aligned with the initial mesh.}
\label{fig:num:AlignedLimitedUnlimited:meshes}
\end{figure}

\begin{figure}
\subcaptionbox{The error $\|\nabla(u - \Pi^0_k u_h ) \|_{0,\Omega}$.\label{fig:num:UnalignedUnlimited:HOneError}}[.5\textwidth]{%
\begin{tikzpicture}[scale=0.67]
	\begin{loglogaxis}[xlabel={Number of degrees of freedom},ylabel={$\|\nabla(u - \Pi^0_k u_h ) \|_{0,\Omega}$},grid=major,legend entries={squares, Voronoi, random quadrilaterals, dofs$^{-0.5}$ },]
		\addplot[mark=triangle] table[x index = 1, y index = 2] from \UnlimitedUnalignedHOneErrorPlotData;
		\addplot[mark=star] table[x index = 5, y index = 6] from \UnlimitedUnalignedHOneErrorPlotData;
		\addplot[mark=o] table[x index = 9, y index = 10] from \UnlimitedUnalignedHOneErrorPlotData;
		\addplot[dashed] coordinates {(36, 0.9218806977217697) (1142.0, 0.16367898387859667)};
\end{loglogaxis}
\end{tikzpicture}}%
\hfill%
\subcaptionbox{Convergence of the estimator. \label{fig:num:UnalignedUnlimited:estimator}}[.5\textwidth]{%
\begin{tikzpicture}[scale=0.7]
	\begin{loglogaxis}[xlabel={Number of degrees of freedom},ylabel={the estimator of $\|\nabla(u - u_h ) \|_{0,\Omega}$},grid=major,legend entries={squares, Voronoi, random quadrilaterals, dofs$^{-0.5}$ }]
		\addplot[mark=triangle] table[x index = 1, y index = 2] from \UnlimitedUnalignedEstimatorConvergencePlotData;
		\addplot[mark=star] table[x index = 5, y index = 6] from \UnlimitedUnalignedEstimatorConvergencePlotData;
		\addplot[mark=o] table[x index = 9, y index = 10] from \UnlimitedUnalignedEstimatorConvergencePlotData;
		\addplot[dashed] coordinates {(36, 8.132170558859686) (1142.0, 1.443858643630391)};
\end{loglogaxis}
\end{tikzpicture}}
\vspace{0.1cm}

\subcaptionbox{Effectivity of the estimator. \label{fig:num:UnalignedUnlimited:efficiency}}[.5\textwidth]{%
\begin{tikzpicture}[scale=0.67]
	\begin{semilogxaxis}[xlabel={Number of degrees of freedom},ylabel={Effectivity},grid=major,ymin=3.981071706,ymax=30,legend entries={squares, Voronoi, random quadrilaterals},]
		\addplot[mark=triangle] table[x index = 1, y index = 2] from \UnlimitedUnalignedEfficiencyPlotData;
		\addplot[mark=star] table[x index = 5, y index = 6] from \UnlimitedUnalignedEfficiencyPlotData;
		\addplot[mark=o] table[x index = 9, y index = 10] from \UnlimitedUnalignedEfficiencyPlotData;
	\end{semilogxaxis}
\end{tikzpicture}}%
\hfill%
\subcaptionbox{The non-zero components of the estimator on the square mesh. \label{fig:num:UnalignedUnlimited:estimatorComponents}}[.5\textwidth]{%
\begin{tikzpicture}[scale=0.67]
	\begin{loglogaxis}[xlabel={Number of degrees of freedom},ylabel={Estimator components},grid=major,ymax=20,legend entries={$\left( \sum \norm{\edgeresid}_{0,\s}^2 \right)^{1/2}$, $\left( \sum \virtualosc^\E \right)^{1/2}$ , $\left( \sum \dataest^E \right)^{1/2}$ , $\left( \sum \stabest^\E \right)^{1/2}$},]
		\addplot[mark=triangle] table[x index = 1, y index = 2] from \UnlimitedUnalignedEstimatorComponentsPlotData;
		\addplot[mark=star] table[x index = 5, y index = 6] from \UnlimitedUnalignedEstimatorComponentsPlotData;
		\addplot[mark=o] table[x index = 9, y index = 10] from \UnlimitedUnalignedEstimatorComponentsPlotData;
		\addplot[mark=10-pointed star] table[x index = 13, y index = 14] from \UnlimitedUnalignedEstimatorComponentsPlotData;
\end{loglogaxis}
\end{tikzpicture}}
\caption{The behaviour of the estimator for the Kellogg problem (see \S\ref{sec:num:kellogg}) when the discontinuities in $\diff$ cannot align with any mesh in the sequence, using $\k=1$ and three different types of mesh, depicted in Figure~\ref{fig:num:Unaligned:mesh}.}
\label{fig:num:UnalignedUnlimited:errorPlots}
\end{figure}

\begin{figure}
	\subcaptionbox{The initial square mesh \label{fig:num:Unaligned:mesh:initialSquares}}[.3\textwidth]{%
		\includegraphics[width=.3\textwidth]{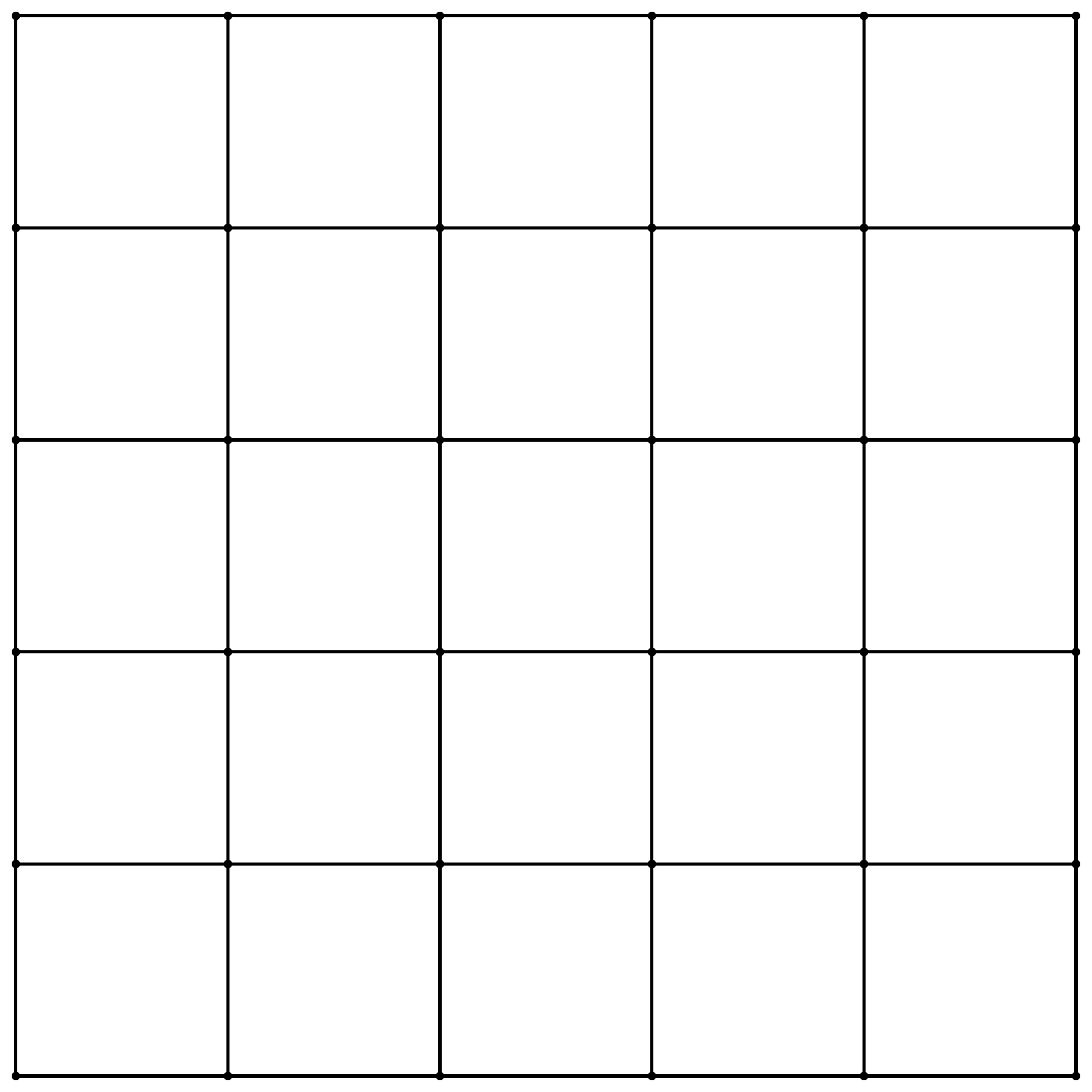}%
	}\hfill%
	\subcaptionbox{The initial Voronoi mesh \label{fig:num:Unaligned:mesh:initialVoronoi}}[.3\textwidth]{%
		\includegraphics[width=.3\textwidth]{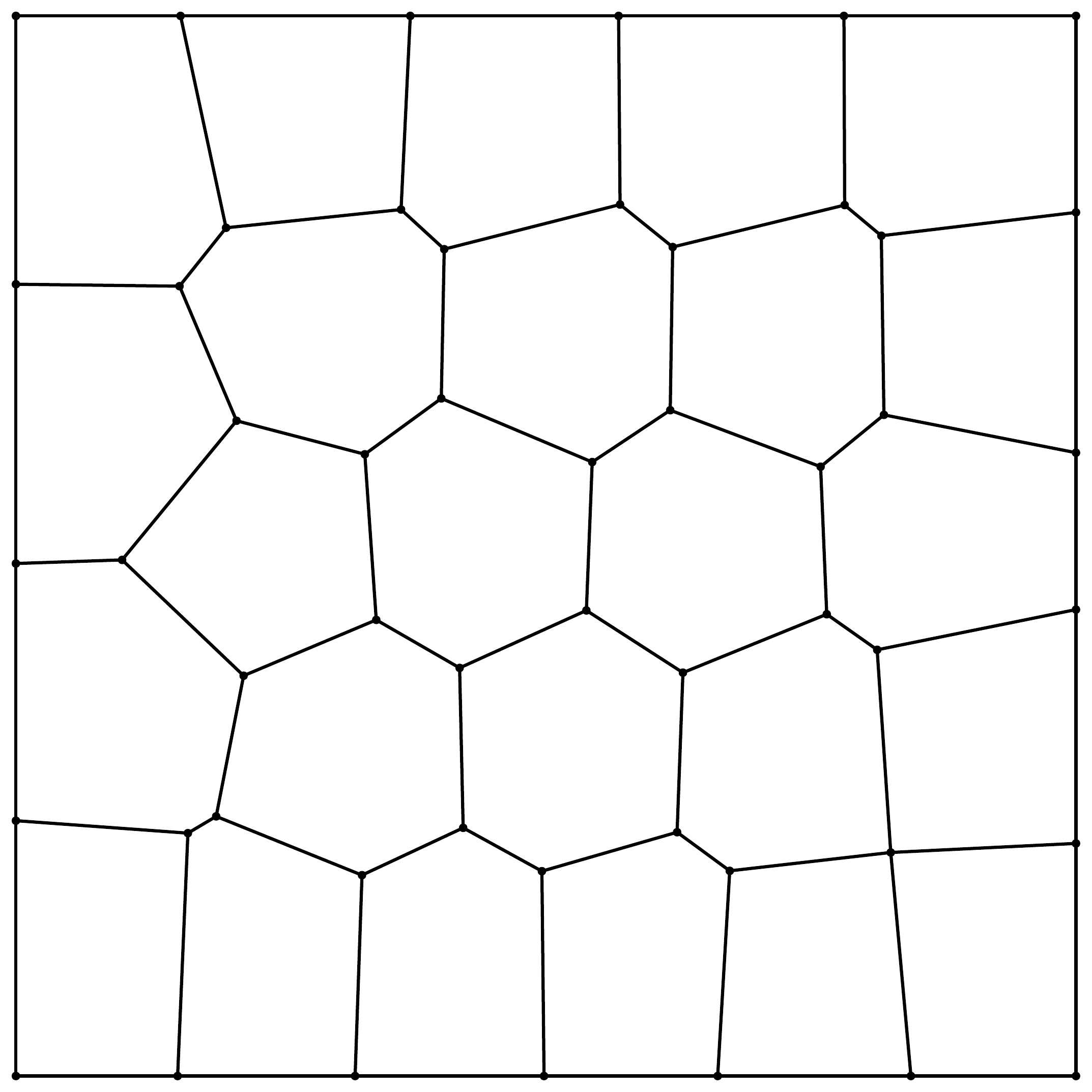}%
	}\hfill%
	\subcaptionbox{The initial randomised quadrilateral mesh \label{fig:num:Unaligned:mesh:initialRandomQuads}}[.3\textwidth]{%
		\includegraphics[width=.3\textwidth]{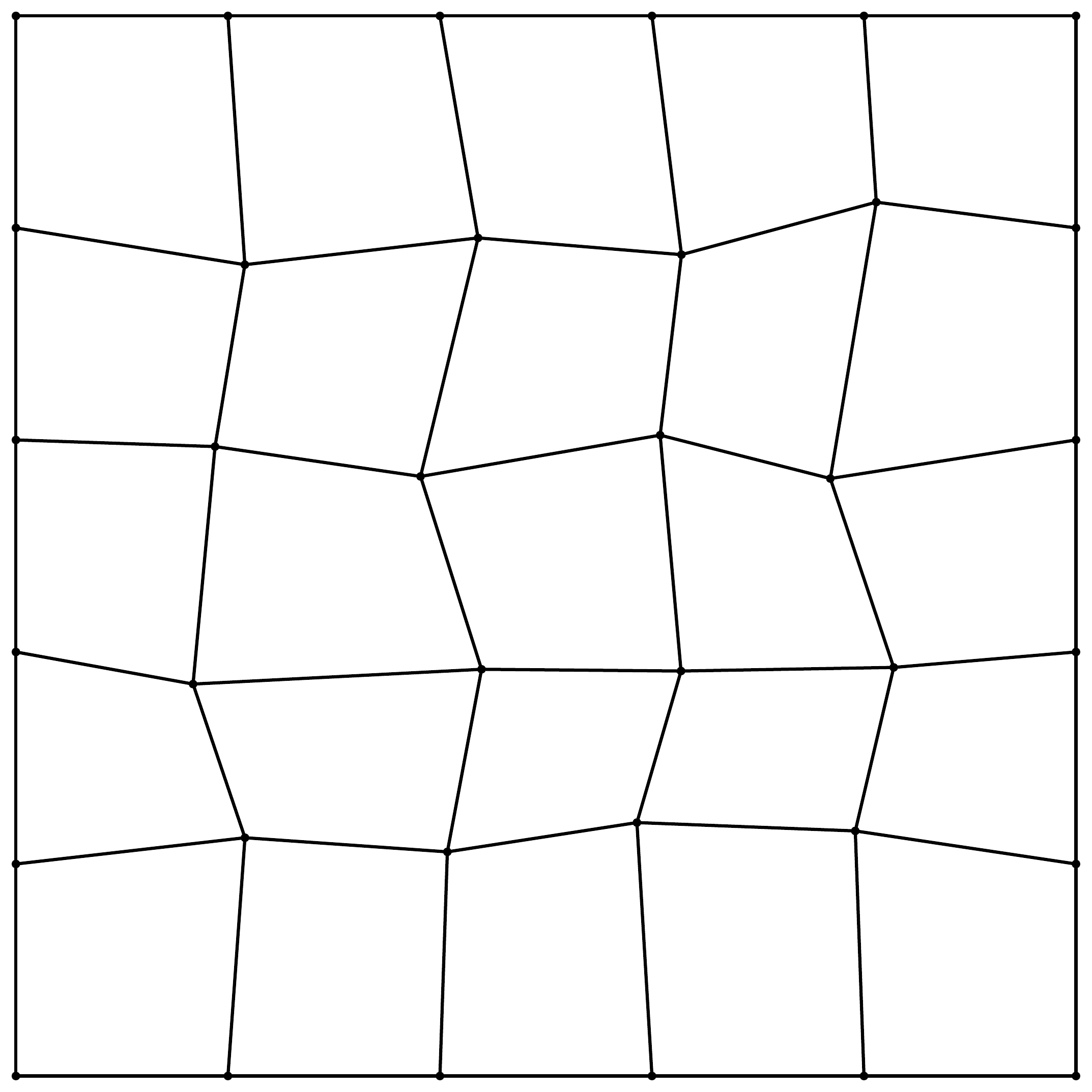}%
		}

	\subcaptionbox{The final square mesh with no limit on the number of hanging nodes \label{fig:num:Unaligned:unlimited:mesh:finalSquares}}[.3\textwidth]{%
		\includegraphics[width=.3\textwidth]{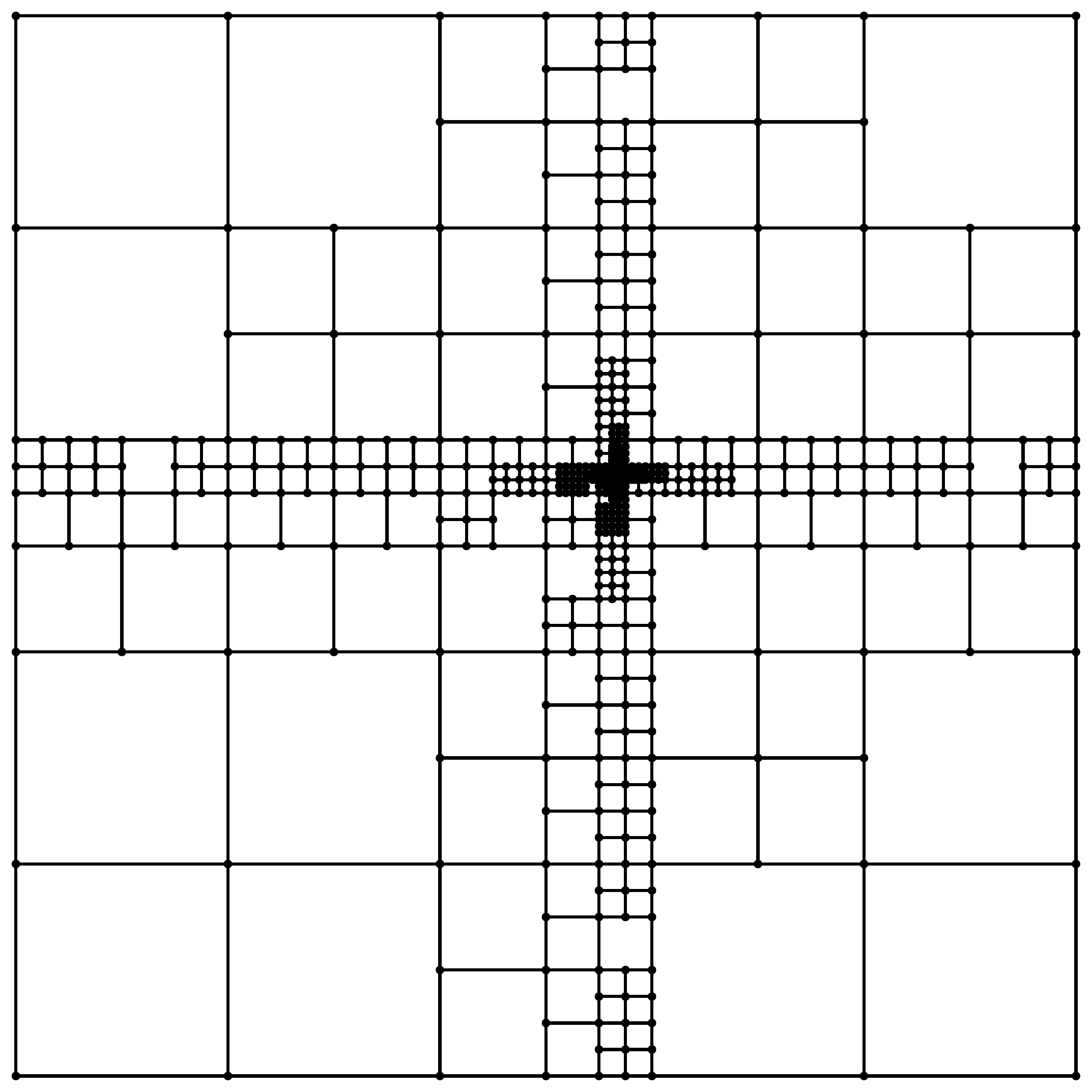}%
	}\hfill%
	\subcaptionbox{The final Voronoi mesh with no limit on the number of hanging nodes \label{fig:num:Unaligned:unlimited:mesh:finalVoronoi}}[.3\textwidth]{%
		\includegraphics[width=.3\textwidth]{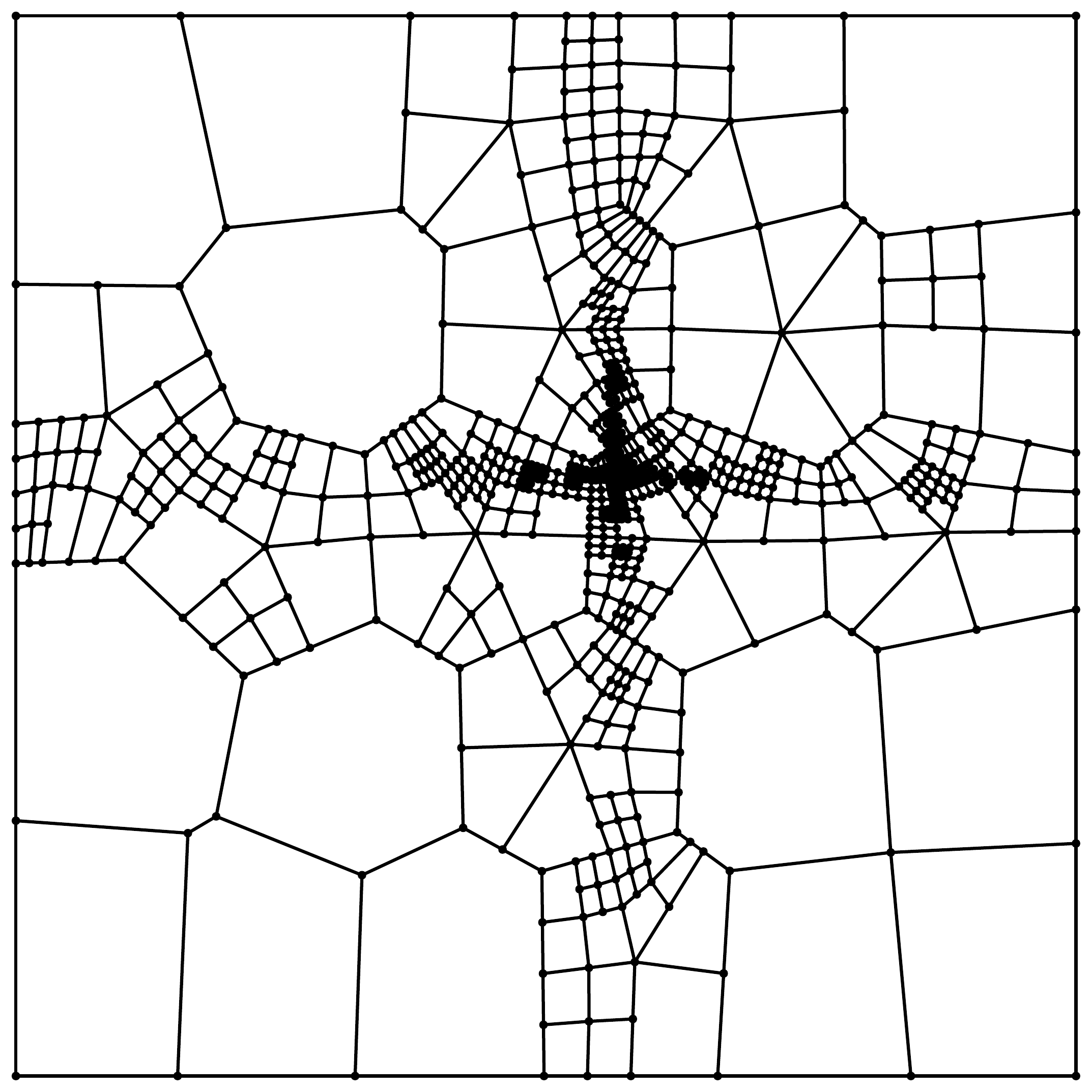}%
	}\hfill%
	\subcaptionbox{The final randomised quadrilateral mesh with no limit on the number of hanging nodes \label{fig:num:Unaligned:unlimited:mesh:finalRandomQuads}}[.3\textwidth]{%
		\includegraphics[width=.3\textwidth]{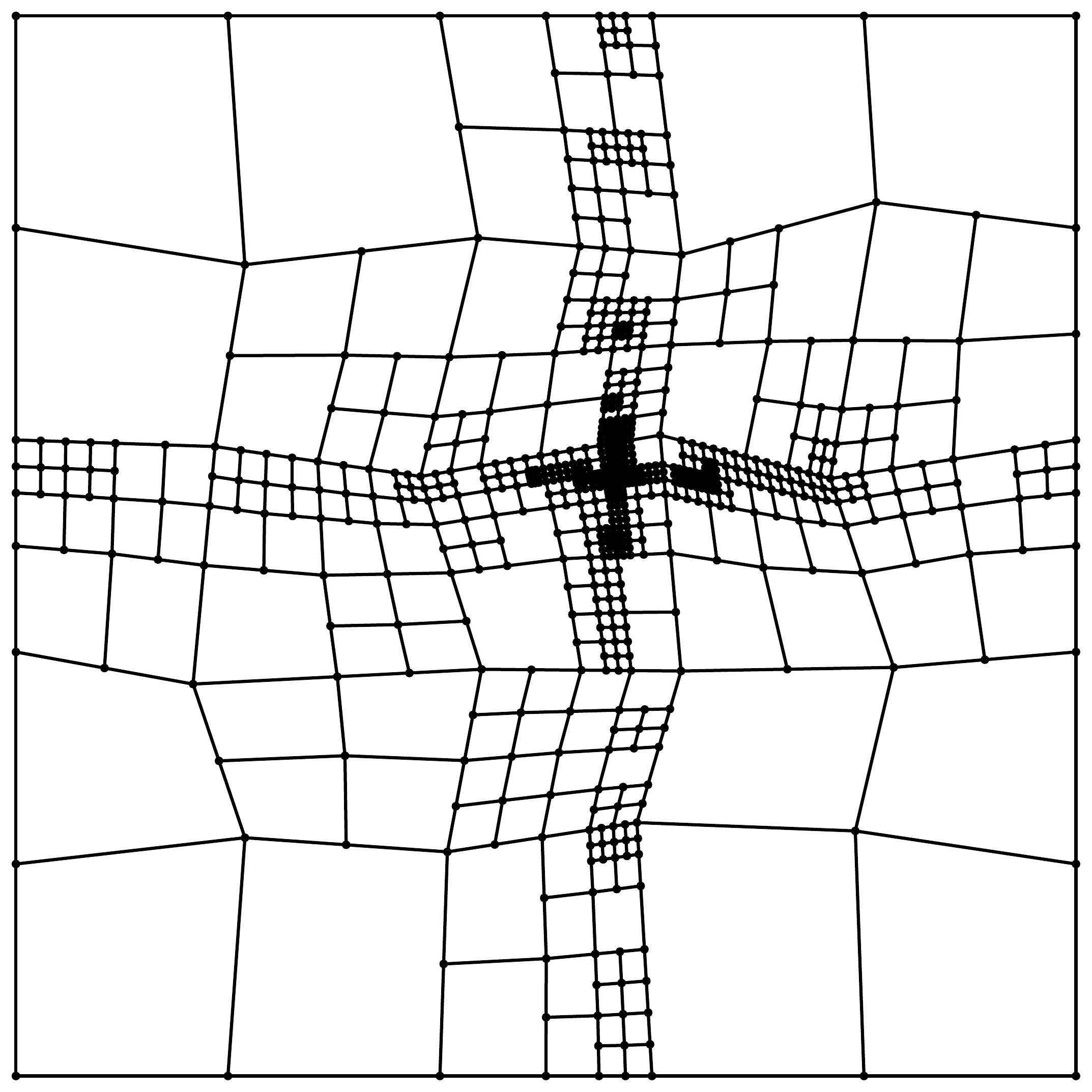}%
		}

	\subcaptionbox{The final square mesh with one hanging node per edge \label{fig:num:Unaligned:limited:mesh:finalSquares}}[.3\textwidth]{%
		\includegraphics[width=.3\textwidth]{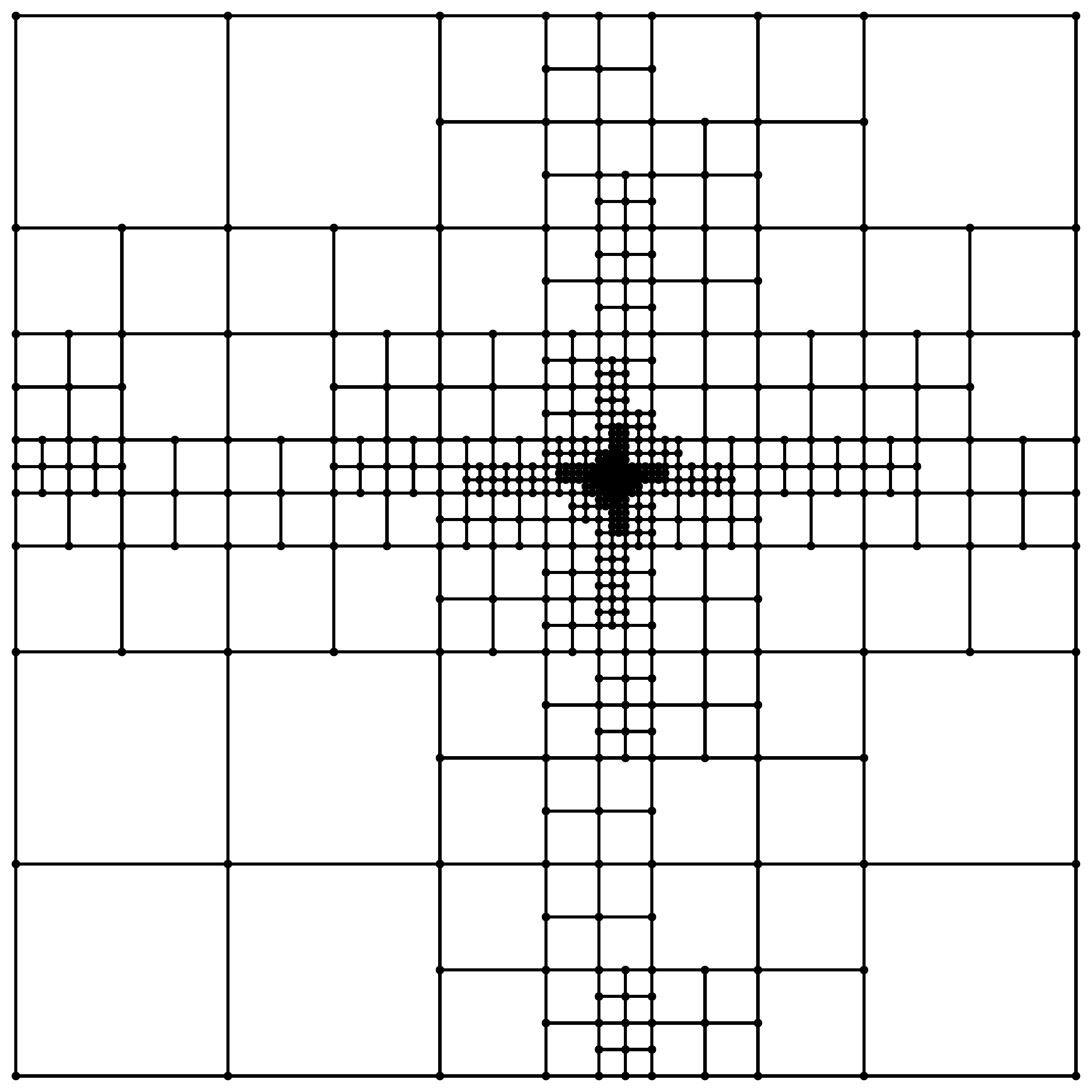}%
	}\hfill%
	\subcaptionbox{The final Voronoi mesh with one hanging node per edge \label{fig:num:Unaligned:limited:mesh:finalVoronoi}}[.3\textwidth]{%
		\includegraphics[width=.3\textwidth]{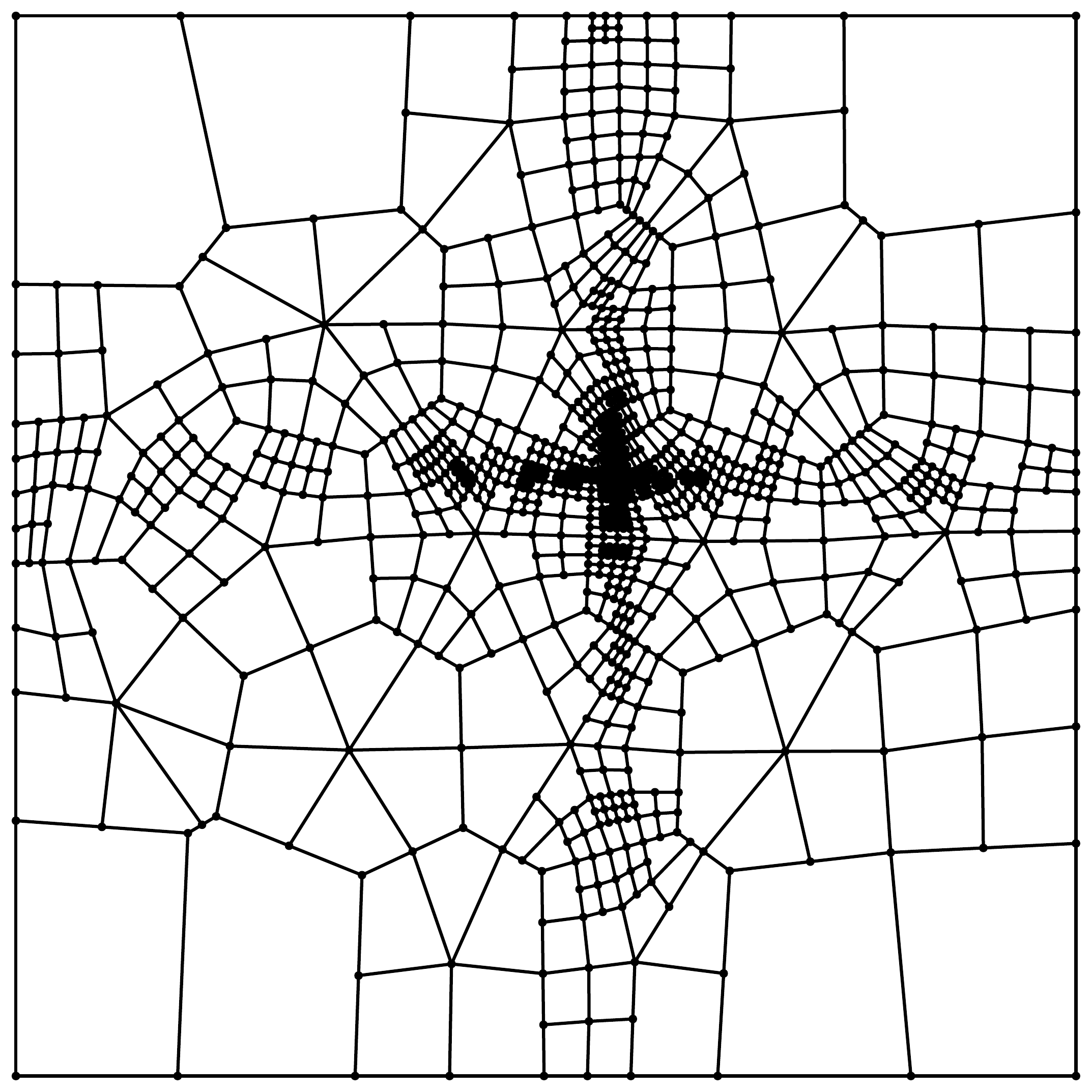}%
	}\hfill%
	\subcaptionbox{The final randomised quadrilateral mesh with one hanging node per edge \label{fig:num:Unaligned:limited:mesh:finalRandomQuads}}[.3\textwidth]{%
		\includegraphics[width=.3\textwidth]{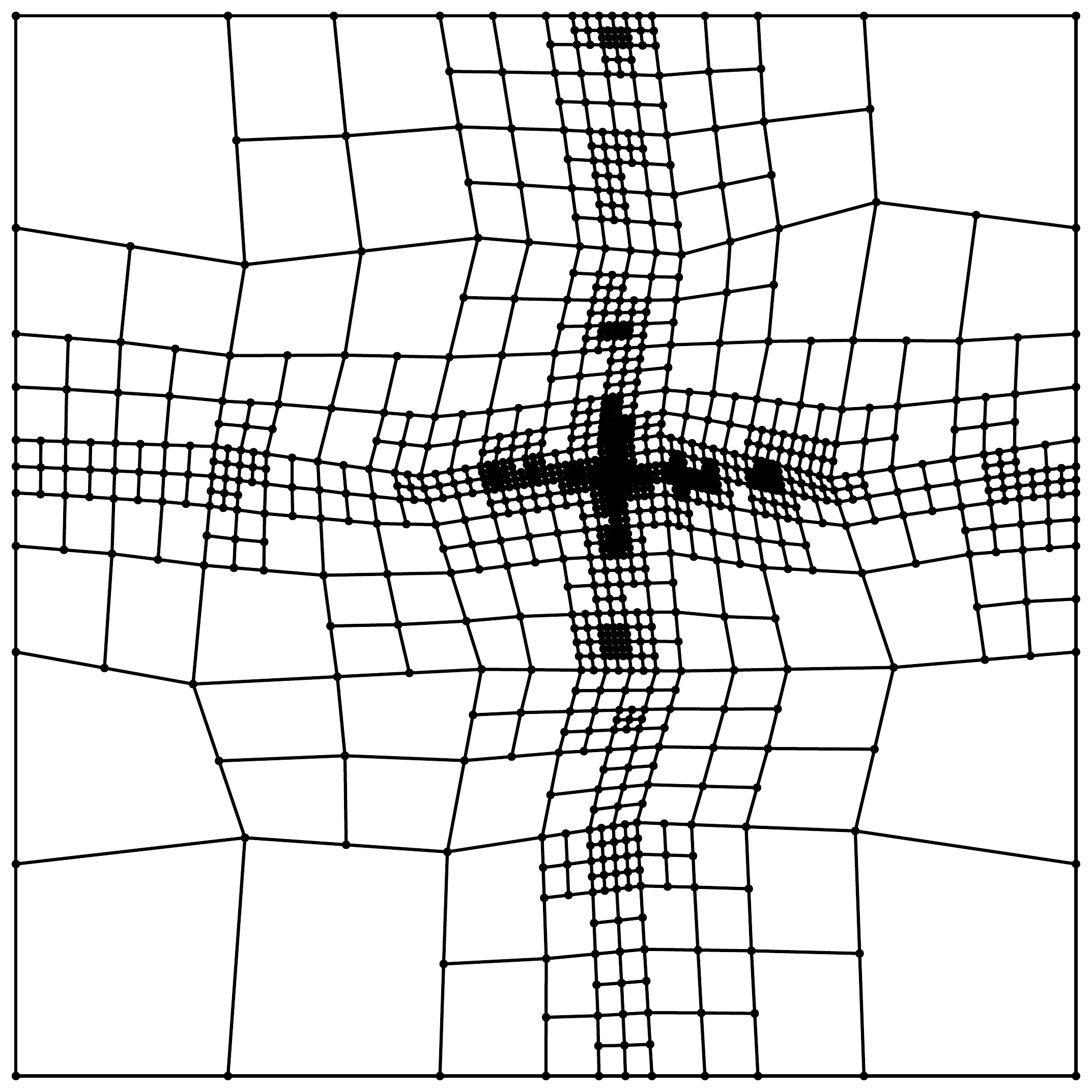}%
		}
\caption{The initial and final adapted meshes when solving the Kellogg problem (see \S\ref{sec:num:kellogg}) when the discontinuity in $\diff$ cannot align with any mesh in the sequence.}
\label{fig:num:Unaligned:mesh}
\end{figure}

\section{Conclusions and extensions}
\label{sec:conclusion}
We have derived and analysed a residual a posteriori error estimate for the $C^0$-conforming virtual element method of~\cite{UnifiedVEM} applied to general second order elliptic problems with nonconstant coefficients.
This analysis has given rise to a fully computable a posteriori error estimator which we have shown to be equivalent to the error between the true solution and the virtual element approximation, measured in the energy norm.
The analysis rests crucially on a new  Cl\'ement-type interpolation result.
We have also presented an extensive set of numerical results to demonstrate the behaviour of this estimator when used to drive an adaptive algorithm on a variety of problems using several families of meshes, consisting of general polygonal elements.

We stress that the analysis above can also be applied to other related virtual element formulations of the same problem subject to only minor modifications.
For instance, the same a posteriori analysis can be applied to the corresponding VEM obtained by discretising problem~\eqref{eq:origVariationalForm} directly, without splitting the differential operator into its symmetric and skew-symmetric parts.
The resulting local discrete bilinear form would take the form
\begin{align*}
	\AhE(\uh, \vh) := &(\diff \Po{\k-1} \nabla \uh, \Po{\k-1} \nabla \vh)_{\E} + (\conv \Po{\k-1} \nabla \uh, \Po{\k} \vh)_{\E} + \\ &+ (\reac \Po{\k} \uh, \Po{\k} \vh)_{\E} + \StaE((\Id - \Po{\k}) \uh, (\Id - \Po{\k}) \vh),
\end{align*}
(see~\cite{General} for a similar approach). The analysis would provide the same a posteriori error estimator presented in Theorem~\ref{thm:reliability}, but without the term $\virtualosc^{\E}_3$.

\section*{Acknowledgements}

AC was partially supported by the EPSRC (Grant EP/L022745/1).
OS was supported by an EPSRC Doctoral Training Grant.
All this support is gratefully acknowledged.

\bibliographystyle{acm}
\bibliography{references}

\end{document}